\documentclass{amsart}
\usepackage[utf8]{inputenc}

\usepackage{xr}
\usepackage{xcite}

\usepackage[alphabetic]{amsrefs}
\usepackage{amsmath,amssymb,amsthm,enumerate, amscd}
\usepackage{tikz-cd}
\usepackage{tikz,graphicx}
\usepackage[all,cmtip]{xy}

\makeindex 
\begin{document}

\externalcitedocument{Lit}

\pagenumbering{arabic}

 \setlength{\parskip}{7pt}

\def\LU{{\mathcal L\mathcal U}}
 \def\dach{\!\widehat{\phantom{G}}}
\def\Aut{\operatorname{Aut}}
\def\Int{\operatorname{Int}}
\def\GL{\operatorname{GL}}  
\def\QQ{{\mathbb{Q}}}
\def\Prim{\operatorname{Prim}}
\def\tr{\operatorname{tr}}
\def\Inn{\operatorname{Inn}} \def\Ad{\operatorname{Ad}}
\def\id{\operatorname{id}} \def\supp{\operatorname{supp}}
\def\csp{\overline{\operatorname{span}}}
\def\Ind{\operatorname{Ind}} \def\nd{\mathrm{es}}
\def\triv{\mathrm{triv}} 
\def\eps{\epsilon} 
\newcommand\veps{\varepsilon}
\def\K{{\mathcal K}}
\def\L{\mathcal L} \def\R{\mathcal R} \def\L{\mathcal L}
\def\C{\mathcal C} \def\E{\mathcal E} \def\Q{\mathcal Q}
\def\F{\mathcal F} \def\B{\mathcal B}
\def\H{\mathcal H}
\def\PU{\mathcal PU}
\def\M{\mathcal M}  \def\V{\mathcal V}

\def\EE{\mathbb E} \def\DD{\mathbb
    D} \def\I{{\mathcal I}} \def\U{{\mathcal U}} \def\UM{{\mathcal
      U}M} \def\ZUM{{\mathcal Z}UM} 
\newcommand{\ZM}{\mathcal{ZM}}
\def\KK{{KK}} \def\RK{\operatorname{RK}}
\def\RKK{\operatorname{RKK}} \def\ind{\operatorname{ind}}
\def\res{\operatorname{res}} \def\inf{\operatorname{inf}}
\def\ker{\operatorname{ker}} \def\infl{\operatorname{inf}}
\def\pt{\operatorname{pt}} 
\newcommand{\comp}{\operatorname{comp}}
\def\infl{\operatorname{infl}}
\def\IND{\operatorname{IND}}
\def\Bott{\operatorname{Bott}}
\def\Mor{\operatorname{Mor}}
\def\Rep{\operatorname{Rep}}
\def\Incr{\operatorname{Incr}}
\def\sp{\operatorname{span}}
\def\Hom{\operatorname{Hom}}
\def\LH{\operatorname{LH}}
\def\k{\operatorname{K}}
\def\EEG{\underline{\underline{\mathcal{E}G}}} 
\def\EG{\underline{EG}}
\def\EH{\underline{EH}}
\def\EGT{\underline{E\tilde{G}}}
\def\EGN{\underline{E G/N}} 
\renewcommand{\top}{\operatorname{top}}
\newcommand{\A}{\mathcal{A}} 
\newcommand{\D}{\mathcal{D}} 
\def\sm{\backslash} \def\Ind{\operatorname{Ind}}
\def\TT{\mathbb T} \def\ZZ{\mathbb Z} \def\CC{\mathbb C}
\def\FF{\mathbb F}
\def\RR{\mathbb R} \def\NN{\mathbb N} \def\om{\omega}
\def\ClmV{C_{\!\!\!{}_{-V}}} \def\ClV{C_{\!{}_{V}}}
\newcommand{\rk}{\rangle}
\newcommand{\lk}{\langle}
\newcommand{\bmtr}{\left(\begin{matrix}}
\newcommand{\emtr}{\end{matrix}\right)}
\newcommand*{\into}{\hookrightarrow}
\newcommand{\op}{\operatorname{op}}
\newcommand*{\Corr}{\mathfrak{Corr}}
\newcommand{\Cl}{\operatorname{\it{Cl}}}
\newcommand{\sign}{\operatorname{sign}}
\newcommand{\diag}{\operatorname{diag}}
\newcommand{\indx}{\operatorname{index}}
\newcommand{\one}{\mathrm 1}
\newcommand{\Homeo}{\operatorname{Homeo}}
\newcommand{\SL}{\operatorname{SL}}
\newcommand{\SO}{\operatorname{SO}}
\newcommand{\Ort}{\operatorname{O}}
\newcommand{\inte}{\operatorname{int}}
\newcommand{\spn}{\operatorname{span}}
\newcommand{\cspn}{\overline{\operatorname{span}}}
\newcommand{\Spec}{\operatorname{Spec}}
\newcommand{\Om}{\Omega}
\newcommand*{\onto}{\twoheadrightarrow}

\newcommand{\ddS}{\stackrel{\scriptscriptstyle{o}}{S}}
\renewcommand{\oplus}{\bigoplus}

\newtheorem{theorem}{Theorem}[section] \newtheorem{corollary}[theorem]{Corollary}
\newtheorem{lemma}[theorem]{Lemma} \newtheorem{proposition}[theorem]{Proposition}
\newtheorem{lemdef}[theorem]{Lemma and Definition}

\theoremstyle{definition} \newtheorem{definition}[theorem]{Definition}
\theoremstyle{definition} \newtheorem{notation}[theorem]{Notations}

\theoremstyle{remark} \newtheorem{remark}[theorem]{\bf Remark}
\newtheorem{example}[theorem]{\bf Example}
\newtheorem{exercise}[theorem]{\bf Exercise}

\numberwithin{equation}{section} \emergencystretch 25pt

\renewcommand{\theenumi}{\roman{enumi}}
\renewcommand{\labelenumi}{(\theenumi)}

\title[Bivariant $KK$-Theory and the Baum-Connes conjecure]{Bivariant $KK$-Theory and the Baum-Connes conjecure}
\author{Siegfried Echterhoff}
\address{Mathematisches Institut\\
 Westf\"alische Wilhelms-Universit\"at M\"un\-ster\\
 Einsteinstr.\ 62\\
 48149 M\"unster\\
 Germany}
 
\begin{abstract}
This is a survey on Kasparov's bivariant $KK$-theory in connection with the 
Baum-Connes conjecture on the $K$-theory of crossed products $A\rtimes_rG$ by actions 
of a locally compact group $G$  on a C*-algebra $A$. In particular we shall discuss
Kasparov's Dirac dual-Dirac method as well as the 
permanence properties of the conjecture and the 
``Going-Down principle'' for the left hand side of the conjecture, which often allows
to reduce $K$-theory computations for $A\rtimes_rG$
to computations for  crossed products by  compact subgroups of $G$. 
We give several applications for this principle including a discussion of a method 
developed by Cuntz,  Li and the author in \cite{CEL2} for explicit 
computations of the $K$-theory groups of crossed products for certain group actions  
on totally disconnected spaces. This provides an important tool for the computation 
of $K$-theory groups of semi-group C*-algebras.
\end{abstract}

\maketitle
\section{Introduction}\footnote{The content of this note will appear 
as Chapter 3 of the book ``$K$-theory for group C*-algebras and semigroup C*-algebras'' 
which will appear in the Oberwolfach-seminar series of the Birkhäuser publishing company.
The research for this paper has been supported by the DFG through CRC 878 Groups, Geometry \& Actions}

The extension of  $K$-theory from topological spaces to operator algebras provides the 
most powerful tool for the study of $C^*$-algebras. On one side there now exist far reaching classification results 
in which certain classes of $C^*$-algebras can be classified by their $K$-theoretic data.
This started with the early work of  Elliott \cite{Elliott-AF} on the classification of 
$AF$-algebras -- inductive limits of finite dimensional $C^*$-algebras. It went on with the 
classification of simple, separable, nuclear, purely infinite $C^*$-algebras by Kirchberg and Phillips \cites{KP, Phillips}.
In  present time, due to the work of many authors (e.g., see \cite{Wi} for a survey on the most recent developments) the
 classification program covers a very large class of nuclear algebras.
 
 On the other hand, the $K$-theory groups of group algebras $C^*(G)$ and $C_r^*(G)$ 
 serve as recipients of indices of $G$-invariant elliptic operators and the study of such 
 indices has an important  impact in modern topology and geometry.
To get a rough idea, the Baum-Connes conjecture  implies that every element in the 
$K$-theory groups of the reduced $C^*$-group algebra $C_r^*(G)$ of a locally compact group 
$G$ appears as such index of some generalised $G$-invariant elliptic operator. 
To be more precise, these generalised elliptic operators form the cycles of the $G$-equivariant 
$K$-homology (with $G$-compact supports) $K_*^G(\underline{EG})$, in which $\underline{EG}$ 
is a certain classifying space for proper actions of $G$ (often realised as a $G$-manifold)
and the index map
$$\mu_G: K_*^G(\underline{EG})\to K_*(C_r^*(G))$$
 is then a well defined group homomorphism. It is called the {\em assembly map} for $G$. 
The {\em Baum-Connes conjecture (with trivial coefficients)} asserts that
the assembly map is an isomorphism  for all $G$.

The construction of the assembly map naturally extends to crossed-products and provides a
 map
$$\mu_{(G,A)}: K_*^G(\underline{EG}, A)\to K_*(A\rtimes_rG).$$
The {\em Baum-Connes conjecture with coefficients} asserts that this more general 
assembly map should be an isomorphism as well.
Although  this general version of the conjecture is now 
known to be false in general (e.g. see \cite{HLS}), it is known to be true for a large 
class of groups, including the class of all amenable groups, and it appears to be 
an extremely useful tool for the computation of $K$-theory groups in several important 
applications. 

In this chapter we want to give a concise introduction to the Baum-Connes conjecture and 
to some of the applications which allow the explicit computation of  $K$-theory groups 
with the help of the conjecture. We start with a very short reminder of the basic properties
of $C^*$-algebra $K$-theory
 before we give an introduction of 
Kasparov's bivariant  $K$-theory functor which assigns to each pair 
of $G$-$C^*$-algebras $A,B$ a pair of abelian groups $KK_*^G(A,B)$, $*=0,1$.
Kasparov's theory is not only fundamental for the definition of the 
groups $K_*^G(\EG)$ and $K_*^G(\EG, A)$ and the 
construction of the assembly map, but it also provides the most powerful tools
for proving the Baum-Connes conjecture for certain classes of groups.
In this chapter we will restrict ourselves to Kasparov's picture of $KK$-theory
and we will not touch on other descriptions or variants 
like the Cuntz picture of $KK$-theory 
or  $E$-theory as introduced by Connes and Higson. We refer to 
\cite{Bla86} for a treatment of these and their connections to Kasparov's theory.
%

As part of our introduction to $KK$-theory we will give a detailed and a fairly elementary 
proof of Kasparov's Bott-periodicity theorem in one dimension by constructing 
Dirac and dual Dirac elements which 
implement a $KK$-equivalence between $C_0(\RR)$ and the first complex Clifford algebra 
$\Cl_1$. We shall later use these computations to give a complete proof of the 
Baum-Connes conjecture for $\RR$ and $\ZZ$ with the help of Kasparov's Dirac-dual Dirac method.
This method is the most powerful tool for proving the conjecture and  has been successfully 
applied to a very large class of groups including all amenable groups. 
 As corollaries of our proof of the conjecture for $\RR$ and $\ZZ$, we shall also 
present proofs of Connes's Thom isomorphism for crossed products by $\RR$ 
and the Pimsner-Voiculescu six-term exact sequence for crossed products by $\ZZ$.

In the last part of this chapter we shall present the ``Going-Down'' principle 
which roughly says the following: Suppose $G$ satisfies the Baum-Connes conjecture with coefficients.
Then any $G$-equivariant $*$-homomorphism (or $KK$-class) 
between two  $G$-algebras $A$ and $B$ which induces isomorphisms between the $K$-theory groups
of  $A\rtimes K$ and $B\rtimes K$ for {\bf all} compact subgroups $K$ of $G$ also
 induces an isomorphism between the $K$-groups of $A\rtimes_rG$ and $B\rtimes_rG$.
We shall give a complete proof of this principle if $G$ is discrete and we present a number
of applications of this result. In particular, as one application we shall present a theorem about possible 
explicit computations of the $K$-theory of crossed products $C_0(\Om)\rtimes_rG$
in which a discrete group $G$ acts on a totally disconnected space $\Om$ with
some additional ``good'' properties which we shall explain in detail.
This result is basic for the $K$-theory computations of the reduced 
semi-group $C^*$-algebras as presented in  \cite{SgpC} and \cite{AlgAct}.

There are many other surveys on the Baum-Connes conjecture which look at the conjecture 
from quite different angles. 
The reader should definitely have a look at the paper \cite{BCH} of Baum, Connes, and Higson,
where a  broad discussion of  various applications of the conjecture is given. 
The survey \cite{ValBC} by Alain Valette restricts itself to a discussion of the Baum-Connes conjecture for 
discrete groups, but also provides a good discussion of applications to other important conjectures.
The survey \cite{MV} by Mislin discusses the conjecture from the topologist's point of view,
where the left hand side (the topological $K$-theory of $G$) is defined in terms of the 
Bredon cohomology -- a picture of the Baum-Connes conjecture 
 first given by  Davis and L{\"u}ck \cite{DL}. We also want to mention the 
paper \cite{HG} of Higson and Guentner, which gives an introduction of 
the Baum-Connes conjecture based on $E$-theory. Last but not least, we suggest to
the interested reader to study the book \cite{HR} by Higson and Roe, where many of the 
relevant techniques for producing important  $KK$-classes by elliptic operators (as the Dirac-class in 
$K$-homology) are treated in a very nice way.

Throughout this chapter we assume that the reader is familiar with the basics on $C^*$-algebras, the
basic constructions and properties of full and reduced crossed products and 
 the notion of Morita equivalence and Hilbert $C^*$-modules.
A detailed introduction to these topics is given in the first six sections of 
\cite{CroPro}.

The author of this chapter likes to thank Heath Emerson and Michael Joachim
for helpful discussions on some of the topics in this chapter.

\section{Operator $K$-Theory}\label{sec-Ktheory}
In this section we want to give a very brief overview of the definition and some basic properties 
of the $K$-theory groups of $C^*$-algebras. We urge the reader to have a look 
at one of the standard books on operator $K$-theory (e.g., \cites{Bla86, RL, WO}) for more detailed 
expositions of this theory.

Let us fix some notation: If $A$ is a $C^*$-algebra, we denote by $M_n(A)$  
the $C^*$-algebra of all $n\times n$-matrices 
over $A$ and by $A[0,1]$ the $C^*$-algebra of continuous functions $f:[0,1]\to A$. 
Moreover, we denote by $A^1=A\oplus \CC1$  the 
unique $C^*$-algebra with underlying vector space $A\oplus \CC1$ and with multiplication and involution 
given by 
$$(a+\lambda 1)(b+\mu1)=ab+\lambda b+\mu a+\lambda\mu 1\quad\text{and}\quad (a+\lambda 1)^*=a^*+\bar{\lambda}1,$$
for $a+\lambda 1, b+\mu 1\in A+\CC1$. Let $\epsilon:A_1\to\CC; \epsilon(a+\lambda 1)=\lambda$. 
We call $A^1$ the {\em unitisation} of $A$ (even if $A$ already has a unit).  
We write
$$M_\infty(A):=\cup_{n\in \NN}M_n(A)$$ 
where we regard $M_n(A)$ as a subalgebra of $M_{n+1}(A)$ via $T\mapsto\bmtr T&0\\0&0\emtr$.
We denote by $\mathcal P(A)$  the set of projections $p\in M_\infty(A)$, i.e., $p=p^*=p^2$.

\begin{definition}\label{def-equivKK}
Let $A$ be a unital $C^*$-algebra and let $p,q\in \mathcal P(A)$. Then $p,q$ are called
\begin{itemize}
\item {\em Murray-von Neumann equivalent} (denoted $p\sim q$) if there exist $x,y\in M_\infty(A)$ 
such that $p=xy$ and $q=yx$.
\item {\em unitarily equivalent} (denoted $p\sim_u q$) if there exists some $n\in \NN$ and a unitary $u\in U(M_n(A))$
such that $p,q\in M_n(A)$ and $q=upu^*$ in $M_n(A)$.
\item {\em homotopic} (denoted $p\sim_hq$), if there exists a projection $r\in \mathcal P(A[0,1])$ such that 
$p=r(0)$ and $q=r(1)$.
\end{itemize}
\end{definition}

All three equivalence relations coincide on $\mathcal P(A)$ (but not on the level of  $M_n(A)$ 
for fixed $n\in \NN$). 
If  $A$ is unital and 
if $p,q\in M_n(A)$, then $\bmtr p&0\\0&q\emtr$ and $\bmtr q&0\\0&p\emtr$ are Murray-von Neumann equivalent in $M_{2n}(A)$
with $x=\bmtr 0&p\\q&0\emtr$ and $y=\bmtr 0&q\\p&0\emtr$.
This allows us to define an
abelian semigroup structure on $\mathcal P(A)/\!\!\sim$ with addition given by
$[p]+[q]=\left[\bmtr p&0\\0&q\emtr\right]$. For every unital $C^*$-algebra $A$, we define 
$K_0(A)$ as the Grothendieck group of the semigroup $\mathcal P/\!\!\sim$, that 
is 
$$K_0(A)= \left\{\big[[p]-[q]\big]: [p],[q]\in \mathcal P(A)/\!\!\sim\right\}$$
where we write 
$\big[[p]-[q]\big]=\big[[p']-[q']\big]$ if and only if there exists $h\in \mathcal P(A)$ such that 
$$[p]+[q']+[h]=[p']+[q]+[h]\;\text{in}\;\mathcal P(A)/\!\!\sim.$$ 
\begin{example}\label{ex-CC} 
If $A=\CC$, then two projections $p,q\in \mathcal P(\CC)$ are 
homotopic, if and only if they have the same rank. It follows from this that 
$\mathcal P(\CC)/\!\!\sim\cong \NN$ as semigroup and hence we get $K_0(\CC)\cong \ZZ$.
\end{example}

If $\Phi: A\to B$ is a unital $*$-homomorphism between the unital $C^*$-algebras $A$ and $B$, 
then there exists a unique group homomorphism $\Phi_0:K_0(A)\to K_0(B)$ such that 
$\Phi_0([p])=[\Phi(p)]$. 

 We then define
$$K_0(A):=\ker\big(K_0(A^1)\stackrel{\epsilon_0}{\to} K_0(\CC)\cong \ZZ\big)$$
for any $C^*$-algebra $A$.
If $A$ is unital, then $A^1\cong A\oplus \CC$ as a direct sum of the $C^*$-algebras $A$ and $\CC$ (the isomorphism 
is given by $a+\lambda 1\mapsto (a-\lambda 1_A, \lambda)$) and it is not difficult to check that in this case both 
definitions of $K_0(A)$ coincide. 
Any $*$-homomorphism 
 $\Phi:A\to B$  extends to a unital $*$-homomorphism $\Phi^1:A^1\to B^1; \Phi^1(a+\lambda 1)=\Phi(a)+\lambda 1$
and the resulting map $\Phi^1_0:K_0(A^1)\to K_0(B^1)$ factors through a well defined homomorphism 
$\Phi_0:K_0(A)\to K_0(B)$. 

For the construction of $K_1(A)$ let $U_n(A^1)$ denote the group of unitary elements of  $M_n(A^1)$. 
We embed $U_n(A^1)$ into $U_{n+1}(A^1)$ via $U\mapsto \bmtr U&0\\0&1\emtr$, and we define 
$U_\infty(A^1)=\cup_{n\in\NN} U_n(A)$. Let $U_\infty(A^1)_0$ denote the path-connected component of $U_\infty(A^1)$, where 
we say that two unitaries can be joined by a path in $U_\infty(A^1)$ if and only if they can be joined by a continuous path
in $U_n(A^1)$ for some $n\in \NN$. 
For $u,v\in U_n(A^1)$, one can  check that
\begin{equation}\label{eq-unitaryhomotop}
\bmtr uv&0\\0&1\emtr\sim_h \bmtr u&0\\0&v\emtr\sim_h\bmtr v&0\\0&u\emtr
\end{equation}
in $U_{2n}(A^1)\subseteq U_{\infty}(A^1)$. 
Therefore, if we define
$$K_1(A):=U_\infty(A^1)/U_\infty(A^1)_0$$
with addition given by 
$$[u]+[v]=\left[\bmtr u&0\\0&v\emtr\right],$$
(which by (\ref{eq-unitaryhomotop}) is equal to $[uv]$)
we see that $K_1(A)$ is an abelian group. As for $K_0$, for unital $A$ we can alternatively construct $K_1(A)$ 
without passing to the unitization $A^1$ as $K_1(A)=U_\infty(A)/U_\infty(A)_0$. 

\begin{example}
Since $U_n(\CC)$ is path connected for all $n\in \NN$, we have $K_1(\CC)=\{0\}$.
\end{example}

If $\Phi:A\to B$ is a $*$-homomorphism, there is a well defined group homomorphism
$$\Phi_1:K_1(A)\to K_1(B); \Phi_1([u])=\big[\Phi^1(u)\big],$$
where, as before, $\Phi^1:A^1\to B^1$ denotes the unique unital extension of $\Phi$ to $A^1$.

\begin{proposition} The assignments $A\mapsto K_0(A), K_1(A)$ are homotopy invariant 
 covariant functors from the category of 
$C^*$-algebras to the category of abelian groups.
\end{proposition}

Homotopy invariance means that if $\Phi,\Psi:A\to B$ are two homotopic $*$-homomorphisms, 
then $\Phi_*=\Psi_*:K_*(A)\to K_*(B)$, $*=0,1$. Here a homotopy between $\Phi$ and $\Psi$ is a 
$*$-homomorphism $\Theta:A\to B[0,1]$ such that $\Phi=\epsilon_0\circ \Theta$ and 
$\Psi=\epsilon_1\circ \Theta$, where for all $t\in [0,1]$, $\epsilon_t: B[0,1]\to B$ denotes evaluation at $t$.
Of course, the homotopy invariance is a direct consequence of the fact that ``$\sim$'' coincides with 
``$\sim_h$'' on $\mathcal P(B^1)$. 

Recall that a $C^*$-algebra is called {\em contractible}, if  the identity $\id:A\to A$ is homotopic to the zero map $0:A\to A$.
As an example, let $A$ be any $C^*$-algebra, then $A(0,1]:=\{a\in A[0,1]:  a(0)=0\}$ is contractible.
A homotopy between $\id:A(0,1]\to A(0,1]$ and $0$ is given by the path of $*$-homomorphism
$\Phi_t:A(0,1]\to A(0,1]; \big(\Phi_t(a)\big)(s)=a(ts)$. 

\begin{corollary} Suppose that $A$ is a contractible $C^*$-algebra. Then $K_0(A)=\{0\}=K_1(A)$.
\end{corollary}

If $0\to I\stackrel{\iota}{\to} A\stackrel{q}{\to} B\to 0$ is a short exact sequence of $C^*$-algebras, then functoriality of $K_0$ and $K_1$ gives
two sequences
\begin{equation}\label{eq-Kseq}
K_0(I) \stackrel{\iota_0}{\to}K_0(A)\stackrel{q_0}{\to} K_0(B)\quad\text{and}\quad
K_1(I) \stackrel{\iota_1}{\to}K_1(A)\stackrel{q_1}{\to} K_1(B)
\end{equation}
which both can be shown to be exact in the middle. If there exists a splitting homomorphism $s:B\to A$ for the quotient map $q$,
it induces a splitting homomorphism $s_*: K_*(B)\to K_*(A)$, $*=0,1$, and  in this case the groups $K_0(A)$ and $K_1(A)$ decompose as direct 
sums 
$$K_0(A)=K_0(I)\oplus K_0(B)\quad\text{and}\quad K_1(A)=K_1(I)\oplus K_1(B).$$
In particular, in this special case the sequences in (\ref{eq-Kseq}) become short exact sequences of abelian groups. 

{\bf Six-term exact sequence:} In general, the sequences in (\ref{eq-Kseq}) can be joined into a six-term exact sequence
$$
\begin{CD}
K_0(I) @>\iota_0>> K_0(A) @>q_0>> K_0(B)\\
@A\partial AA    @.     @VV \exp V\\
K_1(B)  @<< q_1<   K_1(A)  @<< \iota_1 < K_1(I)
\end{CD}
$$
which serves as a very important tool for explicit computations as well as for proving theorems on $K$-theory.
We refer to \cites{Bla86, RL} for a precise description of the boundary maps $\exp$ and $\partial$.
Note that the six-term sequence is natural in the sense that if we have a morphism between two 
short exact sequences, i.e., we have a commutative diagram
$$
\begin{CD}
0 @>>> I @>>> A @>>>B @>>>0\\
@. @V\varphi VV  @V\psi VV @VV\theta V   @.\\
0 @>>> J @>>> C @>>>D @>>>0
\end{CD}
$$
in which the horizontal lines are exact sequences of $C^*$-algebras, then we have corresponding 
commutative diagrams
$$
\begin{CD}
@>>> K_i(A) @>>>K_i(B) @>>> K_{i+1}(I) @>>>K_{i+1}(A) @>>>\\
@. @V\psi_i VV @V \theta_i VV @VV \varphi_{i+1} V @VV\psi_{i+1}V  \\
@>>> K_i(C) @>>>K_i(D) @>>> K_{i+1}(J) @>>>K_{i+1}(C) @>>>\\
\end{CD}
$$
Aside of theoretical importance, this fact can often be used quite effectively for 
explicit computations of the boundary maps in the six-term sequence.

If we apply the six-term sequence to the short exact sequence
$$0\to C_0(0,1)\otimes A\stackrel{\iota}{\to} A(0,1]\stackrel{q}{\to} A\to 0$$
in which the quotient map $q: A(0,1]\to A$ is given by evaluation at $1$, we get the six-term sequence
$$
\begin{CD}
K_0(C_0(0,1)\otimes A) @>\iota_0>>0 @>q_0>> K_0(A)\\
@A\partial AA    @.     @VV \exp V\\
K_1(A)  @<< q_1<  0 @<< \iota_1 < K_1(C_0(0,1)\otimes A)
\end{CD}
$$
which shows that the connecting maps $\exp: K_0(A)\to K_1(C_0(0,1)\otimes A)$ and 
$\partial : K_1(A)\to K_0(C_0(0,1)\otimes A) $ are isomorphisms which are natural in $A$. 
Hence, by identifying $(0,1)$ with $\RR$
and $C_0(\RR)\otimes C_0(\RR)$ with $C_0(\RR^2)$, we can deduce 
the following important results from the above six-term sequence

\begin{theorem}[Bott-periodicity]\label{thm-BottK}
For each $C^*$-algebra $A$ there  are natural isomorphisms 
$$K_1(A)\cong K_0(C_0(\RR)\otimes A)\quad \text{and}\quad K_1(C_0(\RR)\otimes A)\cong K_0(A).$$
Moreover, if we apply the first isomorphism to $B=C_0(\RR)\otimes A$, we obtain isomorphisms
$$K_0(A)\cong K_1(C_0(\RR)\otimes A)\cong K_0( C_0(\RR^2)\otimes A).$$
\end{theorem}

We should note that the proof of the six-term sequence usually uses the 
Bott-periodicity theorem, so
we do present the results in the wrong order. In any case, 
the proofs of the six-term sequence and of the Bott-periodicity theorem are quite deep and we refer to 
the standard literature on $K$-theory (e.g., \cite{Bla86}) for the details.
We shall later provide a proof of Bott-periodicity in $KK$-theory, which does imply Theorem \ref{thm-BottK}.
We close this short section with two more important features of $K$-theory:

 {\bf Continuity:} 
  If  $A=\lim_{i}A_i$ is the inductive limit of an inductive system 
$\{A_i, \Phi_{ij}\}$ of $C^*$-algebras, then 
$$K_*(A)=\lim_i K_*(A_i),\quad *=0,1.$$
 {\bf Morita invariance:} If $e\in M(n,\CC)$ is any rank-one projection, then $\Phi_e:A\to M_n(A); \Phi_e(a)=e\otimes a$
induces an isomorphism between $K_*(A)$ and $K_*(M_n(A))$, $*=0,1$. More generally, if $e$ is any rank-one projection in $\K=\K(\ell^2(\NN))$,
then  the homomorphism $a\mapsto e\otimes a$ induces isomorphisms $K_*(A)\cong K_*(\K\otimes A)$, $*=0,1$. 
Since, by \cite{BGR} two $\sigma$-unital C*-algebras $A$ and $B$ are stably isomorphic if and only if they are Morita equivalent, 
it follows that Morita equivalent $\sigma$-unital C*-algebras have isomorphic $K$-theory groups.

\section{Kasparov's equivariant $KK$-theory}\label{sec-KK}

We now come to Kasparov's construction of the $G$-equivariant bivariant $K$-theory, which in 
some sense is built on the correspondence category as described in Section 1.5. 
Readers who are not familiar with the notion of Hilbert modules and correspondences are advised to 
read that section before going on here.
Since some of the constructions require that the $C^*$-algebras are  separable and that 
Hilbert $B$-modules $\mathcal E$ are countably generated (which means that there is a countable subset $\mathcal C\subseteq \mathcal E$ such that
$\mathcal C\cdot B$ is dense in $\mathcal E$), we shall from now on assume that 
these conditions will hold throughout, except for the multiplier algebras of separable $C^*$-algebras and the 
algebras of adjointable operators on a countably generated Hilbert module. We refer to \cite{Bla86} or Kasparov's original paper 
\cite{K3} for a more detailed account on where these conditions can be relaxed.

\subsection{Graded $C^*$-algebras and Hilbert modules.} 
We write $\ZZ_2$ for the group with two elements.
A $\ZZ_2$ grading of a $G$-$C^*$-algebra $(A,\alpha)$ is given by an action 
$\epsilon_A:\ZZ_2\to \Aut(A)$ which commutes with $\alpha$.
We then might consider $A$ as a $G\times \ZZ_2$-$C^*$-algebra with action $\alpha\times\eps_A$ and a 
graded $G$-equivariant correspondence between the graded $G$-$C^*$-algebras 
$(A,\alpha\times \eps_A)$ and $(B,\beta\times\eps_B)$ is just a $G\times\ZZ_2$-equivariant correspondence 
$(\E, u\times \eps_\E, \Phi)$ between these algebras.

Moreover, if $\eps_A$ is a grading of $A$, we write $A_0:=\{a\in A: \eps_A(a)=a\}$ and $A_1=\{a\in A: \eps_A(a)=-a\}$,
for the eigenspaces of the eigenvalues $1$ and $-1$ for $\eps_A$, and similar for gradings on Hilbert modules.
The elements in $A_0$ and $A_1$ are called the {\em homogeneous elements} of $A$. 
We write $\deg(a)=0$ if $a\in A_0$ and $\deg(a)=1$ if $a\in A_1$. $\deg(a)$ is called the {\em degree} of the 
homogeneous element $a$. 
Note that $A_0$ is a $C^*$-subalgebra of $A$ and every element $a\in A$ has a unique decomposition $a=a_0+a_1$ 
with $a_0\in A_0, a_1\in A_1$. 
If $A$ is a $\ZZ_2$ graded $C^*$-algebra, the graded commutator $[a,b]$ is defined as 
$$[a,b]=ab-(-1)^{\deg(a)\deg(b)} ba$$
for homogeneous elements $a,b\in A$ and it is defined on all of $A$ by bilinear continuation. 

If $A$ and $B$ are two graded $C^*$-algebras, we  define the {\em graded algebraic tensor product} $A\odot_{gr}B$
as the usual algebraic tensor product with graded multiplication and involution given on elementary tensors of homogeneous elements by
\begin{align*}
(a_1\otimes b_1)\cdot(a_2\otimes b_2)&=(-1)^{\deg(b_1)\deg(a_2)}(a_1a_2\otimes b_1b_2)\\
(a\otimes b)^*&=(-1)^{\deg(a)\deg(b)}(a^*\otimes b^*).
\end{align*}
In what follows we write $A\hat{\otimes}B$ for the minimal (or spatial) completion of $A\odot_{gr}B$. We refer to 
\cite{Bla86}*{14.4} for more details of this construction.

\begin{example}\label{ex-grading} 
{\bf (a)} For any $C^*$-algebra $A$ there is a grading on $M_2(A)$ given by conjugation with the symmetry $J=\bmtr 1&0\\0&-1\emtr$.
This grading is called the  {\em standard even} grading on $M_2(A)$. We then have 
$$M_2(A)_0=\bmtr A&0\\0&A\emtr\quad\text{and}\quad M_2(A)_1=\bmtr 0&A\\A&0\emtr.$$
{\bf (b)} If $A$ is a $C^*$-algebra, then the direct sum $A\oplus A$ carries a grading given by $(a,b)\mapsto (b,a)$, which 
is called the {\em standard odd} grading. We then have
$$\left(A\oplus A\right)_0=\{(a,a): a\in A\}\quad\text{and}\quad \left(A\oplus A\right)_1=\{(a,-a):a\in A\}.$$
{\bf (c)} Examples of nontrivially graded $C^*$-algebras which play an important r\^ole in the theory are the 
Clifford algebras $\Cl(V, q)$ where $q:V\times V\to\RR$ is a (possibly degenerate) symmetric  bilinear form on a finite dimensional real vector space 
$V$. $\Cl(V,q)$ is defined as the universal $C^*$-algebra generated by the elements $v\in V$ subject to the relations 
$$v^2=q(v,v)1\quad \forall v\in V$$
and  such that the embedding $\iota: V \into \Cl(V,q)$ is $\RR$-linear.
Using the equation $(v+v')^2-(v-v')^2=2(vv'+v'v)$  we obtain the relations
$$ vv'+v'v=2q(v, v') 1\quad \forall v,v'\in V.$$
If $\dim(V)=n$, then $\dim(\Cl(V,q))=2^n$. 
The grading on $\Cl(V,q)$ is given as follows: the linear span of all products of the form 
$v_1v_2\cdots v_m$ with $m=2k$ even is  the set of homogeneous elements of degree $0$ 
and the linear span of all such products with $m=2k-1$ odd is the set of homogeneous elements of degree $1$.

For all $n\in \NN_0$ we  write $\Cl_n$ for $\Cl(\RR^n, \langle\cdot,\cdot\rangle)$ where $\langle\cdot,\cdot\rangle$ denotes the
  standard inner product on $\RR^n$. 
Then  $\Cl_0\cong \CC$ and 
$\Cl_1=\CC1+  \CC e_1\cong \CC\oplus \CC$ with  the standard odd grading (sending $\lambda 1+\mu e_1$ to $(\lambda+\mu, \lambda-\mu)\in \CC^2$).
If $n=2$ and if $\{e_1,e_2\}$ is the standard orthonormal basis of $\RR^2$, then 
there is an isomorphism of $\Cl_2\cong M_2(\CC)$,  equipped with the standard even grading, given 
by sending the generator $e_1$ to $\bmtr 0&1\\1&0\emtr$ and $e_2$ to $i\bmtr 0&-1\\1&0\emtr$.

In general we have the formula $\Cl_n\hat\otimes\Cl_m\cong \Cl_{n+m}$ as graded $C^*$-algebras, where 
we use the {\em graded tensor product} and the diagonal grading on the left hand side of this equation.
Note that the isomorphism is given on the generators $\{v: v\in \RR^n\}$ and $\{w: w\in \RR^m\}$ 
by sending $v\otimes w$ to $(v,0)\cdot (0,w)\in \Cl_{n+m}$.
In particular, $\Cl_n$ can be constructed as the $n$th graded tensor product of $\Cl_1$ with itself.

Note that for any $C^*$-Algebra $A$, it is an easy exercise to show that the 
 graded tensor product $\left(A\oplus A\right)\hat\otimes \Cl_1$, where $A\oplus A$ is equipped 
with the standard odd grading, is isomorphic to $M_2(A)$ equipped with the standard even grading. 
As a consequence, it follows that
$$\Cl_{2n}\cong M_{2^n}(\CC)\quad\text{and}\quad \Cl_{2n+1}\cong  M_{2^n}(\CC)\oplus  M_{2^n}(\CC)$$
where the grading in the even case is given by conjugation with a symmetry $J\in M_{2^n}(\CC)$ (i.e., an isometry with $J^2=1$) 
and  the standard odd grading  in the odd case (e.g., see 
\cite{Bla86}*{\S 14.4} for more details).
\end{example}

Note that the grading $\eps_\E$ of a Hilbert $B$-module induces a  grading $\Ad\eps_\E$ 
on $\L_B(\E)$ and $\K(\E)$ in a canonical way and the  morphism $\Phi:A\to \L_B(\E)$ in a $\ZZ_2$-graded correspondence 
has to be equivariant for the given grading on $A$ and this grading on $\L_B(\E)$. 
In what follows we want to suppress the grading in our notation and just keep in mind that everything in sight  will 
be $\ZZ_2$ graded. In most cases (except for Clifford algebras), we shall consider the trivial grading $\eps_A=\id_A$ for our $C^*$-algebras $A$, 
but we shall usually have non-trivial gradings on our Hilbert modules. 

\subsection{Kasparov's bivariant $K$-groups.}\label{sec-KKgroup}

In this section we are going to give the definition of Kasparovs's $G$-equivariant bivariant $K$-groups.
We refer to \cite{K3} for the details (but see also \cites{Bla86, Skand}).
We start with the definition of the underlying $KK$-cycles:

\begin{definition}\label{def-KK-cycle} Suppose that $(A,\alpha)$ and $(B,\beta)$ are $\ZZ_2$-graded $G$-$C^*$-algebras.
A {\em $G$-equivariant $A$-$B$ Kasparov cycle} is a quadruple $(\E, u,\Phi, T)$ in which $(\E, u,\Phi)$ is a 
$\ZZ_2$-graded $(A,\alpha)$-$(B,\beta)$ correspondence and $T\in \L_B(\E)$ is a homogeneous element with $\deg(T)=1$ such that
\begin{enumerate}
\item $g\mapsto \Ad u_g(\Phi(a)T); G\to \L_B(\E)$ is continuous for all $a\in A$;
\item for all $a\in A$ and $g\in G$ we have 
$$(T-T^*)\Phi(a), \; (T^2-1)\Phi(a), \; (\Ad u_g(T)-T)\Phi(a), \; [T,\Phi(a)]\in \K(\E)$$
\end{enumerate} 
(where $[\cdot,\cdot]$ denotes the graded commutator).
Two Kasparov cycles $(\E, u,\Phi, T)$ and $(\E', u',\Phi', T')$ are called {\em isomorphic}, if there exists an isomorphism 
$W:\E\to\E'$ of the correspondences $(\E, u,\Phi)$ and $(\E', u',\Phi')$ such that $T' = W\circ T\circ W^{-1}$. 
A Kasparov cycle is called to be {\em degenerate} if 
$$(T-T^*)\Phi(a), \; (T^2-1)\Phi(a),\;  (\Ad u_g(T)-T)\Phi(a), \; [T,\Phi(a)]=0$$
for all $a\in A$ and $g\in G$.
We write $\EE_G(A,B)$ for the set of isomorphism classes of all $G$-equivariant $A$-$B$ Kasparov cycles and we 
write $\DD_G(A,B)$ for the equivalence classes of degenerate Kasparov cycles.
\end{definition}

\begin{example} Every $G$-equivariant $*$-homomorphism $\Phi:A\to B$ determines a 
$G$-equivariant $A$-$B$ Kasparov cycle $(B, \beta, \Phi, 0)$, where $B$ is considered as a 
Hilbert $B$-module in the obvious way. More generally, if $(\E, u,\Phi)$ is any 
$\ZZ_2$ graded $(A,\alpha)$-$(B,\beta)$ correspondence such that $\Phi(A)\subseteq \K(\E)$
(i.e., $(\E, u,\Phi)$  is a morphism in  the compact correspondence category in the sense 
of \cite{CroPro}*{Definition 2.5.7}, then it is an easy 
exercise to check that $(\E, u,\Phi, 0)$
is a $G$-equivariant $A$-$B$ Kasparov cycle as well. 
Note that the condition $(T^2-1)\Phi(a)\in \K(\E)$ for a Kasparov cycle implies that 
conversely, if $(\E, u,\Phi, 0)$ is a Kasparov cycle, then $\Phi(A)\subseteq \K(\E)$.
A special situation of the above is the case in which $A=B$ and $\Phi=\id_B$ which 
gives us the 
Kasparov cycle
$(B, \beta,\id_B, 0)$. It will play an important role when looking at Kasparov products below.
\end{example}

In what follows, if $(B,\beta)$ is a $\ZZ_2$-graded $G$-$C^*$-algebra, then we denote by $B[0,1]$ the algebra $C([0,1],B)$ 
with point-wise $G$-action and grading. Suppose now that $(\E,u,\Phi, T)$ is a $G$-equivariant 
 $A$-$B[0,1]$ Kasparov cycle. For each $t\in [0,1]$ let $\delta_t:B[0,1]\to B; \delta_t(f)=f(t)$ be evaluation at $t$. 
 Then we obtain a $G$-equivariant $A$-$B$ Kasparov cycle 
 $(\E_t, u_t,\Phi_t, T_t)$ by putting
 $$\E_t=\E\otimes_{B[0,1],\delta_t}B,\; u_t=u\otimes \beta, \;\Phi_t=\Phi\otimes 1,\; \text{and}\; T_t=T\otimes 1.$$
Alternatively, we could define a $B$-valued inner product on $\E$ by 
$$\lk e,f\rk_B:=\lk e,f\rk_{B[0,1]}(t)$$
which factors through $\E_t:= \E/(\E\cdot \ker\delta_t)$. Then  $u, \Phi, T$ factor uniquely through some action 
$u_t$ of $G$, a $*$-homomorphism  $\Phi_t:A\to \L_B(\E_t)$ and an operator $T_t\in \L_B(\E_t)$
such that  $(\E_t, u_t,\Phi_t, T_t)$ is a $G$-equivariant $A$-$B$ Kasparov cycle.  It is isomorphic to the one
constructed above. We call
 $(\E_t, u_t,\Phi_t, T_t)$ the {\em evaluation} of $(\E,u,\Phi, T)$ at $t\in [0,1]$.

\begin{definition}[Homotopy] Two Kasparov cycles $(\E_0,u_0,\Phi_0,T_0)$ and $(\E_1,u_1,\Phi_1,T_1)$ in 
$\EE_G(A,B)$ are said to be {\em homotopic} if there exists a $G$-equivariant 
$A$-$B[0,1]$ Kasparov cycle $(\E,u,\Phi, T)$  such that 
 $(\E_0,u_0,\Phi_0,T_0)$ is isomorphic to the evaluation of  $(\E,u,\Phi,T)$ at $0$ and 
 $(\E_1,u_1,\Phi_1,T_1)$ is isomorphic to the evaluation of  $(\E,u,\Phi,T)$ at $1$.
 We then write $(\E_0,u_0,\Phi_0,T_0)\sim_h (\E_1,u_1,\Phi_1,T_1)$. 
 We  define 
 $$KK^G(A,B):=\EE_G(A,B)/\sim_h.$$
 \end{definition}

\begin{remark}\label{rem-degenerate}
Every degenerate Kasparov cycle $(\E,u,\Phi,T)$ is homotopic to the zero-cycle 
$(0,0,0,0)$. To see this consider the quadruple $(\E\otimes C_0([0,1)), u\otimes\id, \Phi\otimes 1, T\otimes 1)$
where we view $\E\otimes C_0([0,1))\cong C_0([0,1),\E)$ as a $B[0,1]$-Hilbert module in the obvious way.
It follows from degeneracy of $(\E,u,\Phi,T)$ that $(\E\otimes C_0([0,1)), u\otimes\id, \Phi\otimes 1, T\otimes 1)$ is an 
$A$-$B[0,1]$ Kasparov cycle and it is straightforward to check that its evaluation at $0$ coincides with 
$(\E,u,\Phi,T)$ while its evaluation at $1$ is the zero-cycle.
\end{remark}

\begin{remark}\label{rem-homotopies}
A special kind of homotopies are the {\em operator homotopies} which are defined as follows:
Assume that $(\E,u,\Phi,T_0)$ and $(\E,u,\Phi,T_1)$ are two $A$-$B$ Kasparov cycles such that 
the underlying correspondence $(\E,u,\Phi)$ coincides for both cycles. 
An operator homotopy between $(\E,u,\Phi,T_0)$ and $(\E,u,\Phi,T_1)$ is a family 
of $A$-$B$ Kasparov cycles $(\E,u,\Phi,T_t)$, $t\in [0,1]$, such that the path of operators 
$(T_t)_{t\in [0,1]}$ is norm continuous and connects $T_0$ with $T_1$.  Such operator homotopy determines 
a homotopy between  $(\E,u,\Phi,T_0)$ and $(\E,u,\Phi,T_1)$ in which the
$A$-$B[0,1]$ Kasparov cycle  is given by $(\E\otimes C[0,1], u\otimes \id,  \Phi\otimes 1, \tilde{T})$
with $\big(\tilde{T}(e)\big)(t)=T_te(t)$ for $e\in\E\otimes C[0,1]=C([0,1],\E)$.
\end{remark}

\begin{example}\label{rem-compact-pertubation}
Assume that $(\E,u,\Phi,T_0)$ and $(\E,u,\Phi,T_1)$ are two $A$-$B$ Kasparov cycles. We then say that 
$T_1$ is a {\em compact perturbation} of $T_0$ if $(T_1-T_0)\Phi(a)\in \mathcal \K(\E)$ for all $a\in A$. In this case, the 
path $T_t=(1-t)T_0+tT_1$ gives an operator homotopy between $T_0$ and $T_1$, so both Kasparov cycles 
are homotopic. 
\end{example}

The following observation is due to Skandalis (see \cite{Skand}):

\begin{proposition}\label{prop-homotopies}
The equivalence relation $\sim_h$ on  $\EE_G(A,B)$ coincides with the equivalence relation generated by 
operator homotopy together with adding degenerate Kasparov cycles. 
\end{proposition}

\begin{theorem}[Kasparov]\label{thm-KKgroup} 
$KK^G(A,B)$  is an abelian group with addition 
defined by direct sum of Kasparov cycles:
$$[\E_1,u_1,\Phi_1, T_1]+[\E_2,u_2,\Phi_2, T_2]=[\E_1\oplus \E_2, u_1\oplus u_2, \Phi_1\oplus \Phi_2, T_1\oplus T_2],$$
where $[\E,u,\Phi,T]$ denotes the homotopy class of the Kasparov cycle $(\E,u,\Phi,T)$. The inverse of a class $[\E,u,\Phi, T]$ is given by 
the class $[\E^{\op}, u,\Phi\circ \eps_A, -T]$ in which $\E^{\op}$ denotes the module $\E$ with the {\em opposite grading}
$\eps_{\E^{\op}}=-\eps_{\E}$.
\end{theorem}

{\bf Functoriality:}  Every $G$-equivariant $*$-homomorphism $\Psi:A_1\to A_2$ induces a 
group homomorphism 
$$\Psi^*: KK^G(A_2,B)\to KK^G(A_1, B)\;;\; [\E, u, \Phi, T]\mapsto [\E, u, \Phi\circ \Psi, T]$$
and every $G$-equivariant $*$-homomorphism $\Psi: B_1\to B_2$ induces a group homomorphism
$$\Psi_*: KK^G(A, B_1)\to KK^G(A, B_2)\;;\;[\E, u,\Phi, T]\mapsto [\E\otimes_{B_1}B_2, u\otimes\beta_2, \Phi\otimes 1, T\otimes 1].$$
Hence, $KK^G$ is contravariant in the first variable  and covariant in the second variable. 

{\bf Direct sums.} If $A=\oplus_{i=1}^l A_i$ is a finite direct sum, then $KK^G(B,A)\cong \oplus_{i=1}^lKK^G(B, A_i)$.
The formula does not hold in general for (countable) infinite direct sums (see \cite{Bla86}*{19.7.2}).
On the other side, if  $A=\oplus_{i\in I} A_i$ is a countable direct sum of $G$-$C^*$-algebras, then 
$ KK^G(A,B)\cong\prod_{i\in I} KK^G(A_i,B)$
for every $G$-$C^*$-algebra $B$. We leave it to the reader to construct these isomorphisms.

{\bf The ordinary $K$-theory groups.} Recall that for a trivially graded unital
$C^*$-algebra $B$ the ordinary $K_0$-group 
can be defined as the Grothendieck group generated by the 
semigroup of all  
homotopy classes $[p]$ of projections $p\in M_\infty(B)=\cup_{n\in \NN}M_n(B)$ with direct sum 
$[p]+[q]=[p\oplus q]$ as addition.
If $p\in M_n(B)$, then  $p$ determines a $*$-homomorphism $\Phi_p:\CC\to M_n(B)\cong \K(B^n); \lambda\mapsto \lambda p$,
and hence a class $[B^n, \Phi_p,0]\in KK_0(\CC,B)$. Note that all elements of the module $B^n$ are homogeneous of degree $0$.
This construction preserves homotopy and direct sums and therefore 
induces a homomorphism of $K_0(B)$ into $KK(\CC,B)$, which, as shown by Kasparov, is actually an 
isomorphism of abelian groups. Thus we get
$$KK(\CC, B)\cong K_0(B).$$ 
If $B$ is not unital, we may first apply the above to the unitisation $B^1$ and for $\CC$ and then use 
split-exactness to get the general case. If $\Phi:A\to B$ is a $*$-homomorphism, then the 
isomorphisms $KK(\CC,A)\cong K_0(A)$ and $KK(\CC,B)\cong K_0(B)$ intertwines the induced 
homomorphism $\Phi_*:KK(\CC,A)\to KK(\CC,B)$ in  $KK$-theory with the morphism of $K$-theory groups 
sending $[p]$ to $[\Phi(p)]$.

If we put the complex numbers $\CC$ into the second variable, we obtain Kasparov's 
operator theoretic $K$-homology functor $K^0(A):=KK(A,\CC)$. 

{\bf The group $KK(\CC, \CC)$.} Each element in $KK(\CC,\CC)$ can be represented 
by a Kasparov cycle of the form $(\H, {\mathbf 1}, T)$ in which $\H=\H_0\oplus \H_1$ is a graded Hilbert space,
${\mathbf 1}:\CC\to \L(\H)$ is the action  ${\mathbf 1}(\lambda)\xi=\lambda \xi$ and $T$ is a self adjoint operator
satisfying $T^2-1\in \K(H)$. This follows from the standard simplifications as described in detail in \cite{Bla86}*{\S 17.4}.
If $T=\bmtr 0&P^*\\P&0\emtr$, then $T^2= \bmtr P^*P & 0\\ 0 & PP^*\emtr$ and the condition $T^2-1\in \K(H)$ then 
implies that $P$ is a Fredholm operator. We then obtain a well defined map
$$\indx: KK(\CC,\CC)\to \ZZ; [\H, {\mathbf 1}, T]\mapsto \indx(\H, {\mathbf 1}, T):=\indx(P),$$
where $\indx(P)=\dim(\ker(P))-\dim(\ker(P^*))$ denotes the Fredholmindex of $P$.
The index map induces the isomorphism $KK(\CC,\CC)\cong \ZZ=K_0(\CC)$. (Compare this with the above 
isomorphism $KK(\CC, B)\cong K_0(B)$ in case $B=\CC$.)

\subsection{The Kasparov product}
We are now coming to the main feature of Kasparov's $KK$-theory, namely the Kasparov product which 
is a pairing 
$$KK^G(A,B)\times KK^G(B,C)\to KK^G(A,C),$$
where $A,B,C$ are $G$-$C^*$-algebras.
Starting with an $A$-$B$ Kasparov cycle $(\E, u,\Phi, T)$ and a $B$-$C$ Kasparov cycle 
$(\F, v, \Psi, S)$, the Kasparov product will be represented by a Kasparov cycle 
of the form $(\E\otimes_B\F, u\otimes v, \Phi\otimes 1, R)$, where all ingredients
with exception of the operator $R$ are well known objects by now: 
$\E\otimes_B\F$ denotes the internal tensor product of $\E$ with $\F$ over $B$ (with diagonal grading),
$u\otimes v:G\to\Aut(\E\otimes_B\F)$ denotes the diagonal action, and $\Phi\otimes 1: A\to \L(\E\otimes_B\F)$ 
sends $a\in A$ to the operator $\Phi(a)\otimes 1$ of $\E\otimes_B\F$. 
But the construction of the operator $R$ is, unfortunately, quite complicated and 
 reflects the high complexity of Kasparov's theory. 
 
As a first attempt one would look at the operator 
 $$R=T\otimes 1+1\otimes S.$$
But there are several problems with this choice. First of all, 
 the operator $1\otimes S$ on the internal tensor product 
 is not well defined in general (it only makes sense, if $S$ commutes with $\Psi(B)\subseteq \L(\F)$).
 To resolve this, we need to replace $1\otimes S$ with a so-called {\em $S$-connection}, which we 
 shall explain below. But even if $1\otimes S$ is well-defined, the triple
 $(\E\otimes_B\F, u\otimes v, \Phi\otimes 1, T\otimes 1 + 1\otimes S)$
will usually fail to be a Kasparov  triple unless $S=0$, and one needs to replace 
$T\otimes 1+1\otimes S$ by a combination 
$$M^{1/2}(T\otimes 1)+N^{1/2}(1\otimes S)$$
where $M, N\geq 1$ are suitable operators with $M+N=1$ which can be obtained 
by an application of Kasparov's technical theorem \cite{K3}*{Theorem 1.4}.

{\bf $S$-connection:} For any $\xi\in \E$ define 
 $$\theta_\xi: \F\to \E\otimes_B\F; \theta_\xi(\eta)=\xi\otimes \eta.$$
 Then $\theta_\xi\in \K(\F,\F\otimes_B\E)$  with adjoint given by
 $$\theta_\xi^*(\zeta\otimes\eta)=\Psi(\lk \xi, \zeta\rk_B)\eta.$$
 An operator $S_{12}\in \L(\E\otimes_B\F)$ is then called an {$S$-connection},  if 
 for all homogeneous elements $\xi\in \E$ we have
\begin{equation}\label{eq-connect}
 \theta_\xi S-(-1)^{\deg(\xi)\deg{S}} S_{12}\theta_\xi, \;\; \theta_\xi S^*-(-1)^{\deg(\xi)\deg{S}} S_{12}^*\theta_\xi \in \K(\F, \E\otimes_B\F)
 \end{equation}
 It is a good exercise to check that if $S$ commutes with $\Psi(B)\subseteq \L(\F)$, then 
 $S_{12}=1\otimes S$ makes sense and is an $S$-connection in the above sense. 

\begin{definition}[Kasparov product]
Suppose that $A,B,C$ are $G$-$C^*$-algebras, and that $(\E, u,\Phi, T)$ is an $A$-$B$ Kasparov cycle  and 
$(\F, v, \Psi, S)$ is a $B$-$C$ Kasparov cycle. Let $S_{12}\in \L(\E\otimes_B\F)$ be an $S$-connection as 
explained above. Then the quadruple $(\E\otimes_B\F, u\otimes v, \Phi\otimes 1, S_{12})$ will be 
a {\em Kasparov product} for $(\E, u,\Phi, T)$ and $(\F, v, \Psi, S)$ if the following two conditions hold:
\begin{enumerate}
\item $(\E\otimes_B\F, u\otimes v, \Phi\otimes 1, S_{12})$ is an $A$-$C$ Kasparov cycle.
\item For all $a\in A$ we have $(\Phi(a)\otimes 1)\left[T\otimes 1, S_{12}\right](\Phi(a^*)\otimes 1)\geq 0$ mod. $\K(\E\otimes_B\F)$.
\end{enumerate}
In this case the class $[\E\otimes_B\F, u\otimes v, \Phi\otimes 1, S_{12}]\in KK^G(A,C)$ is called a Kasparov product of 
$[\E, u,\Phi, T]\in KK^G(A,B)$ with $[\F, v, \Psi, S]\in KK^G(B,C)$.
\end{definition}

We should note that the existence of an $S$-connection $S_{12}$ which satisfies the conditions of the above definition 
follows from an application of Kasparov's technical theorem \cite{K3}*{Theorem 14}. The proof is quite technical and we 
refer to the literature (see one of the references \cite{Skand, K3, Bla86}).

\begin{remark}\label{rem-homotop} 
{\bf (a)} One can show that the operator $S_{12}$ in a Kasparov product is unique up to operator homotopy.

{\bf (b)} 
The following easy case is often very useful: 
Suppose that $B$ acts on $\F$ by compact operators, i.e., $\Psi:B\to \L(\F)$ takes its values in $\K(\F)$.
Then $(\F, v, \Psi, 0)$ is a $B$-$C$ Kasparov cycle and the $0$-operator on $\E\otimes_B\F$ is then clearly 
a $0$-connection. Now if $\Phi:A\to \L(\E)$ also takes values in $\K(\E)$ and if $T=0$, it follows that 
$[\E\otimes_B\F, u\otimes v, \Phi\otimes 1, 0]$ is a Kasparov product for $[\E,u,\Phi,0]$ and $[\F, v, \Psi, 0]$.
We therefore obtain a functor from the compact correspondence category $\Corr_c(G)$ (see \cite{CroPro}*{Section 2.5.3}) -- here we use countably 
generated $\ZZ_2$-graded Hilbert modules) into $KK^G$ given by $(\E, u,\Phi)\mapsto [\E, u, \Phi, 0]$ which preserves multiplication.  
\end{remark}

The details of the following theorem can be found in \cite{Bla86} or \cite{K3}.
 
\begin{theorem}[Kasparov]\label{thm-product}
Suppose that $A,B, C$ are separable $G$-$C^*$-algebras and let 
$x=[\E,u, \Phi, T]\in KK^G(A,B)$ and $y=[\F, v, \Psi, S]\in KK^G(B,C)$. Then the Kasparov product
$$x\otimes_By:=[\E\otimes_B\F, u\otimes v, \Phi\otimes 1, S_{12}]$$ 
exists and induces a well-defined bilinear pairing
$$\otimes_B:KK^G(A,B)\times KK^G(B,C)\to KK^G(A,C).$$
Moreover, the Kasparov product is associative: If $D$ is another $G$-$C^*$-algebra
and $z\in KK^G(C,D)$, then we have
$$(x\otimes_By)\otimes_Cz=x\otimes_B(y\otimes_Cz)\in KK^G(A,D).$$
The elements $1_A=[A,\alpha, \id_A,0]\in KK^G(A,A)$ and $1_B=[B,\beta,\id_B,0]\in KK^G(B,B)$ act as identities 
from the left and right on $KK^G(A,B)$, i.e., we have
$$1_A\otimes_Ax=x=x\otimes_B1_B\in KK^G(A,B)$$
for all $x\in KK^G(A,B)$. In particular, $KK^G(A,A)$ equipped with the Kasparov product is a unital ring.
\end{theorem}

The following result is helpful for the computation of Kasparov products in some 
important special cases. For the proof we refer to \cite{Bla86}*{8.10.1}.

\begin{proposition}\label{prop-product}
Suppose that $A,B,C$ are $G$-$C^*$-algebras, $(\E,u, \Phi, T)$ is an $A$-$B$ Kasparov cycle 
and  $(\F, v, \Psi, S)$ is a $B$-$C$ Kasparov cycle.
Suppose further that $T=T^*$ and $\|T\|\leq 1$. Let $S_{12}\in \L(\E\otimes_B\F)$ be a $G$-invariant $S$-connection 
of degree one and let
$$R:= T\otimes 1 + (\sqrt{1-T^2}\otimes 1)S_{12}.$$
If $[R, \Phi(A)\otimes 1]\in \K(\E\otimes_B\F)$, then  $(\E\otimes_B\F, u\otimes v,\Phi\otimes 1, R)\in \EE^G(A,C)$ and represents
the Kasparov product of 
$[\E,u, \Phi, T]$ with $[\F, v, \Psi, S]$
\end{proposition}

We should notice that the conditions $T=T^*$ and $\|T\|\leq 1$ can always be fulfilled by choosing 
an appropriate Kasparov cycle representing the given class $x\in KK^G(A,B)$. 

Associativity and the existence of neutral elements for the Kasparov product gives rise to 
an easy notion of $KK^G$-equivalence for two $G$-$C^*$-algebras $A$ and $B$:
Assume that there are elements $x\in KK^G(A,B)$ and $y\in KK^G(B,A)$ such that 
$$x\otimes_By=1_A\in KK^G(A,A)\quad\text{and}\quad y\otimes_Ax=1_B\in KK^G(B,B).$$
Then taking products with $x$ from the left induces an isomorphism
$$x\otimes_B\cdot: KK^G(B,C)\to KK^G(A,C); z\mapsto x\otimes_B z$$
with inverse given by 
$$y\otimes_A\cdot: KK^G(A,C)\to KK^G(B,C); w\mapsto y\otimes_A w.$$
This follows from the simple identities
$$y\otimes_A(x\otimes_B z)=(y\otimes_Ax)\otimes_Bz=1_B\otimes_Bz=z$$
for all $z\in KK^G(B,C)$ and similarly we have $x\otimes_B(y\otimes_Aw)=w$
for all $w\in KK^G(A,C)$. Of course, taking products from the right by $x$ and $y$ 
will give us an isomorphism
$\cdot\otimes_Ax: KK^G(C,A)\to KK^G(C,B)$ with inverse  
$\cdot\otimes_By: KK^G(C,B)\to KK^G(C,A)$. 

\begin{definition}\label{def-KK-equivalence}
Suppose that $x\in KK^G(A,B)$ and $y\in KK^G(B,A)$ are as above. Then 
we say that $x$ is a {\em $KK^G$-equivalence from $A$ to $B$ with inverse $y$}.
Two $G$-$C^*$-algebras $A$ and $B$ are called {\em $KK^G$-equivalent}, 
if such elements $x$ and $y$ exist.
\end{definition}

\begin{lemma}\label{lem-Morita-KK}
Suppose that $(\E, u, \Phi)$ is a $G$-equivariant $A$-$B$ Morita equivalence for the
$G$-$C^*$-algebras $(A,\alpha)$ and $(B,\beta)$ and let  $(\E^*, u^*,\Phi^*)$ denote its 
inverse. Then $x=[\E,u,\Phi, 0]\in KK^G(A,B)$ is a $KK^G$-equivalence with inverse 
$y=[\E^*,u^*,\Phi^*, 0]$.
\end{lemma}
\begin{proof} It follows from Remark \ref{rem-homotop} that the Kasparov product 
$x\otimes_By$ is represented by the Kasparov cycle $[\E\otimes_B\E^*, u\otimes u^*, \Phi\otimes 1, 0]$. 
But since the correspondence $(\E\otimes_B\E^*, u\otimes u^*, \Phi\otimes 1)$ 
is isomorphic to $(A,\alpha,\id_A)$, it follows that $x\otimes_By=1_A$.  
Similarly we have $y\otimes_Bx=1_B$.
\end{proof}

\begin{remark}\label{rem-full-corner}
If $p\in M(A)$ is a full projection in a $C^*$-algebra $A$, then $pA$ is a 
$pAp$-$A$ Morita equivalence and hence, if $\varphi:pAp\to \L(pA)$ denotes the 
canonical morphism, the element $[pA, \varphi, 0]\in KK(pAp, A)$ is a $KK$-equivalence
(it is a $KK^G$-equivalence if $A$ is a $G$-$C^*$-algebra and $p\in M(A)$ is $G$-invariant).
On the other hand, if $\psi:pAp\to A$ denotes the canonical inclusion, then 
$[A, \psi, 0]$ determines the class of the $*$-homomorphism $\psi$.
Both classes actually coincide, which follows from the simple fact that we can decompose 
the $KK$-cycle $(A,\psi,0)$ as the direct sum $(pA, \varphi, 0)\oplus ((1-p)A, 0, 0)$
where the second summand is degenerate. 
In particular, it follows from this that $\psi_*:K_*(pAp)\to K_*(A)$ is an isomorphism.
\end{remark}

We are now going to describe a more general version of the Kasparov product.
For this we  first need to introduce a homomorphism 
$$\cdot \hat\otimes 1_D: KK^G(A,B)\to KK^G(A\hat\otimes D, B\hat\otimes D)$$ 
which is defined for any $G$-$C^*$-algebra $(D,\delta)$ by 
$$[\E, u, \Phi, T]\hat\otimes 1_D:=[\E\hat\otimes D, u\hat\otimes \delta, \Phi\hat\otimes 1, T\hat\otimes 1].$$
One can check that $\cdot\hat\otimes 1_D$ is compatible with Kasparov products in the sense that
$$(x\otimes_By)\hat\otimes 1_D=(x\hat\otimes 1_D)\otimes_{B\hat\otimes D}(y\hat\otimes 1_D)$$
and it follows directly from the definition that  $1_A\hat\otimes 1_D=1_{A\hat\otimes D}$. 
In particular, it follows that $\cdot\hat\otimes 1_D$ sends $KK^G$-equivalences to $KK^G$-equivalences.
Of course, in a similar way we can define a homomorphism 
$$1_D\otimes \cdot: KK^G(A,B)\to KK^G(D\hat\otimes A, D\hat\otimes B).$$

\begin{remark} By our conventions ``$\hat\otimes$'' denotes the minimal graded tensor product 
of the $C^*$-algebras $A$ and $B$. But a similar map $\cdot\hat\otimes_{\max}D: KK^G(A,B)\to KK^G(A\hat\otimes_{\max}D,B\hat\otimes_{\max}D)$
exists for the maximal graded tensor product.
\end{remark}

\begin{theorem}[Generalized Kasparov product]\label{gen-product}
Suppose that $(A_1,\alpha_1)$, $(A_2, \alpha_2)$, $(B_1,\beta_1)$, $(B_2, \beta_2)$ and $(D,\delta)$ are 
$G$-$C^*$-algebras. Then there is a pairing
$$\otimes_D: KK^G(A_1, B_1\hat\otimes D)\times KK^G(D\hat\otimes A_2, B_2)\to 
KK^G(A_1\hat\otimes A_2, B_1\hat\otimes B_2)$$
given by 
$$(x,y)\mapsto x\otimes_D y:= (x\hat\otimes 1_{A_2})\otimes_{B_1\hat\otimes D\hat\otimes A_2}(1_{B_1}\hat\otimes y).$$
This pairing is associative (in a suitable sense) and coincides with the ordinary Kasparov product 
if $B_1=\CC=A_2$. Moreover, in case $D=\CC$,  the product 
$$\otimes_\CC: KK^G(A_1, B_1)\times KK^G(A_2, B_2)\to 
KK^G(A_1\hat\otimes A_2, B_1\hat\otimes B_2)$$
is commutative.
\end{theorem}
Note that there are several other important properties of the generalized Kasparov product, which we don't want 
to state here. We refer to \cite{K3}*{Theorem 2.14} for the complete list and their proofs. We close this section with a useful 
description of Kasparov cycles in terms of unbounded operators due to Baaj and Julg (see \cite{BJ}).  For our purposes 
it suffices to restrict to the case where $B=\CC$, in which case we may rely on the classical theory of unbounded operators
on Hilbert spaces. But the picture extends to the general case using a suitable theory of regular unbounded  operators 
on Hilbert modules.

\begin{lemma}[Baaj-Julg]\label{lem-unbounded}
Suppose that $A$ is a graded $C^*$-algebra and $\Phi:A\to \B(\H)$ is a graded $*$-representation of $A$ on the 
graded separable Hilbert space $\H$. Suppose further that $D=\bmtr 0 & d^*\\ d &0\emtr$ 
is an unbounded selfadjoint operator on $\H$ of degree one such that 
\begin{enumerate}
\item $(1+D^2)^{-1}\Phi(a)\in \K(\H)$ for all $a\in A$, and
\item the set of all $a\in A$ such that $[D,\Phi(a)]$ is densely defined and bounded is dense in $A$.
\end{enumerate}
Then $(\H, \Phi, T=\frac{D}{\sqrt{1+D^2}})$ is an $A$-$\CC$ Kasparov cycle. 
\end{lemma}
For the proof of this lemma, even in the more general context of $A$-$B$ Kasparov cycles, we refer to
\cite{Bla86}*{Proposition 17.11.3}. Note that the operator $T=\frac{D}{\sqrt{1+D^2}}$ is constructed via  
 functional calculus for unbounded selfadjoint operators.

\subsection{Higher $KK$-groups and Bott-periodicity}\label{sec-Bott}
\begin{definition}\label{def-higherKK}
Suppose that $(A,\alpha)$ and $(B,\beta)$ are $G$-$C^*$-algebras. 
For each $n\in \NN_0$ we define the (higher) $KK^G$-group as 
$$KK^G_n(A,B):=KK^G(A, B\hat\otimes \Cl_n)\quad\text{and}\quad KK^G_{-n}(A,B):=KK^G(A\hat\otimes \Cl_n, B),$$
where $\Cl_n$ denotes the $n$th complex Clifford algebra with trivial $G$-action and grading   as defined
  in Section \ref{sec-KK}. 
 \end{definition}
 
 With this definition of higher $KK$-groups it is easy to prove a (formal) version of Bott-periodicity. We need
 the following easy lemma: 
 
 \begin{lemma}\label{lem-cliff}
 If $n\in \NN$ is even, then $\Cl_n$ is Morita equivalent to $\CC$ as graded $C^*$-algebras. If $n$ is odd,
 then $\Cl_n$ is Morita equivalent to $\Cl_1$ as graded $C^*$-algebra. 
 \end{lemma}
 \begin{proof} Let $n\in \NN_0$. 
 We know from Section \ref{sec-KK} that 
 $\Cl_{2n}\cong M_{2^{n}}(\CC)$ with grading given by cunjugation with a symmetry $J\in M_{2^{n}}(\CC)$.
 It is then easy to check that the Hilbert space $\CC^{2^{n}}$ equipped with the grading operator $J$ and 
 the canonical left action of $\Cl_n$ on $\CC^{2^{n}}$ gives the desired Morita equivalence.
Similarly, we have $\Cl_{2n+1}\cong M_{2^{n}}(\CC)\oplus M_{2^{n}}(\CC)\cong M_{2^n}\hat\otimes \Cl_1$ 
as graded $C^*$-algebras, which is Morita equivalent to $\CC\hat\otimes \Cl_1=\Cl_1$.
\end{proof}

\begin{notation}\label{notation-Morita}
In what follows we shall denote by $x_{2n}\in KK(\Cl_{2n},\CC)$ the (invertible) class of the Morita equivalence between $\Cl_{2n}$ and $\CC$
as in the above lemma and by $x_{2n+1}\in KK(\Cl_{2n+1},\Cl_1)$ the class of the Morita equivalence between $\Cl_{2n+1}$ and $\Cl_1$.
\end{notation}

\begin{proposition}[Formal Bott-periodicity]\label{prop-Bott}
For each $n\in \NN_0$ there are canonical isomorphisms
$$KK^G(A\hat\otimes \Cl_{2n}, B)\cong KK^G_0(A,B)\cong KK^G(A, B\hat\otimes \Cl_{2n})$$
and 
\begin{align*}
&KK^G(A\hat\otimes \Cl_{2n+1}, B)\cong KK^G(A\hat\otimes \Cl_1, B)\\
&\quad\quad\cong 
KK^G(A, B\hat\otimes \Cl_1)\cong KK^G(A, B\hat\otimes \Cl_{2n+1}).
\end{align*}
As a consequence, we have $KK_l^G(A,B)\cong KK_{l+2}^G(A,B)$ for all $l\in \ZZ$.
\end{proposition}
\begin{proof}  In the even case, the isomorphisms follow by taking Kasparov products with the  $KK$-equivalences
$1_A\otimes x_{2n}^{-1}\in KK^G(A,A\hat\otimes \Cl_{2n})$ and $1_B\otimes x_{2n}^{-1}\in KK^G(B, B\hat\otimes \Cl_{2n})$
from the left and right, respectively (where we consider the trivial action of $G$ on the Clifford algebras).
The same argument will provide isomorphisms 
$KK^G(A\hat\otimes \Cl_{2n+1}, B)\cong KK^G(A\hat\otimes \Cl_1, B)$ and 
$KK^G(A, B\hat\otimes \Cl_1)\cong KK^G(A, B\hat\otimes \Cl_{2n+1})$, respectively. 

To  finish off, one checks that the composition
{\small $$KK^G(A\hat\otimes \Cl_1, B) \stackrel{y\mapsto y\otimes 1_{\Cl_1}}{\longrightarrow} KK^G(A\hat\otimes \Cl_1\hat\otimes \Cl_1, B\hat\otimes \Cl_1) 
\stackrel{(1_A\otimes x_2^{-1})\otimes\cdot}{\cong} KK^G(A, B\hat\otimes \Cl_1)$$}
is an isomorphism with inverse given by the composition
{\small $$KK^G(A, B\hat\otimes \Cl_1)\stackrel{y\mapsto y\otimes 1_{\Cl_1}}{\longrightarrow} KK^G(A\hat\otimes \Cl_1, B\hat\otimes \Cl_1\hat\otimes \Cl_1)
\stackrel{\cdot\otimes (1_B\otimes x_2)}{\cong } KK^G(A\hat\otimes \Cl_1, B),$$}
where we use that $\Cl_1\hat\otimes\Cl_1\cong \Cl_2\sim_M\CC$. 
\end{proof}

Of course we would like to have a version of Bott-periodicity  showing that, alternatively, we could define the 
higher $KK$-groups via suspension. For this we are going to construct a $KK$-equivalence between
$\Cl_1$ and $C_0(\RR)$. Indeed, we shall do this by first constructing a $KK$-equivalence 
between $C_0(\RR)\hat\otimes \Cl_1\cong C_0(\RR,\Cl_1)$ with $\CC$. Since we consider the trivial $G$-action on $C_0(\RR)$ and 
$\Cl_1$ it suffices to do this for the trivial group $G=\{e\}$.
In what follows next we write $\mathcal D:=C_0(\RR)\hat\otimes \Cl_1$ and $\D_0:=C_c^{\infty}(\RR)\hat\otimes \Cl_1\subseteq \D$,
where $C_c^{\infty}(\RR)$ denotes the dense subalgebra of $C_0(\RR)$ consisting of smooth functions with compact supports.
A typical element of $\mathcal D$ can be written as $f_1+f_2 e$ with $f_1,f_2\in C_0(\RR)$, where we identify $f_1$ with $f_1 1_{\Cl_1}$ and where
$e=e_1$ denotes the generator of $\Cl_1$ with $e^2=1$.

{\bf The Dirac element:} We define an element 
$$\alpha=[\mathcal H, \Phi, T]\in KK(\mathcal D, \CC)$$
as follows: We let $\mathcal H= L^2(\RR)\oplus L^2(\RR)$ be equipped with the grading induced by the 
operator $J=\bmtr 1&0\\0&-1\emtr$. We define $T=\frac{D}{\sqrt{1+D^2}}$ with 
$D=\bmtr 0 & -d\\d&0\emtr$, where 
$d:L^2(\RR)\to L^2(\RR)$ denotes the densely defined operator $d=\frac{d}{dt}$.
Then $D$ is an essentially selfadjoint operator on the dense subspace 
$C_c^{\infty}(\RR)\oplus C_c^{\infty}(\RR)$ of $\mathcal H$ and therefore extends to a densely defined 
selfadjoint operator on $\mathcal H$ (we refer to \cite{HK}*{Chapter 10} for details).
Let $\Phi: \D\to \L(\mathcal H)$ be given  by
\begin{equation}\label{eq-Phi}
\Phi{(f_1+f_2e)}\cdot\bmtr \xi_1\\ \xi_2\emtr=\bmtr f_1 &f_2\\ f_2& f_1\emtr\cdot \bmtr \xi_1\\\xi_2\emtr=
\bmtr f_1\cdot\xi_1+ f_2\cdot\xi_2\\ f_2\cdot\xi_1+f_1\cdot\xi_2\emtr.
\end{equation}
In order to check that $(\mathcal H, \Phi, T)$ is a $\D$-$\CC$ Kasparov cycle
we need to check the conditions of Lemma \ref{lem-unbounded} for the triple $(\H, \Phi, D)$.
Notice that $D^2=\diag(\Delta,\Delta)$, where 
$\Delta=-\frac{d^2}{dt^2}$ denotes the (positive) Laplace operator on $\RR$.
By \cite{RS}*{XIII.4 Example 6} (or \cite{HK}*{10.5.1}) the operator $(1+\Delta)^{-1} M(f)$ is a compact operator
for all $f\in C_c^{\infty}(\RR)$ (where $M:C_b(\RR)\to\B(L^2(\RR))$ denotes the representation as multiplication operators). Hence
$$(1+D^2)^{-1}\circ \Phi(f_1+f_2e)=\bmtr (1+\Delta)^{-1} M(f_1) & (1+\Delta)^{-1} M(f_2)\\ (1+\Delta)^{-1} M(f_2)& (1+\Delta)^{-1} M(f_1) \emtr\in \K(\H)$$
 for all $f_1+f_2e\in \D_0=C_c^\infty(\RR)\oplus C_c^\infty(\RR)$. Since $\D_0$ is dense in $\D$
 and since $(1+D^2)^{-1}\Phi(f_1+f_2e)$ depends continuously on $(f_1,f_2)$, this proves condition (i) of
  Lemma \ref{lem-unbounded}. 
  To see condition (ii) we first observe that for all 
  $f\in C_c^{\infty}(\RR)$ the operator $[d,M(f)]$ is defined for all $\xi\in C_c^\infty(\RR)\subseteq L^2(\RR)$  and 
  we have  $$[d, M(f)]\xi=\frac{d}{dt}(f\xi)-f\cdot (\frac{d}{dt}\xi)=(\frac{d}{dt}f)\cdot \xi,$$
Hence $[d, M(f)]$ extends to a bounded operator on $L^2(\RR)$ and
  $$[D,\Phi(f_1+f_2e)]=\bmtr -[d, M(f_2)] & -[d, M(f_1)]\\ [d, M(f_1)] & [d, M(f_2)]\emtr$$
is densely defined and bounded for all $f_1+f_2e\in \D_0$.  
\\
\\
{\bf The dual-Dirac element:} Choose any odd  continuous function $\varphi:\RR\to [-1,1]$ such that 
 $\varphi(x)>0$ for $x>0$ and  $\lim_{x\to\infty} \varphi(x)=1$.
For instance we could take 
$$\varphi(x)=\frac{x}{\sqrt{1+x^2}}\quad\text{or}\quad\varphi(x)=\left\{\begin{matrix} \sin(x/2) & |x|\leq \pi\\ \frac{x}{|x|} & |x|\geq \pi
\end{matrix}\right\}.$$

We then define an element
$\beta=[\D, \one, S]\in KK(\CC, \D)$  as follows: We consider $\D$ as
 graded Hilbert $\D$-module in the canonical way, and we put $\one(\lambda)a=\lambda a$ 
for all $\lambda \in \CC$ and $a\in \D$.  The operator $S\in \M(\D)$ is defined via multiplication with the element
$\varphi e\in C_b(\RR)\hat\otimes\Cl_1\subseteq \M(\D)$.
To check that $(\D, \Phi, S)$ is  a Kasparov cycle it suffices to check that
$S^2-1 \in \K(\D)=\D$. 
But this follows from the fact that $S^2-1$ is given by pointwise multiplication with the function 
$\varphi^2-1$ which lies in $\D$ since $\lim_{\pm x\to\infty} \varphi^2(x)=1$ by conditions (i) and (ii) for  $\varphi$.
Since $S=S^*$ all other conditions of Definition \ref{def-KK-cycle} are trivial.

Note that the class $\beta$ does not depend on the particular choice of the function $\varphi:\RR\to \RR$. 
Indeed, if two functions $\varphi_0, \varphi_1$ are given which satisfy conditions (i) and (ii), we can define for each
 $t\in [0,1]$ a function 
$\varphi_t:\RR\to\RR$ by 
$$\varphi_t(x)=t\varphi_1(x)+(1-t)\varphi_0(x).$$
Then each $\varphi_t$ satisfies the requirements (i) and (ii) and if $S_t$ denotes the corresponding operators it follows
that $t\mapsto S_t$ is an operator homotopy joining $S_0$ with $S_1$. 

\begin{notation}\label{def-Bott}
The  element $\beta\in KK(\CC, C_0(\RR)\hat\otimes \Cl_1)=KK_1(\CC, C_0(\RR))$ constructed above
is called the {\em Bott class}. 
\end{notation}

{\bf The Kasparov product $\beta\otimes_\D\alpha\in KK(\CC,\CC)$.} We are now going to show that 
$\beta\otimes_\D\alpha=1_\CC\in KK(\CC,\CC)$. For this we first claim that $\beta\otimes_\D\alpha$ is represented 
by the triple
$(\H, {\mathbf 1}, T')$ with 
\begin{equation}\label{eqR}
T'=\bmtr 0& M(\varphi)\\ M(\varphi)&0\emtr - M(\sqrt{1-\varphi^2}) \bmtr 0&-\frac{d}{\sqrt{1+\Delta}}\\ \frac{d}{\sqrt{1+\Delta}}& 0\emtr,
\end{equation}
with $\H=L^2(\RR)\oplus L^2(\RR)$ as above. Indeed, since $\Phi: \D\to\L(\H)$ is a non-degenerate representation
we obtain an isomorphism
\begin{equation}\label{eq-iso}\D\otimes_{\D} \H\cong \H; a\otimes \xi\mapsto \Phi(a)\xi.
\end{equation}
Let $T_{12}$ denote the operator on $\D\otimes_\D\H$ corresponding to 
$T=\bmtr 0&-\frac{d}{\sqrt{1+\Delta}}\\ \frac{d}{\sqrt{1+\Delta}}& 0\emtr$ under this isomorphism.
We claim that $T_{12}$ is a $T$-connection. Recall that for any $a\in \D$ the operator
$\theta_a: \H\to\D \otimes_\D\H$ is given by $\xi\mapsto a\otimes \xi$. Composed with the above isomorphism 
we get the operator $\xi\mapsto \Phi(a)\xi$ on $\H$.
Condition (\ref{eq-connect})  follows then 
for $T_{12}$ from the fact that $[T,\Phi(a)]\in \K(\H)$ for all $a\in \D$.
We now use Proposition \ref{prop-product} to see that the Kasparov product $\beta\otimes_\D\alpha$ is represented 
by the triple
$$( \D\otimes_\D\H, \one\otimes 1, S\otimes 1 + (\sqrt{1-S^2}\otimes 1)T_{12}).$$
We leave it as an exercise to check that this operator corresponds to the operator $R$ of (\ref{eqR}) under 
the isomorphism (\ref{eq-iso}).

Hence to see that $\beta\otimes_\D\alpha=1_\CC\in KK(\CC,\CC)$ we only need to show that the Fredholm index 
of the operator 
$$F:=M(\varphi)+M\left(\sqrt{1-\varphi^2}\right)\frac{d}{\sqrt{1+\Delta}}:L^2(\RR)\to L^2(\RR)$$
is one (see the discussion at the end of Section \ref{sec-KKgroup}).

Recall that $\varphi:\RR\to\RR$ can be any function satisfying the conditions (i) and (ii) as stated in the 
construction of $\beta$. Thus we may choose
$$\varphi(x)=\left\{\begin{matrix} \sin(x/2) &\text{if $x\in [-\pi, \pi]$}\\ \frac{x}{|x|} & \text{if  $|x|>\pi$}\end{matrix}\right\}.$$
To do the computation we want to restrict the operator to the interval $[-\pi,\pi]$. For this consider the 
orthogonal projection $Q: L^2(\RR)\to L^2[-\pi,\pi]$.  
Since $Q$ commutes with $M(1-\varphi^2)$  and since $[\frac{d}{\sqrt{1+\Delta}}, M(\psi)]\in \K(L^2(\RR))$ for all $\psi\in C_0(\RR)$
(which follows from $[T,\Phi(a)]\in \K(\H)$ for all $a\in \D$), 
we have 
\begin{align*}
&M(1-\varphi^2)\frac{d}{\sqrt{1+\Delta}}\sim M(\sqrt{1-\varphi^2})\frac{d}{\sqrt{1+\Delta}}M(\sqrt{1-\varphi^2})\\
&=M(\sqrt{1-\varphi^2})Q\frac{d}{\sqrt{1+\Delta}}M(\sqrt{1-\varphi^2})Q\sim M(1-\varphi^2)\frac{d}{\sqrt{1+\Delta}}Q,
\end{align*}
where $\sim$ denotes equality up to compact operators. Thus we may replace $F$ by the operator 
$$F_1:=M(\varphi)+M\left(\sqrt{1-\varphi^2}\right)\frac{d}{\sqrt{1+\Delta}}Q:L^2(\RR)\to L^2(\RR)$$
Decomposing $L^2(\RR)$ as the direct sum $L^2[-\pi,\pi]\oplus L^2((-\infty,-\pi)\cup(\pi, \infty))$, we see that the 
operator $F_1$ fixes both summands and acts as the identity on the second summand. 
Hence for computing the index we may restrict our operator  
to the summand $L^2[-\pi,\pi]$ on which it acts by 
$$F_2:=M(\sin(x/2))+M(\cos(x/2))\frac{d}{\sqrt{1+\Delta}}.$$

 Now there comes a slightly tricky point and we need to appeal to some computations given in \cite{HR}*{Chapter 10}.
We want to replace the operator  $\frac{d}{\sqrt{1+\Delta}}$ by the operator  $\frac{\tilde{d}}{\sqrt{1-\Delta^\TT}}: L^2(\TT)\to L^2(\TT)$ 
(identifying $L^2[-\pi,\pi]$ with $L^2(\TT)$),
where $\tilde{d}:L^2(\TT)\to L^2(\TT)$ denotes the  operator
 given by the differential $\frac{d}{dt}$ on the smooth $2\pi$-periodic functions on $\RR$  and  where $\Delta^\TT=-\tilde{d}^2$ denotes 
 the corresponding Laplace operator. 
 Although the operators $d$ and $\tilde{d}$ clearly coincide on 
 $C_c^{\infty}(-\pi,\pi)$ the functional calculus which has been applied for producing the operators 
 $\frac{d}{\sqrt{1+\Delta}}$ and $\frac{\tilde{d}}{\sqrt{1-\Delta^\TT}}$ depends on the full domains of the selfadjoint extensions of 
 these operators, which clearly differ. The solution of this problem is implicitly given in \cite{HR}*{Lemma 10.8.4}:
 Recall that $\frac{d}{\sqrt{1+\Delta}}$ is equal to $-i\chi(id)$ where we apply the functional calculus for unbounded 
 selfadjoint operators for the function $\chi(x)=\frac{x}{\sqrt{1+x^2}}$ to the (unique) selfadjoint extension of $id$.
 Let $\psi\in C_c^\infty(-\pi,\pi)$ be any fixed function. Choose a positive function $\mu\in C_c^\infty(\RR)$ with $\supp\mu\in (-\pi,\pi)$ 
 such that $\mu\equiv 1$ on $U:=\supp\psi+(-\delta,\delta)$ for some suitable $\delta>0$. Let $d_\mu= M(\mu)\circ  d \circ M(\mu)$.
 It follows then from \cite{HR}*{Corollary 10.2.6} that $id_\mu$ is an essentially selfadjoint operator which coincides with 
 $id$ on $U$. It then follows from \cite{HR}*{Lemma 10.8.4} that 
 $M(\psi)\chi(id)\sim M(\psi)\chi( i d_\mu)$ on $L^2[-\pi,\pi]$,
 where, as above,  $\sim$ denotes equality up to compact operators. Applying the same argument to the canonical inclusion 
 of the interval $(-\pi,\pi)$ into $\TT$ shows that 
 $M(\psi)\chi(i d_\mu) \sim M(\psi)\chi(i\tilde{d})$.
 Together we see that $M(\psi)\chi(id)\sim M(\psi)\chi(i\tilde{d})$ on $L^2[-\pi,\pi]$ for all $\psi\in C_c^\infty(-\pi,\pi)$ and then also for all 
 $\psi\in C_0(-\pi,\pi)$. Applying this to $\psi(x)=\cos(x/2)$ gives the desired result.

 Thus we may  replace the operator $F_2$ by the operator
$$F_3:= M(\sin(x/2))+M(\cos(x/2))\frac{\tilde{d}}{\sqrt{1+\Delta^\TT}}.$$
Multiplying $F_3$  from the left with the invertible operator $M(2i e^{i\frac{x}{2}})$ does not change the Fredholm index, so we 
compute the index of the operator
\begin{align*}
F_4&= M(2ie^{i\frac{x}{2}}\sin(x/2))+M(2ie^{i\frac{x}{2}}\cos(x/2))\frac{\tilde{d}}{\sqrt{1+\Delta^\TT}}\\
&= M(e^{ix}-1)+ iM(e^{ix}+1)\frac{\tilde{d}}{\sqrt{1+\Delta^\TT}}.
\end{align*}
In what follows let $\{e_n:n\in \ZZ\}$ denote the standard othonormal basis of $\ell^2(\ZZ)$ and let $U:\ell^2(\ZZ)\to\ell^2(\ZZ)$ 
denote the bilateral shift operator $U(e_n)=e_{n+1}$. 
Using Fourier transform and the Plancherel isomorphism  $L^2[-\pi,\pi]\cong \ell^2(\ZZ)$ the operator $F_4$ 
transforms to the operator $\widehat{F_4}:\ell^2(\ZZ)\to\ell^2(\ZZ)$ given by
$$\widehat{F_4}=(U-\id)+i(U+\id)R$$
where $R:\ell^2(\ZZ)\to\ell^2(\ZZ)$ is given by $R(e_n)=\frac{in}{\sqrt{1+n^2}}e_n$.
 Let $\sign(n)=\left\{\begin{matrix} \frac{n}{|n|}& \text{if $n\neq 0$}\\0& \text{if $n=0$}\end{matrix}\right\}$ and let $R'(e_n)=i\sign(n)e_n$ for $n\in \ZZ$
 and let us write $\widehat{F_5}:=(U-\id)+i(U+\id)R'$.
Since $\left| \frac{in}{\sqrt{1+n^2}}-i\sign(n)\right|\to 0$ for $|n|\to\infty$ we have $R-R'\in \K(\ell^2(\ZZ))$, which implies that 
$\widehat{F_5}-\widehat{F_4}\in \K(\ell^2(\ZZ))$ and hence
$\indx(\widehat{F_5})=\indx(\widehat{F_4})$.
Applying $\widehat{F_5}$ to some basis element $e_n$ gives
\begin{align*}
\widehat{F_5}(e_n)&=(U-\id)+i(U+\id)R'(e_n)\\
&=(e_{n+1}-e_n)-\sign(n)(e_{n+1}+e_n)
=\left\{\begin{matrix} 2 e_{n+1}& \text{if $n<0$}\\ e_1-e_0&\text{if $n=0$}\\
-2 e_n&\text{if $n>0$}\end{matrix}\right\}.
\end{align*}
It follows from this that $\widehat{F_5}$ is surjective and  $\ker(\widehat{F_5})=\CC(e_1+2e_0+e_{-1})$. 
Hence 
$$\indx(\beta\otimes_\D\alpha)=\indx(\widehat{F_5})=1.$$
Let us state  as a lemma what we have proved so far:

\begin{lemma}\label{lem-product}
Let $\D=C_0(\RR)\hat\otimes \Cl_1=C_0(\RR)\oplus C_0(\RR)$ with the standard odd grading and let
$\alpha\in KK(\D,\CC)$ and $\beta\in KK(\CC,\D)$ as above. Then 
$$\beta\otimes_\D\alpha=1_\CC\in KK(\CC,\CC).$$
\end{lemma}

In order to show that $\alpha$ and $\beta$ are $KK$-equivalences, we also need to check that 
the product $\alpha\otimes_\CC\beta=1_\D\in KK(\D,\D)$. For this we shall use a rotation trick which originally
 goes back to Atiyah, and which has been adapted very successfully to this situation by Kasparov.
Recall that by Theorem \ref{gen-product} the Kasparov product over $\CC$ is commutative, i.e., 
we have
\begin{align*}
\alpha\otimes_\CC\beta=\beta\otimes_\CC\alpha&= (\beta\otimes 1_\D)\otimes_{\D\hat\otimes \D} (1_\D\otimes \alpha)\\
&=(\beta\otimes 1_\D)\otimes_{\D\hat\otimes \D} \big(\Sigma_{\D,\D}\otimes_{\D\hat\otimes \D} (\alpha\otimes 1_\D)\big),
\end{align*}
where $\Sigma_{\D,\D}:\D\hat\otimes \D\to \D\hat\otimes \D$ denotes (the $KK$-class of) the flip homomorphism $x\otimes y\mapsto (-1)^{\deg(x)\deg(y)}y\otimes x$.
If we can show that there is an invertible  class $\eta\in KK(\D,\D)$ such that $\Sigma_{\D,\D}=1_\D\hat\otimes\eta\in KK(\D\hat\otimes \D, \D\hat\otimes \D)$ 
the result will follow from the following reasoning:
\begin{align*}
\alpha\otimes_\CC\beta&=(\beta\otimes 1_\D)\otimes_{\D\hat\otimes \D} \big(\Sigma_{\D,\D}\otimes_{\D\hat\otimes \D} (\alpha\otimes 1_\D)\big)\\
&=(\beta\otimes 1_\D)\otimes_{\D\hat\otimes \D} \big((1_\D\otimes \eta)\otimes_{\D\hat\otimes \D} (\alpha\otimes 1_\D)\big)\\
&=(\beta\otimes 1_\D)\otimes_{\D\hat\otimes \D}\big(\Sigma_{\D,\D} \circ (\eta\otimes 1_\D)\circ \Sigma_{\D,\D}^{-1}\big)\otimes_{\D\hat\otimes \D} \big(\Sigma_{\D,\D}\otimes_{\D\hat\otimes \D} (1_\D\otimes \alpha)\big)\\
&=(\beta\otimes 1_\D)\otimes_{\D\hat\otimes \D} \big(\Sigma_{\D,\D} \circ (\eta\otimes 1_\D)\big)\otimes_{\D\hat\otimes \D} (1_\D\otimes\alpha )\\
&=(\beta\otimes 1_\D)\otimes_{\D\hat\otimes \D} \Sigma_{\D,\D} \otimes_{\D\hat\otimes \D} (\eta\otimes_\CC\alpha)\\
&=(\beta\otimes 1_\D)\otimes_{\D\hat\otimes \D} (\alpha\otimes_\CC\eta)\\
&=\big((\beta\otimes 1_\D)\otimes_{\D\hat\otimes \D}(\alpha\otimes 1_\D)\big)\otimes_{\CC\hat\otimes \D} (1_\CC\otimes \eta)\\
&=\eta.
\end{align*}
Since $\eta$ is invertible in $KK(\D,\D)$ this implies that $\alpha$ has a right $KK$-inverse 
$\gamma$, say.
But then $\gamma=\beta$, since
$$\beta=\beta\otimes_\D(\alpha\otimes_\CC\gamma)=(\beta\otimes_\D\alpha)\otimes_\CC \gamma=\gamma.$$

To see that  $\Sigma_{\D,\D}=1_\D\hat\otimes\eta\in KK(\D\hat\otimes \D, \D\hat\otimes \D)$  for a suitable invertible $KK$-class $\eta$ 
 we consider the isomorphism 
$$\D\hat\otimes \D=\big(C_0(\RR)\hat\otimes  \Cl_1\big)\hat\otimes \big(C_0(\RR)\hat\otimes  \Cl_1\big)\cong
C_0(\RR^2)\hat\otimes \Cl_2.$$
If $\tau: \RR^n\to\RR^n$ is any orthogonal transformation, it induces an automorphism $\tau^*:C_0(\RR^n)\to C_0(\RR^n)$ 
by $\tau^*(f)(x)=f(\tau^{-1}(x))$ and an automorphism $\tilde\tau$ of $\Cl_n$ by extending, via the universal property of $\Cl_n$, 
the map $\tilde\tau: \RR^n\to \Cl_n; v\mapsto \iota\circ \tau(v)$ to all of $\Cl_n$, where $\iota:\RR^n\to \Cl_n$ denotes the canonical inclusion.
We then get an automorphism 
$$\Phi_\tau:= \tau^*\hat\otimes \tilde\tau: C_0(\RR^n)\hat\otimes \Cl_n\to C_0(\RR^n)\hat\otimes \Cl_n.$$
Moreover, a homotopy of orthogonal transformations of $\RR^n$ between $\tau_0$ and $\tau_1$ clearly 
 induces a homotopy between the automorphisms $\Phi_{\tau_0}$ and $\Phi_{\tau_1}$. In particular, for any 
 orthogonal transformation which is homotopic to $\id_{\RR^n}$
 we get 
 $$[\Phi_{\tau}]=1_{C_0(\RR^n,\Cl_n)}\in KK(C_0(\RR^n,\Cl_n),C_0(\RR^n,\Cl_n)).$$
It is not difficult to check that under the isomorphism $\D\hat\otimes \D\cong C_0(\RR^2)\hat\otimes \Cl_2$ 
the  flip automorphism $\Sigma_{\D,\D}$  corresponds to  $\Phi_\sigma:C_0(\RR^2)\hat\otimes \Cl_2\to C_0(\RR^2)\hat\otimes \Cl_2$  with 
$$\sigma:\RR^2\to \RR^2; \sigma(x,y)=(y,x).$$
Since $\det(\sigma)=-1$, it is, unfortunately, not homotopic to $\id_{\RR^2}$. 
But the orthogonal transformation $\rho:\RR^2\to \RR^2; \rho(x,y)=(-y, x)$ is homotopic to $\id_{\RR^2}$ via 
the path of transformations $\rho_t$, $t\in [0, \pi/2]$ with
$$\rho_t(x,y)=(\cos(t) x-\sin(t)y, \sin(t)x+\cos(t)y).$$
One checks that $\Phi_\rho$ corresponds to $\Sigma_{\D,\D}\circ (\id_\D\otimes \Phi_{-\id}),$
where $-\id: \RR\to \RR, x\mapsto -x$ is the flip on $\RR$. Hence, if $\eta=[\Phi_{-\id}]\in KK(\D,\D)$, we have 
$$\Sigma_{\D,\D}\otimes_{\D\hat\otimes \D} (1_\D\otimes \eta)=[\Sigma_{\D,\D}\circ (\id_\D\otimes \Phi_{-\id})]=[\Phi_\rho]=1_{\D\hat\otimes \D},$$
where, by abuse of notation, we identify $\Phi_\rho$ with the corresponding automorphism of $\D\hat\otimes \D$. 
Since $\Sigma_{\D,\D}=\Sigma_{\D,\D}^{-1}$ it  follows that 
$1_\D\otimes\eta=\Sigma_{\D,\D}\in KK(\D\hat\otimes \D, \D\hat\otimes \D)$ and we are done.

\begin{corollary}\label{cor-suspension}
Let $\alpha_1\in KK(C_0(\RR), \Cl_1)$ and $\beta_1\in KK(\Cl_1, C_0(\RR))$ be  the images 
of $\alpha$ and $\beta$ under the isomorphisms $KK(C_0(\RR)\otimes \Cl_1,\CC)\cong KK(\C_0(\RR),\Cl_1)$ and 
$KK(\CC, C_0(\RR)\otimes \Cl_1)\cong KK(\Cl_1, C_0(\RR))$ of Proposition \ref{prop-Bott}.  Then 
$\alpha_1$ is a $KK$-equivalence with inverse $\beta_1$. As a consequence, for all $G$-algebras $A$ and $B$, there are 
canonical Bott-isomorphisms
$$KK^G(A\hat\otimes C_0(\RR), B)\cong KK^G_1(A,B)\cong KK^G(A, B\hat\otimes C_0(\RR)).$$
More generally, for all $n,m\in \NN_0$ we get Bott-isomorphisms
$$KK^G(A\hat\otimes C_0(\RR^n), B\hat\otimes C_0(\RR^m))\cong KK_{n+m}^G(A,B).$$
\end{corollary}
\begin{proof} The proof is a straightforward consequence of the results and techniques explained above and is left to the reader.
\footnote{The above proof of Bott-periodicity follows in part some unpublished notes of Walter Paravicini. See
{\text  http://wwwmath.uni-muenster.de/u/echters/Focused-Semester/downloads.html}.}
\end{proof}

Kasparov actually proved a  more general version of the above $KK$-theoretic Bott-periodicity theorem,
which provides a $KK^G$-equivalence between $C_0(V)$ and the Clifford algebra $\Cl(V, \langle\cdot,\cdot\rangle)$, 
in which $G$ is a (locally) compact group which acts by a continuous  orthogonal representation $\rho:G\to \Ort(V)$ on the 
finite dimensional euclidean vector space $V$, and $\langle\cdot,\cdot\rangle$ is any $G$-invariant inner product on $V$.
The action of $G$ on $\Cl(V)$ is the unique action which extends the given action of $G$ on $V\subseteq \Cl(V)$.
Identifying $V$ with $\RR^n$ equipped with the standard inner product, this $KK$-equivalence is constructed as in the 
above special case where $n=1$ via a $KK^G$-equivalence between $C_0(\RR^n)\hat\otimes\Cl_n$ and $\CC$. 
In case of trivial actions, this result follows from the  case $n=1$, using the fact that the $n$-fold graded  tensor product 
of $\C_0(\RR)\hat\otimes\Cl_1$ with itself is isomorphic to $C_0(\RR^n)\hat\otimes \Cl_n$. We refer to 
Kasparov's original papers \cites{Kas1, Kas2} for the  proof of the general case. 

\subsection{Excision in $KK$-theory.}\label{sec-six-term}  Recall that every short exact sequence
$0\to I\stackrel{\iota}{\to} A\stackrel{q}{\to} A/I\to 0$
of $C^*$-algebras induces a six-term exact sequence in $K$-theory
$$
\begin{CD}
K_0(I) @>\iota_0 >> K_0(A)   @> q_0 >> K_0(A/I)\\
@A\delta AA    @.     @VV\exp V\\
K_1(A/I)  @<<\iota_1 < K_1(A)   @<< q_1 < K_1(I)
\end{CD}
$$
which happens to be extremely helpful for the computation of $K$-theory groups. To some extend we get 
similar six-term sequences in $KK$-theory, but one has to impose some extra conditions on the 
short exact sequence:

\begin{definition}\label{def-semi-split}
Let $G$ be a locally compact group. A short exact sequence of graded $G$-$C^*$-algebras
$$0\to I\stackrel{\iota}{\to} A\stackrel{q}{\to} A/I\to 0$$
is called {\em $G$-equivariantly semisplit} if there exists a 
$G$-equivariant completely positive, normdecreasing, grading preserving  cross setion $\phi:A/I\to A$ for the quotient map
$q:A\to A/I$. We then also say that $A$ is a {\em $G$-semisplit extension} of $A/I$ by $I$. 
\end{definition}
By an important result of Choi-Effros \cite{Choi-Effros}
a short exact sequence $0\to I\stackrel{\iota}{\to} A\stackrel{q}{\to} A/I\to 0$ (with trivial $G$-action) is always semisplit if  $A$ is nuclear.
But there are many other important cases of semisplit extensions. 

Every $G$-semisplit extension determines a unique class in $KK_1^G(A/I, I)$ which plays an important r\^ole 
in the construction of the six-term exact sequences in $KK$-theory. 
The non-equivariant version is well documented (e.g., see \cites{Kas2, Sk-exact-KK, Cu-Sk} and  \cite{Bla86}*{Section 19.5}).
But the details of the equivariant version, which we shall need below, are somewhat scattered in the literature.
The main ingredients are explained in \cite{BS1}*{Remarques 7.5} (see also the proof of \cite{CE2}*{Lemma 5.17}).
We summarise the important steps as follows:
\begin{enumerate}
\item If $0\to I\stackrel{\iota}{\to} A\stackrel{q}{\to} A/I\to 0$ is a $G$-equivariant semisplit extension, then the 
canonical embedding $e: I\to C_q:=C_0([0,1),A)/C_0((0,1), I)$, which sends $a\in I$ to the equivalence class of $(1-t)a$ in $C_q$,
determines a $KK^G$-equivalence $[e]\in KK^G(I,C_q)$.
\item View the Bott-class $\beta\in KK_1(\CC, C_0(0,1))$ as an element in $KK_1^G(\CC, C_0(0,1))$ with trivial 
$G$-actions everywhere and let $i: C_0((0,1),A/I)\to C_q$ denote the canonical map.  
Let  $c\in KK_1^G(A/I, I)$ be the class defined via the Kasparov product
$$ c= (\beta\otimes 1_{A/I})\otimes_{C_0((0,1),A/I)}[i]\otimes_{C_q} [e]^{-1},$$
where $[e]^{-1}\in KK^G(C_q,I)$ denotes the inverse of $[e]$. 
We call $c\in KK^G_1(A/I,I)$ the class attached to the equivariant semisplit extension $0\to I\stackrel{\iota}{\to} A\stackrel{q}{\to} A/I\to 0$.
\end{enumerate}
We then have the following theorem, which can be proved basically along the lines 
of the non-equivariant case using $G$-equivariant versions of Stinespring's theorem and of Kasparov's stabilisation theorem (\cite{Mingo-Phillips}).

\begin{theorem}\label{thm-sixterm}
Suppose that $0\to I\stackrel{\iota}{\to} A\stackrel{q}{\to} A/I\to 0$ is a $G$-equivariant semisplit short exact sequence
of $C^*$-algebras. Then for every $G$-$C^*$-algebra $B$, we have the following two six-term exact sequences:
$$\begin{CD}
KK_0^G(B,I) @>\iota_* >> KK_0^G(B,A) @>q_* >> KK_0^G(B,A/I)\\
@A \partial AA   @.    @VV \partial V\\
KK_1^G(B,A/I) @<< q_* <     KK_1^G(B,A) @<< \iota_* < KK_1^G(B,I)\\
\end{CD}
$$
and
$$\begin{CD}
KK_0^G(A/I, B) @ >q^*>> KK_0^G(A,B) @>\iota^* >> KK_0^G(I,B)\\
@A \partial AA   @.    @VV \partial V\\
KK_1^G(I, B) @<< \iota^* <     KK_1^G(A,B) @<< q^* < KK_1^G(A/I,B),\\
\end{CD}
$$
where the boundary maps are all given by taking Kasparov product with the class $c\in KK_1^G(A/I,I)$ 
of the given extension.
\end{theorem}

\section{The Baum-Connes conjecture} \label{sec-BC}
\subsection{The universal proper $G$-space}

In what follows, for a locally  compact group $G$, a {\em $G$-space} will  mean 
a locally compact space $X$ together with a homomorphism $h: G\to \Homeo(X)$ such that 
the map 
$$G\times X\to X; (s,x)\mapsto s\cdot x:=h(s)(x)$$
is continuous.  A $G$-space $X$ is called {\em proper}, if the map
$$\varphi:G\times X\to X\times X; (s,x)\mapsto (s\cdot x,  x)$$
is proper in the sense that inverse images of compact sets are compact. 
Equivalently, $X$ is a proper $G$-space, if every net $(s_i, x_i)$ in $G\times X$ 
such that $(s_i\cdot x_i, x_i)\to (y,x)$ for some  $(y,x)\in X\times X$ has a convergent subnet.
We  also say that $G$ {\em acts properly} on $X$. 
Proper $G$-spaces have an extremely nice behaviour and they are very closely connected 
to actions by compact groups. 

Let us state some important properties:

\begin{lemma}\label{lem-proper-prop}
Suppose that $X$ is a proper $G$-space. Then the following hold:
\begin{enumerate}
\item For every $x\in X$ the stabiliser $G_x=\{s\in G: s\cdot x=x\}$ is compact.
\item The orbit space $G\backslash X$ equipped with the quotient topology is 
a locally compact Hausdorff space.
\item If $X$ is a $G$-space, $Y$ is a proper $G$-space and $\phi:X\to Y$ is a $G$-equivariant continuous 
map, then $X$ is a proper $G$-space as well.
\end{enumerate}
\end{lemma}
\begin{proof} The first assertion follows from 
$G_x\times\{x\}=\varphi^{-1}(\{(x,x)\})$, if $\varphi:G\times X\to X\times X$ is the structure map.

For the second assertion we first observe that the quotient map $q:X\to G\backslash X$ is open
since for any open subset $U\subseteq X$ we have $q^{-1}(q(U))=G\cdot U$ is open in $X$. This then 
easily implies that $G\backslash X$ is locally quasi-compact. We need to 
show that   $G\backslash X$ is Hausdorff.
For this assume that there is net $(x_i)$ such that the net of orbits $(G(x_i))$ 
converges to two orbits $G(x)$,  $G(y)$. We need to show that $y=s\cdot x$ for some $s\in G$. 
Since the quotient map $q:X\to G\backslash X$ is open, we may assume, after passing to a 
subnet if necessary, that $x_i\to x$ and $s_i\cdot x_i\to y$ for some suitable net $(s_i)$ in $G$.
Hence $(s_i\cdot x_i, x_i)\to (y,x)$, and by properness we may assume, after passing to a subnet if necessary,
that $(s_i, x_i)\to (s,x)$ in $G\times X$ for some $s\in G$. 
 But then $s_i\cdot x_i\to s\cdot x$ which implies $y=s\cdot x$.
 
 For the third assertion let $K\subseteq X$ be compact. 
 If $(s\cdot x,x)\in K\times K$ it follows that $\phi\times \phi(s\cdot x, x)\in \phi(K)\times\phi(K)$, hence,
  by properness 
 of $Y$,
 $(s, \phi(x))$ lies in the compact set $\varphi_Y^{-1}(\phi(K)\times\phi(K))$ of $G\times Y$. 
 If $C\subseteq G$ denotes the compact projection of this set in $G$, we see that 
 $\varphi_X^{-1}(K\times K)\subseteq C\times K$ is compact as well.
  \end{proof}

\begin{example}\label{ex-proper}
{\bf (a)}  If $G$ is compact, then every $G$-space $X$ is proper, since 
for all $C\subseteq X$ compact, we have that
$\varphi^{-1}(C\times C)\subseteq G\times C$ is compact.

{\bf (b)} Suppose that $H\subseteq G$ is a closed subgroup of $G$ and assume that 
$Y$ is a proper $H$-space. The induced $G$-space $G\times_HY$ is defined 
as the quotient $H\backslash (G\times Y)$ with respect to the $H$-action $h\cdot (s,y)=(sh^{-1}, h\cdot y)$.
This action is proper by part (iii) of the above lemma, hence $G\times_HY$ is a locally compact Hausdorff space.
We let $G$ act on $G\times_HY$ by $s\cdot[t,y]=[st,y]$. We leave it as an exercise  to check that this action is proper as well.

{\bf (c)}  It follows as a special case of {\bf (b)} that whenever $K\subseteq G$ is a compact subgroup of $G$ and $Y$ is a 
$K$-space, then the induced $G$-space $G\times_KY$ is a proper $G$-space. Indeed, by a theorem of Abels
(see \cite{Ab-univ}*{Theorem 3.3}) every proper $G$-space is locally induced from compact subgroups. More precise, if $X$ is 
a proper $G$-space and $x\in X$, then there exists a $G$-invariant open neighbourhood $U$ of $x$ such that 
$U\cong G\times_KY$ as $G$-space for some compact subgroup $K$ of $G$ (depending on $U$) and 
some $K$-space $Y$. 
In particular, if $G$ does not have any compact subgroup, then every proper $G$-space is a principal $G$-bundle.

{\bf (d)} If $M$ is a finite dimensional manifold, then the action of the fundamental group $G=\pi_1(M)$ on the universal covering 
$\widetilde{M}$ by deck transformations is a (free and) proper action.
\end{example}

\begin{definition}\label{def-proper-uni}
A proper $G$-space $Z$ is called a {\em universal proper $G$-space} if for every proper $G$-space $X$ there 
exists a continuous $G$-map $\phi:X\to Z$ which is unique up to $G$-homotopy. 
We then write $Z=:\EG$. Note that $\EG$ is unique up to $G$-homotopy equivalence.
\end{definition}

The following result is due to Kasparov and Skandalis (see \cite{KS1}*{Lemma 4.1}):

\begin{proposition}\label{prop-EG}
Let $X$ be a proper $G$-space and let $\mathfrak M(X)$ denote the set of finite Radon measures on $X$ 
with total mass in $(\frac{1}{2}, 1]$ equipped  with the weak-* topology as a subset of the dual $C_0(X)'$ of $C_0(X)$
and equipped with the action induced by the action of $G$ on $C_0(X)$. Then $\mathfrak M(X)$ is a universal proper $G$-space.
\end{proposition}

Since the restriction of a proper $G$-action to a closed subgroup $H$ is again proper, it 
follows that the restriction of the $G$-space $\mathfrak M(X)$ of the above proposition to any closed subgroup $H$ is a universal proper 
$H$-space as well. By uniqueness of $\EG$ up to $G$-homotopy we get

\begin{corollary}\label{cor-subgroup-univ}
Suppose that $H$ is a closed subgroup of $G$ and let $Z$ be a universal proper $G$-space.
Then $Z$ is also a universal proper $H$-space if we restrict the given $G$-action to $H$.
\end{corollary}

\begin{example}\label{eq-EG}
{\bf (a)} If $G$ is compact, then the one-point space $\{\pt\}$ with the trivial $G$-action is 
a universal proper $G$-space. Similarly, every contractible space $Z$ with trivial $G$-action 
is universal. Hence $\EG=\{\pt\}$ and $\EG=Z$.

{\bf (b)} We have $\underline{E \RR^n}=\RR^n$: Since $\RR^n$ has no compact subgroups it follows that every proper $G$-space is a principal $\RR^n$-bundle.
On the other hand, since $\RR^n$ is contractible, it follows that every principal $\RR^n$-bundle is trivial.
It follows that every proper $\RR^n$-space $X$ is isomorphic to $\RR^n\times Y$  with trivial action on $Y$ and translation action on $\RR^n$.
Hence the projection $p:X\cong\RR^n\times Y\to\RR^n$ maps $X$ equivariantly into $\RR^n$.
If $\phi_0,\phi_1:X\to \RR^n$ are two such maps, then
$$\phi_t:X\to \RR^n: \phi_t(x)=t\phi_1(x)+(1-t)\phi_0(x)$$
is a $G$-homotopy of equivariant maps between them. Thus $\RR^n$ is universal.

{\bf (c)} It follows from {\bf (b)} together with Corollary \ref{cor-subgroup-univ} that  $\underline{E\ZZ^n}=\RR^n$.

{\bf (d)} If $G$ is a torsion free discrete group, then, as explained above,  every proper $G$-space
is a principal $G$-bundle. It follows that $\EG=EG$, the universal principal $G$-bundle.

{\bf (e)} If $G$ is an almost connected group (i.e., the quotient $G/G_0$ of $G$ by the connected component $G_0$
of the identity in $G$ is compact), then $G$ has a maximal compact subgroup $K\subseteq G$. It is then shown
by Abels in \cite{Ab} that $G/K$ is a universal proper $G$-space. 

{\bf (f)} It follows from {\bf (e)}  and Corollary \ref{cor-subgroup-univ} that for every closed subgroup $H$ of an 
almost connected group $G$, we have $\EH=G/K$, with $K$ a maximal compact subgroup of $G$.
In particular, we have $\underline{E\SL(n,\ZZ)}=\SL(n,\RR)/\SO(n)$.
\end{example}

\subsection{The Baum-Connes assembly map}\label{subsec-assembly}
The {\em Baum-Connes conjecture with coefficients in a $C^*$-algebra $A$} (denoted BCC for short) states, that for every $G$-$C^*$-algebra $A$ a certain
assembly map
$$\mu_{(G,A)}:K_*^G(\EG;A)\to K_*(A\rtimes_{\alpha}G)$$
is an isomorphism of abelian groups. Here $K_*^G(\EG;A)$, often called the {\em topological $K$-theory of $G$ with coefficients in $A$},
 can be regarded as 
the equivariant $K$-homology of $\EG$ with coefficients in $A$. We give a precise definition of this group 
and of the assembly map below. In case $A=\CC$ we get the {\em Baum-Connes conjecture with trivial coefficients} (BC for short),
which relates the equivariant $K$-homology $K_*^G(\EG):=K_*^G(\EG;\CC)$ with $K_*(C_r^*(G))$, the $K$-theory of the reduced 
group algebra of $G$. 

It is well known by work of Higson, Lafforgue and Skandalis (\cite{HLS}) that the now often called {\em Gromov Monster group}
$G$ fails the conjecture with coefficients.
But there is still no counterexample for the conjecture with trivial coefficients. On the other hand we know by work of Higson and Kasparov \cite{HK}  
that (even a very strong version of) BCC holds for all a-$T$-menable groups -- a large class of groups which 
contains all amenable groups.  We give a more detailed discussion of this in Section \ref{sec-Dirac}  below.
The relevance of the Baum-Connes conjecture comes from a number of facts:
\begin{enumerate}
\item It implies many other important conjectures, like the Novikov conjecture in topology, the Kaplansky conjecture on idempotents 
in group algebras, and the Gromov-Lawson conjecture on positiv scalar curvatures in Differential Geometry. 
So the validity of the conjecture for a given group $G$ has many positive consequences. We refer to \cites{BCH, Val89} for
more detailed discussions on these applications.
\item At least in the case of trivial coefficients the left hand side $K_*^G(\EG)$ is computable (at least in principle) by 
classical techniques from algebraic topology like excision, taking direct limits, and such. These methods are usually 
not available for the computation of $K_*(C_r^*(G))$ (or $K_*(A\rtimes_rG)$).
\item As we shall see further down, the conjecture allows a certain flexibility for the coefficients 
in a number of interesting cases, which makes it possible to perform explicit $K$-theory computations 
for certain crossed products and (twisted) group algebras.
\end{enumerate}
Before we go on with this general discussion, we now want to explain the ingredients of the conjecture.
For this we let $G$ be a locally compact group and $X$ a proper $G$-space. Let us further assume that 
$X$ is {\em $G$-compact}, which means that
$G\backslash X$ is compact.
Then there exists a continuous function 
$c:X\to [0,\infty)$ with compact support such that for all $x\in X$ we have
$$\int_G c(s^{-1}\cdot x)^2\, ds=1.$$
For the construction, just choose any compactly supported positive function $\tilde{c}$ on $X$ such 
that for each $x\in X$ there exists $s\in G$ with $\tilde{c}(s\cdot x)\neq 0$, divide this function by the strictly positive 
function $d(x):=\int_G\tilde{c}(s^{-1}\cdot x)\, ds$ and then  put $c=\sqrt{\frac{\tilde{c}}{d}}$.
We shall call such function $c:X\to[0,\infty)$ a {\em cut-off function} for $(X,G)$.  For such $c$
consider the function
$$p_c:G\times X\to [0,\infty); p_c(s,x)=\Delta_G(s)^{-1/2} c(x)c(s^{-1}\cdot x), \quad \forall (s,x)\in G\times X.$$
It follows from the properness of the action and the fact $c$ has compact support 
 that $p_c\in C_c(G\times X)\subseteq C_c(G, C_0(X))$. Thus $p_c$ can be regarded as an element of 
 the reduced (or full) crossed product $C_0(X)\rtimes_rG=\overline{C_c(G,C_0(X))}$. In fact, $p_c$ is a projection 
 in $C_0(X)\rtimes_rG$: For each $(s,x)\in G\times X$ 
 we have
 \begin{align*} 
 p_c*p_c(s,x)&=\int_G p_c(t,x)p_c(t^{-1}s, t^{-1}\cdot x)\, dt\\
 &=\int_G\Delta_G(s)^{-1/2} c(x)c(t^{-1}\cdot x)^2 c(s^{-1}\cdot x)\,dt\\
 &=p_c(s,x) \cdot \int_G c(t^{-1}\cdot x)^2\, dt=p_c(s,x)
 \end{align*}
 and it is trivial to check that $p_c^*=p_c$.
 Thus $p_c$ determines a class $[p_c]\in K_0(C_0(X)\rtimes_rG)=KK_0(\CC, C_0(X)\rtimes_{r}G)$. Note that this class does not depend on the
 particular choice of the cut-off function $c$, for if $\tilde{c}$ is another cut-off function, then 
 $$c_t=\sqrt{tc^2+(1-t)\tilde{c}^2}, \quad t\in [0,1]$$
 is a path of cut-off functions joining $c$ with $\tilde{c}$, and then $p_{c_t}$ is a path of projections joining 
 $p_c$ with $p_{\tilde{c}}$. We call $[p_c]$ the {\em fundamental $K$-theory class of} $C_0(X)\rtimes_{r}G$.

Recall that Kasparov's descent homomorphism 
$$J_G: KK^G_*(A,B)\to KK_*(A\rtimes_rG, B\rtimes_rG)$$
is defined by sending a class $x=[\mathcal E, \Phi, \gamma, T]\in KK^G(A,B)$ to
the class $J_G(x)=[\mathcal E\rtimes_r G, \Phi\rtimes_rG, \tilde{T}]\in KK(A\rtimes_rG, B\rtimes_rG)$,
where $[\mathcal E, \Phi,\gamma]\mapsto [\mathcal E\rtimes_rG, \Phi\rtimes_rG]$ is the 
descent in the correspondence categories  as described in \cite[\S 5.4]{CroPro}, and  the operator 
$\tilde{T}$ on $\mathcal E\rtimes_rG$ is given on the dense subspace $C_c(G,\mathcal E)$ by 
$$(\tilde{T}\xi)(s)=T(\xi(s))\quad\forall \xi\in C_c(G,\mathcal E), s\in G.$$
A similar descent also exists if we replace the reduced crossed products by full crossed products.

Now, if $A$ is a $G$-$C^*$-algebra, we can consider the following chain of maps
 \begin{align*}
 \mu_X: KK_*^G(C_0(X),A)\stackrel{J_G}{\longrightarrow}& KK_*(C_0(X)\rtimes_r G, A\rtimes_{r}G)\\
 &\stackrel{[p_c]\otimes\cdot}{\longrightarrow}
 KK_*(\CC, A\rtimes_rG)\cong K_*(A\rtimes_rG),
 \end{align*}
 where $J_G$ denotes Kasparov's descent homomorphism. If $X$ and $Y$ are two $G$-compact proper $G$-spaces and 
 if $\varphi:X\to Y$ is a continuous $G$-equivariant map, then one can check that $\varphi:X\to Y$ is automatically proper,
 i.e., inverse images of compact sets are compact, and therefore it induces a $G$-equivariant $*$-homomorphism
 $$\varphi^*: C_0(Y)\to C_0(X); f\mapsto f\circ \varphi.$$
 Moreover, if $c:Y\to[0,\infty)$ is a cut-off function for $(G,Y)$, then $\varphi^*(c):X\to [0,\infty)$ is a cut-off function for $(G,X)$
 such that $p_{\varphi^*(c)}=(\varphi^*\rtimes_rG)(p_c)$. Using this fact, it is easy to check that the diagram
 $$\begin{CD}
 KK_*^G(C_0(X),A) @>>{J_G}> KK_*(C_0(X)\rtimes_r G, A\rtimes_{r}G) @>>[p_{\varphi^*c}]\otimes\cdot> KK_*(\CC, A\rtimes_rG)\\
 @V [\varphi^*]\otimes \cdot VV      @V [\varphi^*\rtimes_rG]\otimes \cdot VV   @VV=V\\
  KK_*^G(C_0(Y),A) @>>{J_G}> KK_*(C_0(Y)\rtimes_r G, A\rtimes_{r}G) @>>[p_{c}]\otimes\cdot> KK_*(\CC, A\rtimes_rG)
  \end{CD}
  $$ 
  commutes. Hence if we define the {\em topological $K$-theory of $G$ with coefficients in $A$} as
  $$K_*^G(\EG;A):=\lim_{\begin{smallmatrix} {X\subseteq \EG}\\{\text{$X$ is $G$-compact}}\end{smallmatrix}}KK_*^G(C_0(X),A)$$
  where the $G$-compact subsets of $\EG$ are ordered by inclusion, we get a well defined homomorphism
  $$\mu_{(G,A)}:K_*^G(\EG;A)=\lim_XKK_*^G(C_0(X),A)\stackrel{\lim_X \mu_X}{\longrightarrow} K_*(A\rtimes_rG).$$
 This is the Baum-Connes assembly map for the system $(A,G,\alpha)$. We say that {\em $G$ satisfies BC for $A$}, if this 
 map is bijective.

 \begin{remark}\label{rem-fullassembly}
 We should remark that  almost the same construction yields an assembly map 
 $$\mu^{full}_{(G,A)}: K_*^G(\EG;A)\to K_*(A\rtimes_{\alpha} G)$$
 for the full crossed product  $A\rtimes_{\alpha}G$ such that
 $$\Lambda_*\circ \mu^{full}_{(G,A)}=\mu_{(G,A)},$$
 where $\Lambda: A\rtimes_{\alpha}G\to A\rtimes_{\alpha,r}G$ denotes the regular representation.
 The only difference is that we then use Kasparov's descent $J_G^{full}: KK^G(C_0(X),A)\to KK(C_0(X)\rtimes G, A\rtimes G)$ 
 for the full crossed products. Note that, since $G$ acts properly (and hence amenably) on $X$, we have 
 $C_0(X)\rtimes G\cong C_0(X)\rtimes_rG$ (see Remark \ref{rem-Green-proper} below), 
 so that we can use the same product of the fundamental class $[p_c]\in K_*(C_0(X)\rtimes G)$
 as for the reduced assembly map. But it is well known that the full analogue of the conjecture must fail for 
 all lattices $\Gamma$ in any almost connected  Lie group $G$ which has Kazhdan's property (T). 
 But for a large class of groups (including the  $K$-amenable groups of Cuntz and Julg-Valette \cites{Cuntz-K-amenable, JV}), 
 the regular representation induces an isomorphism in $K$-theory for the full and the reduced  crossed products, 
 and then the full assembly map  coincides up to this isomorphism with the reduced one.
 \end{remark}

\begin{example}[{\bf The Green-Julg Theorem}]\label{ex-Green-Julg}
If $G$ is a compact group with normed Haar measure we have $\EG=\{\pt\}$ and hence $K_*^G(\EG;A)=KK_*^G(\CC, A)$ is the $G$-equivariant 
$K$-theory $K_*^G(A)$ of $A$. The isomorphism  $K_*^G(A)\cong K_*(A\rtimes G)$ is the content of the Green-Julg theorem
(see \cite{Julg}). Let us briefly look at the special form of the assembly map in this situation:
First of all, if $G$ is compact, we may realise $KK^G(\CC,A)$ as the set of homotopy classes of triples $(\E, \gamma, T)$
in which $\E$ is a graded Hilbert $A$-module, $\gamma:G\to \Aut(\E)$ is a compatible action and $T\in \L(\E)$ is a 
$G$-invariant operator such that 
$$T^*-T, T^2-1\in \K(\E).$$
(If $T$ is not $G$-invariant, it may be replaced by the $G$-invariant operator $\tilde{T}=\int_G \Ad\gamma_s(T)\,ds$). 
The cut-off function of the one-point space is simply the function which sends this point to $1$, and the 
projection $p\in C(G)\subseteq C^*(G)$ is the constant function $1_G$. It acts on $\xi\in C(G,\E)\subseteq \E\rtimes G$ via
$(p\xi)(s)=\int_G \gamma_t(\xi(t^{-1}s))\, dt$.
 
Now, given $(\E,\gamma, T)\in KK^G(\CC, A)$
the assembly map sends this class to the class of the Kasparov cycle $(\E\rtimes G, p, \tilde{T})$, with 
$(\tilde{T}\xi)(s)=T(\xi(s))$ for all $s\in G$. Since $T$ is $G$-invariant, a short computation shows that $\tilde{T}$ commutes with $p$,
and hence we can decompose $(\E\rtimes_rG, p, \tilde{T})=(p(\E\rtimes G), \one, \tilde{T})\oplus ((\one-p)(\E\rtimes G), 0, \tilde{T})$ in which the second summand is degenerate. Thus we get
$$\mu([\E,\gamma, T])=[p(\E\rtimes G), \tilde{T}]\in KK(\CC,A\rtimes G).$$
Note that a function $\xi\in C(G,\E)$ lies in $p(\E\rtimes G)$ if and only if $\xi(s)=\gamma_s(\xi(e))$ for all $s\in G$, and it 
is clear that such functions are dense in $p(\E\rtimes G)$. Using this, the module $p(\E\rtimes G)$ can be described alternatively as follows:
We equip $\E$ with the $A\rtimes G$ valued inner product and left action of $A\rtimes G$ given by
$$\lk e_1, e_2\rk_{A\rtimes G}(s)=\lk e_1, \gamma_s(e_e)\rk_A\quad\text{and}\quad e\cdot f=\int_G\gamma_s(e\cdot f(s^{-1}))\,ds$$
for $f\in C(G,A)\subseteq A\rtimes G$. We denote by $\E_{A\rtimes G}$ the completion of $\E$ as a Hilbert $A\rtimes G$-module with this 
action and inner product. It is then easy to check that 
every $G$-invariant operator $S\in \L(\E)$ extends to an operator $S^G\in \L(\E_{A\rtimes G})$.
A short computation shows that the  map
$\Phi: p(\E\rtimes G)\to \E_{A\rtimes G}$ given by  $\Phi(\xi)=\xi(e)$ for $\xi\in C(G,\E)\cap p(\E\rtimes G)$
is  an isomorphism of Hilbert $A\rtimes G$-modules which intertwines $\tilde{T}$ with
$T^G\in \L(\E_{A\rtimes G})$. Using this, we get the following description of the assembly map
$$\mu: KK^G(\CC,A)\to KK(\CC, A\rtimes G);\quad \mu([\E,\gamma, T])=[\E_{A\rtimes G}, T^G].$$
This map has a direct inverse given as follows: Let $L^2(G,A)$ be the Hilbert $A$-module with $A$-valued inner product given 
by $\lk f, g\rk_A=\int_G \alpha_{t^{-1}}(f(t)^*g(t))\,dt$ and right $A$-action given by $(f\cdot a)(s)=f(s)\alpha_s(a)$. Then 
$A\rtimes G$ acts on $L^2(G,A)$ via the regular representation given by convolution.
There is a canonical $\alpha$-compatible action  $\sigma:G\to \Aut(L^2(G,A))$ given by right translation $\sigma_s(f)(t)=f(ts)$. 
It is then not difficult to check that
$$\nu: KK(\CC,A\rtimes G)\to KK^G(\CC, A);\quad  [\F, S]\mapsto [\F\otimes_{A\rtimes G}L^2(G,A), \id\otimes \sigma, S\otimes 1]$$
is an inverse of $\mu$. For a few more details on these computations see \cite{E-appendix}.
\end{example}

\begin{example}\label{ex-discrete}
If $G$ is a discrete torsion free group, then $\EG=EG$, the universal principal $G$-bundle of $G$. 
Since $G$ acts freely and properly on $EG$ it follows from a theorem of Green \cite{Green-trans} 
that $C_0(EG)\rtimes G$ is Morita equivalent to $C_0(G\backslash EG)=C_0(BG)$, where $BG=G\backslash EG$
is the classifying space of $G$. Now, for any discrete group and any $G$-$C^*$-algebra $A$ we have a canonical isomorphism 
$$KK^G(A,\CC)\cong KK(A\rtimes G, \CC)$$
which sends the class of an equivariant $A-\CC$ $KK$-cycle $(\H, \Phi, \gamma, T)$ to the class of the $A\rtimes G- \CC$ 
$KK$-cycle $(\H, \Phi\rtimes \gamma, T)$. Note that in this situation $(\Phi,\gamma)$ is a covariant representation of 
$(A,G,\alpha)$ on the Hilbert space $\H$, and hence sums up to a representation of $A\rtimes G$ by the universal property of the 
full crossed product. Since $G$ is discrete, one checks that condition (ii) in Definition \ref{def-KK-cycle} for 
$(\H, \Phi\rtimes \gamma, T)$ is equivalent to the corresponding condition for $(\H, \Phi, \gamma, T)$.
Thus, if in addition $EG$ is $G$-compact, we get
\begin{align*}
K_*^G(\EG, \CC)\cong KK_*^G(C_0(EG), \CC)&\cong KK_*(C_0(EG)\rtimes G,\CC)\\
&\stackrel{\text{Morita-eq.}}{\cong } KK_*(C(BG),\CC)=K_*(BG).
\end{align*}
Hence in this situation the left hand side of the Baum-Connes conjecture is the topological $K$-homology of the classifying space of 
$G$. If $\EG$ is not $G$-compact, a similar argument gives
$$K_*^G(\EG, \CC)=\lim_{C\subseteq BG} K_*(C),$$
where $C$ runs through the compact subsets of $BG$, which is the $K$-homology of $BG$ with compact supports.
Hence the Baum-Connes conjecture relates the $K$-theory of  the (often quite complicated) $C^*$-algebra  $C_r^*(G)$ to the 
$K$-homology of the classifying space $BG$ of $G$, which can be handled by methods of classical algebraic topology
(but can still be difficult to compute).
\end{example}

We close this section with an exercise:

\begin{exercise}\label{ex-commute}
Suppose that $A$ and $B$ are $G$-$C^*$-algebras and let $x\in KK^G(A,B)$. Then $x$ induces a map 
$$\cdot\otimes_A x: K_*^G(\EG;A)\to K_*^G(\EG;A)$$
given on the level of $KK^G(C_0(X),A)$ for some $G$-compact subset $X\subseteq \EG$ via the map
$$KK^G(C_0(X),A)\to KK^G(C_0(X),B); y\mapsto y\otimes_Ax.$$
On the other hand, we have a map 
$$\cdot\otimes_{A\rtimes_rG} j_G(x): KK(\CC, A\rtimes_rG)\to KK(\CC, B\rtimes_rG)$$
between the $K$-theory groups of the crossed products.

Show that the map $\cdot\otimes_A x: K_*^G(\EG;A)\to K_*^G(\EG;A)$ is well-defined and that the 
diagram
$$
\begin{CD}
K_*^G(\EG;A)   @>\mu_{(G,A)} >> K_*(A\rtimes_r G)\\
@V \cdot\otimes_A x VV  @VV \cdot\otimes_{A\rtimes_rG}J_G(x)V\\
K_*^G(\EG;B)   @>\mu_{(G,B)} >> K_*(B\rtimes_r G)
\end{CD}
$$
commutes. Show that it follows from this that if $A$ and $B$ are $KK^G$-equivalent, then 
$\mu_{(G,A)}$ is an isomorphism if and only if $\mu_{(G,B)}$ is an isomorphism. 
Check that a similar result holds for the full assembly maps $\mu_{(G,A)}^{full}$  and 
$\mu_{(G,B)}^{full}$ of Remark \ref{rem-fullassembly}.
\end{exercise}

\subsection{Proper $G$-algebras and the Dirac dual-Dirac method}\label{sec-Dirac}
As an extension of the Green-Julg theorem one can prove that the Baum-Connes assembly map is always an isomorphism 
if the coefficient algebra $A$ is a proper $G$-$C^*$-algebra in the sense of Kasparov, which we are now going to explain.
Recall that if $X$ is a locally compact space, then a $C^*$-algebra $A$ is called a $C_0(X)$-algebra, if there exists
a non-degenerate $*$-homomorphism 
$$\Phi: C_0(X)\to \mathcal{ZM}(A),$$
the center of the multiplier algebra of $A$. If $A$ is a $C_0(X)$-algebra, then $A$ can be realized as an algebra of $C_0$-sections 
of a (upper semicontinuous) bundle of $C^*$-algebras $\{A_x: x\in X\}$, where each fibre $A_x$ is given by
$A_x=A/I_x$ with $I_x=(C_0(X\smallsetminus\{x\})\cdot A)$, where we write $f\cdot a:=\Phi(f)a$ for $f\in C_0(X)$, $a\in A$. We refer to 
\cite{Wi-crossed}*{Appendix C}  for a detailed discussion of $C_0(X)$-algebras. 

\begin{definition}\label{def-proper-G-algebra}
Suppose that $G$ is a locally compact group and $A$ is a $G$-$C^*$-algebra. Suppose further that $A$ is a $C_0(X)$-algebra 
such that the structure map $\Phi:C_0(X)\to \mathcal{ZM}(A)$ is $G$-equivariant. We then say that $A$ is an
$X\rtimes G$-$C^*$-algebra. If $A$ is an $X\rtimes G$-$C^*$-algebra for some proper $G$-space $X$, then
 $A$ is called a {\em proper $G$-$C^*$-algebra}.
\end{definition}

Note that in the definition of a proper $G$-$C^*$-algebra we may always assume $X$ to be a realisation of $\EG$: Since if $\varphi:X\to \EG$ is a 
$G$-equivariant continuous map, we get a non-degenerate $G$-equivariant $*$-homomorphism
$$\varphi^*: C_0(\EG)\to C_b(X)=\mathcal M(C_0(X)); \varphi^*(f)=f\circ \varphi$$ and then the composition 
$\Phi\circ \varphi^*: C_0(\EG)\to\mathcal{ZM}(A)$ makes $A$ into an $\EG\rtimes G$-algebra. 
Recall from our discussion of proper $G$-spaces that proper actions behave very much like actions 
of compact groups since they are {\em locally induced} from actions of compact subgroups. It is therefore not very surprising 
that an analogue of the Green-Julg theorem should hold also for proper $G$-$C^*$-algebras. 

\begin{theorem}\label{thm-proper-Green-Julg}
Suppose that $A$ is a proper $G$-$C^*$-algebra. Then the Baum-Connes assembly map
$$\mu: K_*^G(\EG,A)\to K_*(A\rtimes_rG)$$
is an isomorphism.
\end{theorem}

However, the proof of this result is 
much harder than the proof of the Green-Julg theorem shown in the previous section. In what follows we 
want to indicate at least some ideas towards this result. On the way we discuss some useful results about induced dynamical systems 
and their applications to the Baum-Connes conjecture. 

In what follows suppose that $H$ is a closed subgroup of the locally compact group $G$
and that $\beta:H\to \Aut(B)$ is an action of $H$ on the $C^*$-algebra $B$. Recall from  \cite{CroPro}*{Section 2.6}  that the {\em induced $C^*$-algebra}
$\Ind_H^GB$ is defined as the algebra 
$$\Ind_H^GB:=\left\{f\in C_b(G,B): \begin{matrix} f(sh)=\beta_{h^{-1}}(f(s))\;\text{for all $s\in G$ and $h\in H$}\\ \text{and}\; (sH\mapsto\| f(s)\|)\in C_0(G/H)
\end{matrix}\right\}.$$
This is a $C^*$-subalgebra of $C_b(G,B)$ which carries an action $\Ind\beta:G\to \Aut(\Ind_H^GB))$ given by 
$$\big(\Ind\beta_s(f)\big)(t)=f(s^{-1}t).$$
If $Y$ is a locally compact $H$-space, then $\Ind_H^GC_0(Y)\cong C_0(G\times_HY)$ as $G$-algebras.
Hence the above procedure extends the procedure of inducing $G$-spaces as discussed in Example \ref{ex-proper} above.
The following result is quite useful when working with induced algebras. For the formulation recall that for any 
$G$-$C^*$-algebra $A$ we have a continuous action of $G$ on the primitive ideal space $\Prim(A)$ given by $(s, P)\mapsto \alpha_s(P)$.

\begin{theorem}\label{thm-induced}
Suppose that $A$ is a $G$-$C^*$-algebra and let $H$ be a closed subgroup of $G$. Then the following are equivalent:
\begin{enumerate}
\item There exists an $H$-algebra $B$ such that $A\cong \Ind_H^GB$ as $G$-algebras.
\item $A$ carries the structure of a  $G/H\rtimes G$-$C^*$-algebra.
\item There exist a continuous $G$-equivariant map $\phi: \Prim(A)\to G/H$.
\end{enumerate}
\end{theorem}
\begin{proof} (i) $\Leftrightarrow$ (iii) is  \cite{CroPro}*{Theorem 2.6.2}. 
The proof of (iii) $\Leftrightarrow$ (ii) follows from 
the general correspondence between continuous maps $\phi: \Prim(A)\to X$ and non-degenerate $*$-homomorphisms 
$\Phi:C_0(X)\to \mathcal{ZM}(A)$ given by the Dauns-Hofmann Theorem. We refer to \cite{Wi-crossed}*{Appendix C} for a discussion 
of this correspondence. 
\end{proof}

Let us briefly indicate how the objects in (i) and (ii) of the above theorem are related to each other.
If $A=\Ind_H^GB$, then the $G$-equivariant $*$-homomorphism $\Phi:C_0(G/H)\to \mathcal{ZM}(\Ind_H^GB)$ is simply given by
$$\big(\Phi(g)f\big)(s)=g(sH)f(s), \quad g\in C_0(G/H), f\in \Ind_H^GB.$$
Conversely, if $\Phi: C_0(G/H)\to\mathcal{ZM}(A)$ is given, let 
$$B:=A_{eH}=A/(C_0(G/H\smallsetminus \{eH\})\cdot A)$$
 be the fibre of $A$ over the coset $eH$. 
Since $\Phi:C_0(G/H)\to\mathcal{ZM}(A)$ is $G$-equivariant, it follows that the ideal $C_0(G/H\smallsetminus \{eH\})\cdot A$ is $H$-invariant for the restriction of $\alpha$ to $H$. Thus $\alpha|_H$ induces an action 
$\beta:H\to \Aut(A_{eH})=\Aut(B)$. The $G$-isomorphism $\Psi:A\to\Ind_H^GB$ is then given by
$$\Psi(a)(s)=q(\alpha_{s^{-1}}(a)),$$
where $q:A\to A_{eH}$ denotes the quotient map. 

The induction of $H$-algebras  to $G$-algebras extends to an induction map  
$$\Ind_H^G: KK^H(B,C)\to KK^G(\Ind_H^GB, \Ind_H^GC)$$
given as follows: If $[\E,\Phi,\gamma, T]\in KK^H(B,C)$, then we define the induced Hilbert $\Ind_H^GC$-module $\Ind_H^G\E$
as 
$$\Ind_H^G\E=\left\{\xi\in C_b(G,\E): \begin{matrix}\xi(sh)=\gamma_{h^{-1}}(\xi(s))\;\text{for all $s\in G$ and $h\in H$}\\ \text{and}\; (sH\mapsto\| f(s)\|)\in C_0(G/H)
\end{matrix}\right\}$$
with $\Ind_H^GC$-valued inner product and left $\Ind_H^GC$-action given as follows:
$$\lk \xi,\eta\rk_{\Ind_H^GC}(s)=\lk \xi(s),\eta(s)\rk_C \quad \text{and}\quad (\xi\cdot f)(s)=\xi(s)\cdot f(s).$$
Similarly, if $\Phi:B\to \L(\E)$ is a left action of $B$ on $\E$, then we get an action $\Ind\Phi:\Ind_H^GB\to \L(\Ind_H^G\E)$ 
by 
$$(\Ind\Phi(g)\xi)(s)=\Phi(g(s))\xi(s).$$
Finally, we define the operator $\tilde{T}\in \L(\Ind_H^G\E)$ via $(\tilde{T}\xi)(s)=T(\xi(s))$. It is not difficult to check that 
$(\ind_H^G\E, \Ind\Phi, \Ind\beta, \tilde{T})$ is a $G$-equivariant $\Ind_H^GB-\Ind_H^GC$ Kasparov cycle and Kasparov's induction 
map in $KK$-theory is then defined as 
$$\Ind_H^G([\E,\Phi,\gamma, T])=[\ind_H^G\E, \Ind\Phi, \Ind\beta, \tilde{T}]\in KK^G(\Ind_H^GB,\Ind_H^GC).$$
We want to use this map to define an induction map 
$$\Ind_H^G: K_*^H(\EH; B)\to K_*^G(\EG ; \Ind_H^GB)$$
for every $H$-algebra $B$. For this suppose that $Y\subseteq \EH$ is an $H$-compact subset.
Then the induced $G$-space $G\times_HY$ is proper and  $G$-compact  and therefore maps equivariantly 
into $\EG$ via some continuous map $j: G\times_HY\to \EG$ whose image is a $G$-compact subset $X(Y)\subseteq \EG$. 
One can check that the composition of maps 
$$KK^G(C_0(Y), B)\stackrel{\Ind_H^G}{\longrightarrow} KK^G(C_0(G\times_HY), \Ind_H^GB)  \stackrel{j_*}{\longrightarrow} KK^G(C_0(X(Y)), \Ind_H^GB)$$
is compatible with taking limits and therefore induces a well-defined induction map
$$I_H^G: K^H_0(\EH;B)\to K^G_0(\EG;\Ind_H^GB).$$
Replacing $B$ by $B\otimes \Cl_1$ (or $B\otimes C_0(\RR)$) gives an analogous map from 
$ K^H_1(\EH;B)$ to  $K^G_1(\EG;\Ind_H^GB)$. We then have the following theorem, which has been shown in \cite{CE2}*{Theorem 2.2} (for  
 $G$ discrete and  $H$ finite the result has first been obtained  earlier in \cite{GHT}):

\begin{theorem}\label{thm-ind2}
Suppose that $H$ is a closed subgroup of $G$. Then the induction map 
$$I_H^G: K^H_*(\EH;B)\to K^G_*(\EG;\Ind_H^GB)$$
is an isomorphism of abelian groups for every $H$-$C^*$-algebra $B$.
\end{theorem}

Now Green's imprimitivity theorem (see \cite{CroPro}*{Theorem 2.6.4}) says that the full (resp. reduced) crossed products 
$B\rtimes_{\beta, (r)}G$ and $\Ind_H^GB\rtimes_{\Ind\beta, (r)}G$ are Morita equivalent via a canonical $B\rtimes_{(r)} H-\Ind_H^GB\rtimes_{(r)} G$ equivalence bimodule $X_H^G(B)_{(r)}$. Since Morita equivalences provide 
$KK$-equivalences, we obtain the following diagram of maps
$$
\begin{CD}
K^H_*(\EH;B)   @>\mu_H >>  K_*(B\rtimes_rH)\\
@V I_H^G VV      @VV \otimes [X_H^G(B)_r] V\\
K_*^G(\EG, \Ind_H^GB)   @>>\mu_G > K_*(\Ind_H^GB\rtimes_rG)
\end{CD}
$$
in which both vertical arrows are isomorphisms. It is shown in \cite{CE2}*{Proposition 2.3} that this diagram commutes. As a corollary 
we get

\begin{corollary}\label{cor-ind-assembly}
Suppose that $H$ is a closed subgroup of $G$ and $B$ is an $H$-algebra. Then the assembly map
$$\mu_H : K^H_*(\EH;B)  \to  K_*(B\rtimes_rH)$$
is an isomorphism if and only if 
$$\mu_G:K_*^G(\EG, \Ind_H^GB) \to  K_*(\Ind_H^GB\rtimes_rG)$$
is an isomorphism. 
In particular, if $G$ satisfies BCC, then so does $H$. A similar result holds for the assembly maps into the 
$K$-theories of the full crossed products $B\rtimes H$ and $\Ind_H^GB\rtimes G$, respectively (see Remark \ref{rem-fullassembly}).
\end{corollary}

We now come back to general proper $G$-algebras $A$. So suppose that $X$ is a proper $G$-space 
and that $A$ is an $X\rtimes G$-$C^*$-algebra. Since every proper $G$-space is locally induced from a 
compact subgroup, we find for each $x\in X$ an open $G$-invariant neighbourhood $U\subseteq X$ 
such that $U\cong G\times_KY$ for some compact subgroup $K$ of $G$ and some $K$-space $Y$.
Then $C_0(G\times_KY)\cong \Ind_H^GC_0(Y)$ is a $G/K\rtimes G$-algebra by Theorem \ref{thm-induced}.
Let $A(U):=\Phi(C_0(U))A\subseteq A$. Then $A(U)$ is a $G$-invariant ideal of $A$
and carries the structure of a $U\rtimes G$-algebra in the canonical way.
The composition 
$$C_0(G/H)\to C_b(G\times_KY)\cong C_b(U)\stackrel{\Phi}{\longrightarrow} \mathcal{ZM}(A(U))$$
then gives $A(U)$ the structure of a $G/K\rtimes G$-algebra. Thus it follows from Theorem \ref{thm-induced}
that $A(U)\cong \Ind_K^GB$ for some $K$-algebra $B$. By the Green-Julg theorem we know that $K$ 
satisfies BCC and hence it follows from Corollary \ref{cor-ind-assembly} that the assembly map
$$\mu_U: K_*^G(\EG, A(U))\to K_*(A(U)\rtimes_r G)$$
is an isomorphism. Thus we see that for proper $G$-algebras the Baum-Connes conjecture holds {\em locally}.
Now every $G$-invariant open subset $W\subseteq X$ with $G$-compact closure $\overline{W}$ can be covered 
by a finite union of open sets $U_1,\ldots, U_l$ such that each $U_i$ is isomorphic to some induced space $G\times_{K_i}Y_i$
for some compact subgroup $K_i\subseteq G$. Using six-term sequences and induction on the number $l$ of open sets in this 
covering, we then conclude that
$$\mu_W: K_*^G(\EG, A(W))\to K_*(A(W)\rtimes_r G)$$
for all such $W$. Now, taking inductive limits indexed by $W$, one can show that the assembly map
$$\mu: K_*^G(\EG, A)\to K_*(A\rtimes_r G)$$
is an isomorphism as well. This then finishes the proof of Theorem \ref{thm-proper-Green-Julg}. (We refer to \cite{CEM} for the 
original proof and further details.)

\begin{remark}\label{rem-Green-proper} An application of Green's imprimitivity theorem also implies 
 that for any proper $G$-$C^*$-algebra $A$ the 
full and reduced crossed products coincide. To see this, let $A$ be an $X\rtimes G$-algebra for some proper $G$-space $X$.
Let  $\{U_i:i\in I\}$ be an open cover of $X$ consisting of $G$-invariant open sets such that each of these sets is induced by some 
compact subgroup $K$ of $G$. Suppose now that $\pi\rtimes U: A\rtimes G\to \B(\H)$ is any irreducible representation of 
$A\rtimes G$. We claim that there exists at least one $i\in I$ such that $\pi$ does not vanish on the ideal $A(U_i)$ of $A$, 
and hence $\pi\rtimes U$ does not vanish on the ideal $A(U_i)\rtimes G$ of the crossed product. Indeed, since $\{U_i:i\in I\}$ is 
a covering  of $X$ it is an easy exercise, using a partition of unity argument,
 to show that $\sum_{i\in I}A(U_i)$ is a dense ideal in $A$. The claim then follows since $\pi\neq 0$. 

It now suffices to show that $A(U_i)\rtimes G=A(U_i)\rtimes_rG\subseteq A\rtimes_rG$, since this implies that every irreducible 
representation of the ideal $A(U_i)\rtimes G$ corresponds to an irreducible representation of $A\rtimes_rG$.  To see this
recall that $A(U_i)\cong \Ind_K^GB$ for some compact subgroup $K$ of $G$ and some $K$-algebra $B$. 
Since $K$ is compact, hence amenable, we have $B\rtimes K= B\rtimes_r K$. 
 Since the $B\rtimes_rK - \Ind_K^GB\rtimes_rG$-equivalence bimodule $X_H^G(B)_r$ is 
the quotient of the $B\rtimes K - \Ind_K^GB\rtimes G$-equivalence bimodule $X_K^G(B)$ by the submodule 
$(\ker\Lambda_{(B,K)})\cdot X_K^G(B)=\{0\}$, it follows from the Rieffel-correspondence (\cite{CroPro}*{Proposition 2.5.4})
that $A(U_i)\rtimes G\cong \Ind_K^GB\rtimes G=\Ind_K^GB\rtimes_rG\cong A(U_i)\rtimes_rG$, which finishes the proof.
\end{remark}

We are now coming to Kasparov's Dirac-dual Dirac method for proving the Baum-Connes conjecture. 
As we shall discuss below, this has been the most successful method so far for proving that the conjecture 
holds for certain classes of groups. Since, as we saw above, the Baum-Connes conjecture always holds
for proper $G$-$C^*$-algebras as coefficients, the basic idea is to show that for a  given group $G$ 
{\em every} $G$-$C^*$-algebra  $B$ is $KK^G$-equivalent to a proper $G$-$C^*$-algebra. Since by Exercise \ref{ex-commute} the 
validity of the Baum-Connes conjecture is invariant under passing to $KK^G$-equivalent coefficient 
algebras, this would result in a proof that 
the group $G$ satisfies BCC, i.e., Baum-Connes for all coefficients. 
Indeed, we need less:

\begin{definition}\label{def-strongBC} 
Suppose that $G$ is a second countable locally compact group and assume that there is a proper $G$-$C^*$-algebra $\mathcal D$
together with elements 
$$\alpha\in KK^G_0(\mathcal D, \CC) \quad\text{and}\quad \beta\in KK^G(\CC, \mathcal D)$$
such that
$$\gamma:=\beta\otimes_{\mathcal D}\alpha=1_\CC\in KK^G(\CC,\CC).$$
We then say that  {\em $G$ has $\gamma$-element equal to one}.
If, in addition $\alpha\otimes_\CC \beta=1_{\mathcal D}$, then we say that $G$ satisfies the 
{\em strong Baum-Connes conjecture}.\end{definition}

Note that in almost all cases
where  $G$ has  $\gamma$-element equal to one  there is also a proof of strong BC.

If $G$ has $\gamma$-element equal to one and  $B$ is any other  $G$-$C^*$-algebra, it follows that
$$(\beta\hat\otimes 1_B)\otimes_{\mathcal D\hat\otimes B}(\alpha\hat\otimes 1_B)=\gamma\hat\otimes 1_B=1_B\in KK^G(B,B)$$
and similarly, since the descent $KK^G(A,B)\to KK(A\rtimes_rG, B\rtimes_rG)$ is compatible with Kasparov products, we get
$$J_G(\beta\hat\otimes 1_B)\otimes_{(\mathcal D\hat\otimes B)\rtimes_rG}J_G(\alpha\hat\otimes 1_B)=J_G(1_B)=1_{B\rtimes_rG}\in KK(B\rtimes_rG,B\rtimes_rG).$$

Moreover, it follows from Exercise \ref{ex-commute} that the following diagram commutes
\begin{equation}\label{eq-DD}
\begin{CD}
K_*^G(\EG, B) @>\mu_{(G,B)} >> K_*(B\rtimes_rG)\\
@V\cdot \otimes (\beta\otimes 1_B) VV      @VV \cdot\otimes J_G(\beta\otimes 1_B) V\\
K_*^G(\EG, \mathcal D\hat\otimes B) @>\mu_{(G,\mathcal D\hat\otimes B)} >\cong > K_*((\mathcal D\hat\otimes B)\rtimes_rG)\\
@V\cdot \otimes (\alpha\otimes 1_B) VV      @VV \cdot\otimes J_G(\alpha\otimes 1_B) V\\
K_*^G(\EG, B) @>\mu_{(G,B)} >>  K_*(B\rtimes_rG).
\end{CD}
\end{equation}
Since $\mathcal D\hat\otimes B$ is a proper $G$-algebra (via the composition of $\Phi: C_0(X)\to \mathcal{ZM}(\mathcal D)$ with  the 
canonical map of $\mathcal M(\mathcal D)$ to $\mathcal M(\mathcal D\hat\otimes B)$), the middle horizontal  map is an isomorphism of abelian groups,
and by the above discussion it follows that the compositions of the vertical maps on either side are isomorphisms as well. It then follows 
by an easy diagram chase that the upper horizontal map is injective and the lower horizontal map is surjective, hence $\mu_{(G,B)}$ is an isomorphism as well. 
Thus we get

\begin{corollary}\label{cor-strong}
If $G$ has $\gamma$-element equal to one, then $G$ satisfies the Baum-Connes conjecture with coefficients (BCC).
\end{corollary}

\begin{remark}
In  diagram (\ref{eq-DD}) we can  replace all reduced crossed products by the full ones and the (reduced) assembly map
by the full assembly map to see that whenever $G$ has $\gamma$-element equal to one,
then the full assembly map
$$\mu^{full}_{(G,B)}:K_*^G(\EG;B)\to K_*(B\rtimes_\beta G)$$
is an isomorphism as well. Moreover, if  $\Lambda: B\rtimes_{\beta}G\to B\rtimes_{\beta,r}G$ is the regular representation, then  the diagram
$$
\xymatrix{
K_*^G(\EG;B) \ar[rd]_{\mu_{(G,B)}} \ar[r]^{\mu_{(G,B)}^{full}}
&K_*(B\rtimes G) \ar[d]^{\Lambda_*}\\
&K_*(B\rtimes_rG)}
$$
commutes. Since both assembly maps are isomorphisms, 
it follows that the regular representation induces an isomorphism in $K$-theory between the maximal and the reduced 
crossed products by $G$. Indeed, it is shown by Tu in \cite{Tu3} that $G$ is $K$-amenable in the sense of 
Cuntz \cite{Cuntz-K-amenable} and Julg-Valette \cite{JV} (which actually implies that $\Lambda$ is a $KK$-equivalence)
whenever $G$ has $\gamma$-element equal to one.
\end{remark}
\begin{remark}\label{rem-restriction}
If $G$ has $\gamma$-element equal to one, then so does every closed subgroup of $G$. Indeed, if 
$\alpha\in \KK^G(\D,\CC)$ and $\beta\in KK^G(\CC,\D)$ are  as in the definition of strong BC, then 
the action of $G$ on $\D$ restricts to a proper action of $H$ on $\D$.
 Moreover, for every pair of $G$-$C^*$-algebras $A,B$ we have a natural homomorphism
$$\res^G_H:KK^G(A,B)\to KK^H(A,B)$$
which is given by simply restricting all actions on algebras and Hilbert modules from $G$ to $H$. 
It is easy to see that this restriction map is compatible with the Kasparov product, so that we get
$$\res_H^G(\beta)\otimes_\D \res_H^G(\alpha)=\res_H^G(\beta\otimes_\D\alpha)=\res_H^G(1_\CC)=1_\CC\in KK^H(\CC,\CC).$$
\end{remark}

\begin{example}\label{ex-BC-ZZ}
As a sample, we want to show that  $\RR$ and $\ZZ$ satisfy strong BC (and, in particular, have $\gamma=1_\CC$). For this recall the 
construction of the Dirac and dual Dirac elements in the proof of the Bott periodicity theorem in Section \ref{sec-Bott}:
Let  $\D=C_0(\RR)\hat\otimes \Cl_1$. We constructed  elements 
$\alpha\in KK(\D,\CC)$ and $\beta\in KK(\CC,\D)$ which are inverse to each other in $KK$. Let $\tau:\RR\to \Aut(C_0(\RR))$ 
denote the translation action $(\tau_s(f))(x)=f(x-s)$. Then $\D$ becomes a proper $\RR$-algebra via the action $\tau\hat\otimes\id_{\Cl_1}$.
Now recall that the classes $\alpha$ and $\beta$ have been given by
$$\alpha=\left[\H=L^2(\RR)\oplus L^2(\RR), \Phi, T=\frac{D}{\sqrt{1+D^2}}\right]\quad\text{and}\quad \beta=[\D, \one, S],$$
in which $D=\bmtr 0& -\frac{d}{dt}\\ \frac{d}{dt}&0\emtr$, $\Phi: \D\to\L(\H)$ is given as in (\ref{eq-Phi}),  and $S=S_\varphi$ is given by
pointwise   multiplication with a function
 $\varphi:\RR\to[-1,1]$, which can be  any 
odd continuous function with $\varphi(x)=0\Leftrightarrow x=0$ and $\lim_{t\to \infty}\varphi(t)=1$. 
With the given $\RR$-action on $\mathcal D$ we may view $\beta$ as a class
$$\beta=[\D, \tau\hat\otimes\id_{\Cl_1}, \one, S]\in KK^\RR(\CC,\D).$$
The only extra condition to check is the condition that
$\Ad_{\tau\otimes \id(s)}(S)- S\in \K(\D)=\D$, which follows from the fact that for any function $\varphi$ as above, we have
$\tau_s(\varphi)-\varphi\in C_0(\RR)$ for all $s\in \RR$. Similarly, if we equip $\H= L^2(\RR)\oplus L^2(\RR)$ with  the representation $\lambda\oplus \lambda$,
where $\big(\lambda_s(\xi)\big)(x)=\xi(x-s)$ denotes the regular representation of $\RR$, we obtain a class 
$$\alpha=[\H, \lambda, \Phi, T]\in KK^\RR(\D,\CC).$$
As above, the only extra condition to check is that $(\Ad\lambda_s(T)-T)\Phi(d)\in \K(\H)$ for all $s\in \RR$ and $d\in \D$, which we leave as 
an exercise for the reader. We claim that 
\begin{equation}\label{eq-proper-Bott}
\beta\otimes_{\D}\alpha=1_\CC\in KK^\RR(\CC,\CC).
\end{equation} 
Note that this would be true if we  equipped everything with the trivial $\RR$-action instead of the translation action, since then 
it would be a direct consequence of the product $\beta\otimes_\D\alpha=1_\CC\in KK(\CC,\CC)$, which we proved in Section \ref{sec-Bott}.
We show equation (\ref{eq-proper-Bott}) by a simple trick, showing that  the translation action of $\RR$ is homotopic to the trivial action in the following sense:
We consider the algebra $\D[0,1]=\D\hat\otimes C[0,1]$ equipped with the $\RR$-action
 $\tilde{\tau}\hat\otimes \id_{\Cl_1}$ where 
$$\tilde\tau:\RR\to \Aut(C_0(\RR\times [0,1])); \big(\tilde\tau_s(f)\big)(x,t)=f(x-ts, t).$$
We then consider the class
$\tilde\alpha\in KK^\RR(\D[0,1],C[0,1])$, where $C[0,1]$ carries the trivial $\RR$-action, given by
$$\tilde\alpha=\big[\H\hat\otimes C[0,1], \tilde\lambda\oplus\tilde\lambda, \Phi\hat\otimes 1,  T\hat\otimes 1\big],$$
where the $\RR$-action $\tilde\lambda\oplus\tilde\lambda$ on $\H\hat\otimes C[0,1]$ is given by a formula similar to the one for $\tilde\tau$. 
On the other hand, we consider the class 
$$\tilde\beta=\big[\D[0,1], \tilde\tau\hat\otimes 1, \one, S\hat\otimes 1\big]\in KK^\RR(\CC, \D[0,1]).$$
If we evaluate the class $\tilde\beta\otimes_{\D[0,1]}\tilde\alpha\in KK^{\RR}(\CC, C[0,1])$ in $0$, we 
obtain the product of $\beta$ with $\alpha$ equipped with trivial $\RR$-actions, which is $1_\CC\in KK^\RR(\CC, \CC)$ 
by the proof of Bott periodicity. If we evaluate at $1$, we obtain the product $\beta\otimes\alpha$ with respect to the 
proper translation action on $\D$. Hence both classes are homotopic  which proves (\ref{eq-proper-Bott}).

Hence we see that the Dirac dual-Dirac method applies to $\RR$, and by Remark \ref{rem-restriction} it then also applies to $\ZZ\subseteq \RR$
and both have $\gamma$-element equal to one.
Using an easy product argument,  this proof also implies that $\RR^n$ and $\ZZ^n$ have $\gamma$-element equal to one. We leave 
it as an exercise  to check  that $\alpha\otimes_{\CC}\beta=1_{\mathcal D}\in KK^\RR(\D, \D)$, i.e., that
$\RR^n$ and $\ZZ^n$ do  satisfy the strong Baum-Connes conjecture.
\end{example}

Extending  Example \ref{ex-BC-ZZ}  to higher dimensions, one can use Kasparov's equivariant Bott periodicity theorem as discussed in the last paragraph 
of Section \ref{sec-Bott} to show that the Dirac dual-Dirac method works for all groups which act properly and isometrically by affine transformations 
on a finite dimensional euclidean space. This has already been pointed out by Kasparov in his conspectus 
\cite{K2}. Later, in \cite{K3}, he extended this to show that the method works for all amenable Lie groups (and their closed subgroups) and, together 
with Pierre Julg in \cite{JK}, they showed that the method works for the Lie groups $SO(n,1)$ and $SU(n,1)$ and their closed subgroups.
But the most far reaching positive result which includes all cases mentioned above
 has been obtained by Higson and Kasparov in \cite{HK}:

\begin{theorem}[Higson-Kasparov]\label{thm-HK}
Suppose that the second countable locally compact group acts continuously and metrically properly by isometric affine transformations 
on a separable real Hilbert space $\H$. Then  $G$ satisfies the  strong Baum-Connes conjecture.
\end{theorem}

Note that the action of $G$ on $\H$ is called {\em metrically proper} if for any $\xi\in \H$ and  $R>0$ there exists a compact subset $C\subseteq G$ such that 
$ \| s\cdot\xi \|>R$ for all $s\in G\setminus C$. The basic idea of the proof is to construct the proper $G$-algebra as an inductive limit of algebras $C_0(V)\hat\otimes\Cl(V)$, where 
$V\subseteq \H$ runs through the finite dimensional subspaces of the Hilbert space $\H$. But the precise construction of the algebra $\D$ 
and the classes $\alpha$ and $\beta$ is very complex and we refer to the original work \cite{HK} of Higson and Kasparov 
for more details. A detailed exposition of certain aspects of the proof can be found in the very recent paper \cite{Ni}.
We should also note that the original proof of Higson and Kasparov uses $E$-theory, a variant of $KK$-theory introduced 
by Connes and Higson in \cite{CH} (see also \cite{Bla86}*{Chapter 25}).  A groupoid version of the above theorem has been shown by Tu in \cite{Tu2}.

The class of groups which satisfies the conditions of the Higon-Kasparov theorem has been studied first by Gromov who called them
{\em a-$T$-menable} groups. 
A second countable group $G$ is  a-$T$-menable if and only if it satisfies the Haagerup approximation property
which says that the trivial representation $1_G$ can be approximated uniformly on compact sets by a net of positive definite functions 
$(\varphi_i)$ on $G$ such that each $\varphi_i$ vanishes at $\infty$ on $G$.  We refer to \cite{CCJJV} for a detailed exposition on the 
class of a-$T$-menable groups.  As a consequence of the theorem we get

\begin{corollary}\label{cor-HK}
Every amenable second countable locally compact group satisfies strong BC. Also, 
the free groups $F_n$ in $n$ generators, $n\in \NN\cup\{\infty\}$ and all closed subgroups 
of the Lie groups $\operatorname{SU}(n,1)$ and $\operatorname{SO}(n,1)$ satisfy   strong BC.
\end{corollary}

All groups in the above corollary satisfy the Haagerup property. 

The Dirac dual-Dirac method can also be used in cases, in which the element 
$$\gamma=\beta\otimes_\D\alpha\in KK^G(\CC,\CC)$$
is not necessarily equal to $1_\CC$, but where it satisfies the following weaker condition:

\begin{definition}[Kasparov's $\gamma$-element]\label{def-gamma}
Suppose that $\D$ is a proper  $G$-algebra, $\alpha\in \KK^G(\D,\CC)$ and  $\beta\in KK^G(\CC,\D)$.
Then  $\gamma=\beta\otimes_\D\alpha\in KK^G(\CC,\CC)$ is called  a  {\em $\gamma$-element for $G$} iff
$$\gamma\otimes 1_{C_0(X)}=1_{C_0(X)}\in KK^G(C_0(X), C_0(X))$$
for every proper $G$-space $X$.
\end{definition}

The class of groups which admit a $\gamma$-element is  huge.
It has been shown by Kasparov in \cites{K2, K3} 
that it contains all almost connected groups
(i.e., groups with co-compact connected component of the identity) and it is clear that the
existence of a $\gamma$-elements passes to closed subgroups.
In \cites{KS1, KS2} Kasparov and Skandalis proved existence of  $\gamma$-elements
for  many other groups. Note that the above definition of a $\gamma$-element 
is slightly weaker than Kasparov's original 
definition, in which he requires that $\gamma\otimes 1_{C_0(X)}=1_{C_0(X)}$ in the 
$X\rtimes G$-equivariant group $KK^{X\rtimes G}(C_0(X), C_0(X))$, where $X\rtimes G$ 
denotes the transformation groupoid for the $G$-space $X$. Since the above definition 
suffices for our purposes and since we want to avoid to talk about equivariant $KK$-theory 
for groupoids, we use it here. We have:

\begin{theorem}[Kasparov]\label{thm-gamma}
Suppose that $G$ is a second countable group which admits a $\gamma$-element. 
Then for every $G$-$C^*$-algebra $B$ the Baum-Connes assembly map
$$\mu_{(G,B)}:K_*^G(\EG;B)\to K_*(B\rtimes_rG)$$
is split injective (the same holds true for the full assembly map $\mu_{(G,B)}^{full}$).
\end{theorem}
\begin{proof} Going back to diagram (\ref{eq-DD}), we see that for proving split injectivity it suffices 
to show that the composition of the left vertical arrows of the diagram is the identity map.
So we need to check that the map
$\cdot\otimes \gamma: K_*^G(\EG;B)\to  K_*^G(\EG;B)$, which is given on the level of 
any $G$-compact subset $X\subseteq \EG$ by the map
$$KK_*^G(C_0(X), B)\to KK_*^G(C_0(X),B); x\mapsto x\otimes_B(1_B\otimes \gamma),$$
is the identity on $KK_*^G(C_0(X), B)$.
But by commutativity of the Kasparov product over $\CC$ we get
\begin{align*}
x\otimes(1_B\otimes\gamma)&=x\otimes_\CC \gamma=\gamma\otimes_\CC x
=(\gamma\otimes 1_{C_0(X)})\otimes_{C_0(X)}x
=1_{C_0(X)}\otimes_{C_0(X)}x=x.
\end{align*}
\end{proof}

The above proof relies heavily on  Theorem \ref{thm-proper-Green-Julg}, which in turn
relies on Theorem \ref{thm-ind2}. In the course of proving that theorem in \cite{CE2} the 
authors made heavy use of Kasparov's  result that all almost connected groups 
have a $\gamma$-element and that this implies (without using BC for proper coefficients)
that the Baum-Connes assembly map is injective whenever $G$ has a $\gamma$-element
in the stronger sense of Kasparov. 

\begin{remark}\label{rem-Novikov}
It is shown by Kasparov in \cites{K2, K3} that for a discrete group 
$G$ the rational injectivity 
(i.e., injectivity after tensoring both sides with $\QQ$) of the assembly map
$$\mu_G: K_*^G(\EG;\CC)\to K_*(C_r^*(G))$$
implies the famous Novikov-conjecture from topology. We do not want to discuss this 
conjecture here (e.g., see \cite{K3} for the formulation), but we want to mention 
that Theorem \ref{thm-gamma} shows that every group which admits a $\gamma$-element
also satisfies the Novikov conjecture. This fact leads to the following notation:
A group $G$ is said to satisfy the {\em strong Novikov conjecture with coefficients}, 
if the assembly map $\mu_{(G,B)}$ is injective for every $G$-$C^*$-algebras $B$.
\end{remark}

\begin{remark}\label{rem-DD-surjective}
If $G$ has a $\gamma$-element in the sense of Definition \ref{def-gamma} and if 
$B$ is any given $G$-$C^*$-algebra, then the assembly map 
$\mu_{(G,B)}: K_*^G(\EG;B)\to K_*(B\rtimes_rG)$ is surjective if and only 
if the map
$$F_\gamma: K_*(B\rtimes_rG)\to K_*(B\rtimes_rG); x\mapsto x\otimes_{B\rtimes_rG}J_G(1_B\otimes \gamma)$$
coincides with  the identity map. This follows easily from the proof of Theorem \ref{thm-gamma}
together with diagram (\ref{eq-DD}). Indeed, more generally, we may conclude from the lower square 
of diagram (\ref{eq-DD})  that every element in the image of $F_\gamma$ lies in the image of the assembly map,
and then the outer rectangle of (\ref{eq-DD}) implies that we actually have 
$$\mu_{(G,B)}(K_*^G(\EG;B))=F_\gamma(K_*(B\rtimes_rG)).$$
Moreover, it follows also from (\ref{eq-DD}) that $F_\gamma$ is  idempotent, so it is surjective 
if and only if it is the identity.
Kasparov calls $F_\gamma(K_*(B\rtimes_rG))$ the {\em $\gamma$-part of $K_*(B\rtimes_rG)$}. 

So one strategy for proving the Baum-Connes conjecture for given coefficients is to show  
that $F_\gamma$ is the identity on $K_*(B\rtimes_rG)$. This method has been used quite effectively 
by Lafforgue in \cite{Laf} for proving the Baum-Connes conjecture with trivial coefficients 
for a large class of groups (including all real or $p$-adic reductive linear groups). 
For doing this he first introduced a Banach version of $KK$-theory in oder to show
that the $\gamma$-element induces an isomorphism in $K$-theory of  certain 
Banach algebras $\mathcal S(G)$, which can be viewed as algebras of Schwartz-functions, 
which admit an embedding as dense subalgebras  of $C_r^*(G)$ such that the inclusion 
$\mathcal S(G)\into C_r^*(G)$ induces an isomorphism in $K$-theory. As a result, the 
map $F_\gamma$ is the identity on $K_*(C_r^*(G))$ which proves BC.

Extending his methods, Lafforgue later showed that all Gromov-hyperbolic groups satisfy
the Baum-Connes conjecture with coefficients (see \cites{Laf-hyp, Puschnigg}).
\end{remark}

We close this section with a short argument of how Connes's Thom isomorphism for the 
$K$-theory of crossed products by $\RR$ and 
the Pimsner-Voiculescu sequence for the $K$-theory of crossed products by $\ZZ$ can be deduced 
quite easily 
from the Dirac-dual Dirac method for actions of $\RR$ as worked out in Example \ref{ex-BC-ZZ}.

\begin{corollary}[Connes's Thom isomorphism]\label{cor-Thm}
Let $\alpha:\RR\to\Aut(A)$ be an action of $\RR$ on the $C^*$-algebra $A$. Then the crossed product 
$A\rtimes_{\alpha}\RR$ is $KK$-equivalent to $A\hat\otimes \Cl_1$. In particular, there is a canonical 
isomorphism $$K_*(A\rtimes_\alpha \RR)\cong K_*(A\hat\otimes \Cl_1)=K_{*+1}(A).$$
\end{corollary}
\begin{proof} To construct the $KK$-equivalence, let $\beta\in KK^{\RR}(\CC, C_0(\RR)\otimes \Cl_1)$ be as in Example \ref{ex-BC-ZZ}.
Then $1_A\otimes \beta\in KK^{\RR}(A, A\hat\otimes C_0(\RR)\hat\otimes \Cl_1)$ is an $\RR$-equivariant $KK$-equivalence between
$A$ and $A\hat\otimes C_0(\RR)\hat\otimes \Cl_1$ and its descent $J_{\RR}(1_A\otimes \beta)\in KK(A\rtimes_{\alpha}\RR, 
(A\hat\otimes C_0(\RR)\hat\otimes \Cl_1)\rtimes_{\alpha\otimes \tau\otimes \id_{\Cl_1}}\RR) $ is a $KK$-equivalence as well.
But $(A\hat\otimes C_0(\RR)\hat\otimes \Cl_1)\rtimes_{\alpha\otimes \tau\otimes \id_{\Cl_1}}\RR$ is isomorphic to $
A\hat\otimes \K(L^2(\RR))\hat\otimes \Cl_1$
by an application of \cite{CroPro}*{Lemma 2.4.1} and  \cite{CroPro}*{Example 2.6.6} (2). This finishes the proof.
\end{proof}

\begin{theorem}[Pimsner-Voiculescu]
Let $\alpha$ be a fixed automorphism of the $C^*$-algebra $A$ and let $n\mapsto \alpha^n$ be the corresponding action 
of $\ZZ$ on $A$. Then there is a six-term exact sequence
$$
\begin{CD} 
K_0(A) @>\id-\alpha_*>> K_0(A) @>\iota_* >> K_0(A\rtimes_\alpha\ZZ)\\
@A\partial AA @.   @VV\partial V\\
K_1(A\rtimes_{\alpha}\ZZ) @<<{\iota_*}< K_1(A)   @<<\id-\alpha_* <   K_1(A)
\end{CD}
$$
where $\iota:A\to A\rtimes_{\alpha}\ZZ$ denotes the canonical inclusion.
\end{theorem}
\begin{proof}[Scetch of proof]
By Green's theorem \cite{CroPro}*{Theorem 2.6.4} the crossed product $A\rtimes_\alpha \ZZ$ 
is Morita equivalent to the crossed product $\Ind_{\ZZ}^{\RR}A\rtimes_{\Ind\alpha}\RR$ where the 
induced algebra $\Ind_{\ZZ}^{\RR}A$ is isomorphic to the mapping cone
$C_\alpha(A)=\{f:[0,1]\to A: f(0)=\alpha(f(1))\}$. Thus, by  Connes's Thom isomophism, we get
$$K_*(A\rtimes_{\alpha}\ZZ)\cong K_*(C_\alpha(A)\rtimes \RR)\cong K_{*+1}(C_\alpha(A)).$$
The mapping cone $C_\alpha(A)$ fits into a canonical short exact sequence 
$$0\to C_0(0,1)\otimes A\to C_\alpha(A)\to A\to 0,$$
where the quotient map is given by evaluation at $1$, say.
This gives the six-term exact sequence
$$
\begin{CD}
K_0(C_0(0,1)\otimes A) @>>> K_0(C_\alpha(A)) @>>> K_0(A)\\
@AAA  @. @VVV\\
K_1(A)  @<<< K_1(C_\alpha(A)) @<<< K_1(C_0(0,1)\otimes A).
\end{CD}
$$
Using  $K_*( C_0(0,1)\otimes A)\cong K_{*+1}(A)$
 and $K_{*+1}(C_\alpha(A))\cong K_*(A\rtimes_{\alpha}\ZZ)$ this turns into the six-term sequence of the theorem.
 (However, it is not completely trivial to check that the maps in the sequence coincide with the ones 
 given in the theorem).
 \end{proof}

The above method of proof of the Pimsner-Voiculescu theorem is taken from Blackadar's book \cite{Bla86}. 
The original proof of Pimsner and Voiculescu in \cite{PV} was independent of Connes's Thom isomorphism and 
used a certain Toeplitz extension of $A\rtimes_\alpha\ZZ$.

\subsection{The Baum-Connes conjecture for group extensions}\label{sec-extensions}
%
Suppose that $N$ is a closed normal subgroup of the second countable 
locally compact group $G$. Then, if $A$ is a $G$-$C^*$-algebra, we would like to 
relate the Baum-Connes conjecture for $G$ to the Baum-Connes conjecture 
for $N$ and $G/N$. In order to do so, we first need to
write the 
crossed product $A\rtimes_rG$ as an iterated crossed product 
$(A\rtimes_rN)\rtimes_rG/N$ for a suitable action of $G/N$ on $A\rtimes_rN$.
Unfortunately, this is not possible in general if we restrict ourselves to 
ordinary actions, but it can be done by using twisted actions as discussed in  \cite{CroPro}*{Section 2.8}.
%
In what follows we shall simply write $A\rtimes_rG/N$ for the reduced crossed product
of a twisted action of the pair $(G,N)$ in the sense of Green. We then get the desired 
isomorphism 
$$A\rtimes_rG\cong (A\rtimes_rN)\rtimes_rG/N$$
(and similarly for the full crossed products). 
Recall that by  \cite{CroPro}*{Theorem 2.8.9} every  Green-twisted $(G,N)$-action 
 is equivariantly Morita equivalent to an ordinary action of $G/N$,
which allows to cheaply extend many results known for ordinary crossed products to 
the twisted case. 
In \cite{CE1} the authors extended the Baum-Connes assembly map to the category of twisted 
$(G,N)$-actions, and they constructed a partial assembly map
\begin{equation}\label{eq-partial-ass}
\mu_{(N,B)}^{(G,N)}:K_*^G(\EG; B)\to K_*^{G/N}(\underline{E(G/N)}, B\rtimes_rN)
\end{equation}
such that the following diagram commutes
$$
\begin{CD}
K_*^G(\EG;B)   @>\mu_{(G,B)} >> K_*(B\rtimes_rG)\\
@V\mu_{(N,B)}^{(G,N)}VV   @VV\cong V\\
K_*^{G/N}(\underline{E(G/N)}, B\rtimes_rN) @>>\mu_{(G/N, B\rtimes_rN)} > K_*((B\rtimes_rN)\rtimes_rG/N)
\end{CD}
$$
As a consequence, if the partial assembly map (\ref{eq-partial-ass}) is an isomorphism, then 
$G$ satisfies BC for $B$ if and only if $G/N$ satisfies BC for $B\rtimes_rN$. Using these ideas,
the following result has been shown in \cite{CEO}*{Theorem 2.1} extending some earlier results 
of \cites{CE1, CE2, O}:

\begin{theorem}\label{thm-ext}
Suppose that $N$ is a closed normal subgroup of the second countable locally compact group $G$ 
and let $B$ be a  $G$-$C^*$-algebra. Assume further, that the following condition (A) holds:
\begin{enumerate}
\item[(A)] Every closed subgroup $L\subseteq G$ such that $N\subseteq L$ and $L/N$ is compact satisfies the Baum-Connes conjecture for $B$.
\end{enumerate}
Then $G$ satisfies BC for $G$ if and only if $G/N$ satisfies BC  for $B\rtimes_rN$.
\end{theorem} 

Of course, the idea is that one should show that condition (A) implies that the partial assembly map (\ref{eq-partial-ass}) is an isomorphism.
This has been  the approach in \cites{CE1, CE2}, but in \cite{CEO} a slightly different version of the partial assembly map has been used
instead. Since every compact extension $N\subseteq L$ of an a-$T$-menable group $N$ is a-$T$-menable (see  \cite{CCJJV}), and since 
every a-$T$-menable group satisfies the Baum-Connes conjecture with coefficients,
we get the following corollary:

\begin{corollary}\label{cor-ext}
Suppose that $N$ is a closed normal subgroup of the second countable locally compact group $G$ such that $N$ is 
a-$T$-menable. Suppose further that $G$ is any $G$-$C^*$-algebra. Then $G$ satisfies BC for $B$ if and only if 
$G/N$ satisfies BC for $B\rtimes_rN$. In particular, if 
$$1\to N\to G\to G/N\to 1$$
is a short exact sequence of second countable groups such that $G$ and $G/N$ are both a-$T$-menable, then 
$G$ satisfies the Baum-Connes conjecture with coeficients.
\end{corollary} 

Note that it is not true in general that $G$ is a-$T$-menable if $N$ and $G/N$ are a-$T$-menable. For counter examples 
see \cite{CCJJV}.

Theorem \ref{thm-ext} has been used extensively in \cite{CEN} to give the proof of the original Connes-Kasparov conjecture, which 
is equivalent to the Baum-Connes conjecture with trivial coefficients for the class of all (second countable) 
almost connected groups. The basic idea goes as follows: If $G$ is any almost connected group, then one can use structure 
theory for such groups to see that there exists an amenable normal subgroup of $N$ of $G$ such that $G/N$ is a reductive Lie-group.
Since amenable groups are also a-$T$-menable, we can apply Corollary \ref{cor-ext}
to see that $G$ satisfies BC with trivial coefficients if and only $G/N$ satisfies BC with coefficient $C_r^*(N)$.
Now by Lafforgue's results we know that the reductive group $G/N$ satisfies BC with trivial coefficients and we 
somehow need to find good arguments which give us BC for the coefficient algebra $C_r^*(N)$ instead. It is this point 
where the arguments become quite complicated and we refer to \cite{CEN} for the details of the proof.

\section{The Going-Down   (or restriction) principle and applications}\label{sec-rest}
\subsection{The Going-Down principle}\label{subsec-DG}
In this section we want to discuss an application of the Baum-Connes conjecture which 
helps, among other things,  to give explicit $K$-theory computations in some interesting cases.
Assume we have two $G$-$C^*$-algebras $A$ and $B$  and a $G$-equivariant 
$*$-homomorphism  $\phi:A\to B$. This map descents 
to a map 
$$\phi\rtimes_rG: A\rtimes_rG\to B\rtimes_rG.$$
Suppose  we want to prove that this map induces an isomorphism in $K$-theory.
If $G$ satisfies the Baum-Connes conjecture for $A$ and $B$, this problem is equivalent
to the problem that the map
$$\phi_*: K_*^G(\EG;A)\to K_*^G(\EG;B)$$
is an isomorphism (use Exercise \ref{ex-commute}). 
The restriction (or Going-Down) principle allows us to deduce the isomorphism on the 
level of topological $K$-theory from the behaviour on compact subgroups of $G$. 
Let us state the theorem:

\begin{theorem}[Going-Down principle]\label{thm-GD}
Suppose that $G$ is a second countable locally compact group,  $A$ and $B$ are $G$-$C^*$-algebras, and $x\in KK^G(A,B)$ such that 
for all compact subgroups $K\subseteq G$ the class
$\res_K^G(x)\in KK^K(A,B)$ induces an isomorphism
$$\cdot\otimes_A\res_K^G(x): KK_*^K(\CC, A)\stackrel{\cong}{\to} KK_*^K(\CC, B).$$
Then the map
$$\cdot\otimes_A x: K_*^G(\EG;A)\to K_*^G(\EG;B),$$
which is given on the level of $KK_*^G(C_0(X),A)$ for a $G$-compact $X\subseteq \EG$ by 
Kasparov product with $x$, is an isomorphism. As a consequence, if $G$ satisfies the 
Baum-Connes conjecture for $A$ and $B$, then the class $x$ induces an isomorphism
$$\cdot\otimes_{A\rtimes_rG} J_G(x):K_*(A\rtimes_rG)\stackrel{\cong}{\longrightarrow} K_*(B\rtimes_rG).$$
\end{theorem}

\begin{remark}\label{rem-compact-subgroup} 
There are many interesting groups $G$ for which the trivial subgroup is the only compact subgroup
(e.g., $G=\RR^n, \ZZ^n$ or the free group $F_n$ in $n$ generators). For those groups the condition 
on compact subgroups in the theorem reduces to the single condition that 
$$\cdot\otimes_A x: K_*(A)\to K_*(B)$$ 
is an isomorphism. In many applications, this condition comes for free. 
\end{remark}

\begin{remark}\label{remark-GD}
Instead of asking that $\cdot\otimes_A\res_K^G(x): KK_*^K(\CC, A)\stackrel{\cong}{\to} KK_*^K(\CC, B)$ 
is an isomorphism for all compact subgroups $K$ of $G$, we could alternatively require that 
$$\cdot\otimes_{A\rtimes K}J_K(\res_K^G(x)): K_*(A\rtimes K)\to K_*(B\rtimes K)$$
is an isomorphism for all such $K$. This follows from the commutativity of the diagram
$$
\begin{CD}
KK_*^K(\CC, A)  @> \cdot\otimes_A\res_K^G(x) >> KK_*^K(\CC, B)\\
@V\mu_{(K,A)}VV   @VV\mu_{(K,B)}V\\
K_*(A\rtimes K) @>>\cdot\otimes_{A\rtimes K}J_K(\res_K^G(x)) >K_*(B\rtimes K),
\end{CD}
$$
in which the vertical arrows are the isomorphisms of the Green-Julg theorem (see Example \ref{ex-Green-Julg}).
\end{remark}

The proof of Theorem \ref{thm-GD} in the above version is given in  \cite{ELPW}*{Proposition 1.6.}, but 
it relies very heavily on a more general Going-Down principle obtained by Chabert, Echterhoff and Oyono-Oyono 
in \cite{CEO}*{Theorem 1.5}. In that paper we also show how Theorem \ref{thm-ext} on the Baum-Connes conjecture 
for group extensions can be obtained as a consequence of the (more general)
Going-Down principle. In what follows below we shall present the proof in the case where $G$ is discrete. In this case 
the proof becomes much easier, but still reveals the basic ideas. Note that most of the relevant details for the 
discrete case first appeared (in the language of $E$-theory) in \cite{GHT}*{Chapter 12}.

If $G$ is discrete, then the first observation is that each $G$-compact proper $G$-space 
maps continuously and $G$-equivariantly into the geometric realisation of a $G$-finite $G$-simplicial complex.
For this let $F\subseteq G$ be any finite subset of $G$ which contains the identity of $G$. We then define 
$$\mathcal M_F=\big\{f\in C_c^+(G):  \sum_{g\in G}f(g)=1\;\text{and $\forall g,h\in \supp(f): g^{-1}h\in F$}\big\}.$$
Then $\mathcal M_F$ is the geometric realisation of  a locally finite simplicial complex 
with vertices $\{g: g\in G\}$ and $\{g_1,\ldots, g_n\}$ is an $n$-simplex if and only if 
$g_i\neq g_j$ for $i\neq j$ and $g_i^{-1}g_j\in F$ for all
$1\leq i,j\leq n$. It follows directly from the definition that for any simplex $\{g_1,\ldots, g_n\}$ we have
$g_1^{-1}\{g_1,\ldots, g_n\}\subseteq F$, hence $\mathcal M_F$ is $G$-finite in the sense that
there exists a finite set $\mathcal S$ of simplices such that 
every other simplex is a translate of one in $\mathcal S$. Note that this implies that for all $f\in \mathcal M_F$
 the formula
\begin{equation}\label{eq-cut-M_F}
1=\sum_{g\in G} f(g)=\sum_{g\in G}g\cdot f(e)
\end{equation}
holds. 
Note also that if $F\subseteq F'$ for some finite set $F'\subseteq G$, then there is a canonical inclusion 
$\mathcal M_F\subseteq \mathcal M_{F'}$. With this we get:

\begin{lemma}\label{lem-simplicial}
Suppose that $G$ is a discrete group and let $X$ be a $G$-compact proper $G$-space.
Suppose further that $c:X\to [0,1]$ is a cut-off function for $X$ as in Section \ref{subsec-assembly}.
Then there exists a finite subset $F\subseteq G$ such that $g(\supp(c))\cap \supp(c)=\emptyset$ 
for all $g\notin F$ and a continuous $G$-map
$$\varphi_c:X\to \mathcal M_F; \varphi_c(x)=[g\mapsto c^2(g^{-1}x)].$$
Moreover, for any other continuous $G$-map $\psi:X\to \mathcal M_F$ there is a finite set $F'\subseteq G$ containing $F$
such that  $\psi$ is $G$-homotopic to $\varphi$ in $\mathcal M_{F'}$.
\end{lemma}
\begin{proof} Existence of the finite set $F$ as in the lemma follows from compactness of 
the set $\{(g, x): (gx, x)\in \supp(c)\times \supp(c)\}\subseteq G\times X$.  It is compact since $G$ acts
properly on  $X$.
It is then straightforward to check that $\varphi_c$ is a continuous $G$-map. 
Suppose now that $\psi:X\to \mathcal M_F$ is any other continuous $G$-map. We define
$\tilde{c}:X\to [0,1]$ as $\tilde{c}(x):=\sqrt{\psi(x)(e)}$. It follows then from (\ref{eq-cut-M_F})
that $\tilde{c}$ is a cut-off function as well and that $\psi=\varphi_{\tilde{c}}$. Now let 
$d:X\times [0,1]\to [0,1]$ be given by $d(x,t):= \sqrt{(1-t)c^2(x)+ t\tilde{c}^2(x)}$.
Then there exists a finite set $F'\supseteq F$ such that $g(\supp(d))\cap \supp(d)=\emptyset$ for all
$g\notin F'$. The continuous $G$-map $\varphi_d: X\times [0,1]\to \mathcal M_{F'}$ then evaluates 
to $\varphi_c$ at $t=0$ (using $\mathcal M_F\subseteq \mathcal M_{F'}$)
and  to $\varphi_{\tilde{c}}=\psi$ at $t=1$.
\end{proof}

\begin{lemma}\label{lem-simplicial1}
Let $G$ be a discrete group. Then for every $G$-$C^*$-algebra $A$ we have
$$K_*^G(\EG;A)=\lim_{F} KK_*^G(C_0(\mathcal M_F),A),$$
where $F$ runs through all finite subsets of $G$ and the limit is taken 
with respect to the canonical  inclusion $\mathcal M_F\subseteq \mathcal M_{F'}$ if $F\subseteq F'$.
\end{lemma}
\begin{proof}
This follows  from the definition 
$$K_*^G(\EG;A)=\lim_{X}KK_*^G(C_0(X), A),$$ where 
$X$ runs through the $G$-compact subsets of $\underline{EG}$ and Lemma \ref{lem-simplicial}: By the universal 
property of $\underline{EG}$ there are $G$-compact subsets $X_F\subseteq \underline{EG}$ and 
$G$-continuous maps $\mathcal M_F\to X_F\subseteq \underline{EG}$, which, up to a 
possible enlargement of $X_F$, are unique up to $G$-homotopy. On the other hand, 
Lemma \ref{lem-simplicial} provides maps $X_F\to \mathcal M_{F'}$ for some $F'\supseteq F$ which, 
up to passing to a bigger set $F''$ if necessary, is also unique up to $G$-homotopy. Thus we 
get a zigzag diagram
{\tiny$$ \xymatrix{&KK_*^G(C_0(\mathcal M_F),A)   \ar[d]  \ar[r] & KK_*^G(C_0(\mathcal M_{F'}),A)   \ar[d]  \ar[r] &  KK_*^G(C_0(\mathcal M_{F''}),A).....\\
KK^G_*(C_0(X), A)  \ar[ur]  \ar[r]&    KK_*^G(C_0( X_F),A)   \ar[ur]  \ar[r] & KK_*^G(C_0(X _{F'}),A)   \ar[ur]  \ar[r] & .....
}
$$}
which commutes sufficiently well to induce an isomorphism of the inductive limits.
\end{proof}

The next lemma gives the crucial point in the proof Theorem \ref{thm-GD} in case of discrete $G$.
It has first been shown in the setting of $E$-theory in \cite{GHT}*{Lemma 12.11}. A more general version 
for arbitrary open subgroups $H$ of a second countable locally compact group $G$ has been shown
in \cite{CE2}*{Proposition 5.14}.

\begin{lemma}\label{lem-compression}
Suppose that $K\subseteq G$ is a finite subgroup of the discrete group $G$. Then,
for every $G$-$C^*$-algebra $B$, 
there is a well defined {\em compression isomorphism}
$$\comp_K: KK^G_*(C_0(G/K), B)\to KK_*^K(\CC, B)$$
given as the composition of the maps
$$KK^G_*(C_0(G/K), B) \stackrel{\res_K^G}{\longrightarrow} KK_*^K(C_0(G/K), B)\stackrel{\iota^*}{\longrightarrow}
KK_*^K(\CC, B),$$
where $\iota:\CC\into C_0(G/K)$ denotes the inclusion $\lambda\mapsto \lambda \delta_{eK}$ with 
$\delta_{eK}$  the characteristic function of the open one-point set $\{eK\}\subseteq G/K$.
\end{lemma} 

\begin{proof}We  construct an inverse 
$$\ind_K^G: KK_*^K(\CC, B)\to KK^G_*(C_0(G/K), B) $$
for the compression map. We may restrict ourselves 
to the case of the $K_0$-groups, the $K_1$-case then follows from passing from 
 $B$ to $B\otimes C_0(\RR)$.
Let  $(\E, \one, \gamma, T)$ be a representative for a class $x\in KK^K(\CC,B)$ where
$T\in \L(\E)$ 
is a $K$-invariant operator  such that $T^*-T, T^2-1\in \K(\E)$. We then define a Hilbert  $B$-module $\ind_K^G\E$ as 
$$\Ind_K^G\E=\left\{ \xi:G\to\E: \;\begin{matrix}\text{s.t.}\;\xi(gk)=\gamma_{k^{-1}}(\xi(g))\;\text{for all} \,g\in G, k\in K\\
 \;\text{and}\;\sum_{g\in G}\beta_g(\lk \xi(g), \xi(g)\rk_B) <\infty\end{matrix}\right\}$$
where $\sum_{g\in G}\beta_g(\lk \xi(g), \xi(g)\rk) <\infty$ just means that the sum converges in the  norm-topology of $B$.
The grading on $\Ind_K^G\E$ is given by the grading of $\E$ applied pointwise to the elements of $\Ind_K^G\E$.
We define the $B$-valued inner product  and the right $B$-action 
 on $\Ind_K^G\E$ by 
$$
\lk \xi,\eta\rk_B=\frac{1}{|K|}\sum_{g\in G}\beta_g(\lk \xi(g), \eta(g)\rk_B)\quad\text{and}\quad (\xi\cdot b)(g)=\xi(g)\beta_{g^{-1}}(b)
$$
for all $\xi,\eta\in \Ind_K^G\E$, $b\in B$ and $g\in G$. Moreover, we define a $*$-homomorphism
$$M:C_0(G/K)\to \L(\Ind_K^G\E); (M(f)\xi)(g):=f(gK)\xi(g)$$
and an operator $\tilde{T}\in \L(\Ind_K^G\E)$ by $(\tilde{T}\xi)(g)=T\xi(g)$. Finally, the $G$-action
$\Ind\gamma:G\to \Aut(\Ind_K^G\E)$ is given by $(\Ind\gamma_g\xi)(h)=\xi(g^{-1}h)$ for $g,h\in G$.

It is then an easy exercise to check that
$(\Ind_K^G\E, M, \Ind\gamma, \tilde T)$ is a $G$-equivariant $C_0(G/K)-B$ Kasparov cycle such that 
$$\comp_K\big([\Ind_K^G\E, M, \Ind\gamma, \tilde T]\big)=[\E, \one, \gamma, T].$$
For the converse, averaging over $K$, we may first assume that for a given 
class $x=[\F, \Phi, \nu, S]\in KK^G(C_0(G/K), B)$ the operator $S$ is $K$-invariant and 
that $\Phi:C_0(G/K)\to \L(\F)$ is non-degenerate. 
Let $\tilde{S}=\sum_{gK\in G/K} \Phi(\delta_{gK})S\Phi(\delta_{gK})$. 
We claim that $\tilde{S}$ is a compact perturbation of $S$, i.e., 
$$(S-\tilde{S})\Phi(f)=\sum_{gK\in G/K} (S-\Phi(\delta_{gK})S)\Phi(\delta_{gK})f(gK)\in \K(\F)$$ for all
$f\in C_0(G/K)$. For this we first observe that, since $[S,\Phi(\delta_{gK})]\in \K(\F)$ for all $gK\in G/K$, each summand lies in $\K(\F)$.
Moreover, since $f\in C_0(G/K)$, the sum converges in norm, which proves the claim.
Thus, replacing $S$ by $\tilde{S}$ if necessary, we may assume that $[S,\Phi(f)]=0$ for all $f\in C_0(G/K)$. In particular,
if $p:=\Phi(\delta_{eK})$, it follows that $S=pSp+(1-p)S(1-p)$.

The class $\comp_K(x)$ is  represented by the $KK$-cycle 
$[\F, \Phi|_{\CC\delta_{eK}}, \nu|_K, S]$. 
For $p=\Phi(\delta_{eK})$,  let $\E:=p\F$, $T=pSp$ and $\gamma:K\to \Aut(\E)$ be the restriction of $\nu|_K$ to the summand $\E$ of $\F$. 
Since $S=pSp+(1-p)S(1-p)$ and since $[(1-p)\F, \Phi|_{\CC\delta_{eK}}, \nu|_K, (1-p)S(1-p)]$
is degenerate, we see that 
$\comp_K(x)=[\E, \one, \gamma, T]$.  It is then straightforward to check that 
$$U: \Ind_K^G\E\to \F; U\xi=\frac{1}{|K|} \sum_{g\in G} \nu_g(\xi(g)) $$
is a an isomorphism of Hilbert-$B$-modules which induces an isomorphism between the $KK$-cycles
$(\Ind_K^G\E, M, \Ind\gamma, \tilde{T})$ and $(\F, \Phi, \nu, S)$. This finishes the proof. \end{proof}

Suppose now that $X$ is a proper $G$-space, $U\subseteq X$ is an open $G$-invariant subset of $X$,  and  $Y:=X\smallsetminus U$.
Since $C_0(X)$ is nuclear, there exists a completely positive contractive 
section $\Phi: C_0(Y)\to C_0(X)$ for the restriction homomorphism $\res_Y: C_0(X)\mapsto C_0(Y): f\mapsto f|_Y$.
By properness of the action, we may average $\Phi$ with the help of a cut-off function $c:X\to [0,\infty)$ to get 
the $G$-equivariant completely positive and contractive section
$$\Phi^G(\varphi)(x):=\int_G c^2(g^{-1}x) \Phi(\varphi)(x)\,dg$$
for $\varphi\in C_0(Y)$. It follows then from Theorem \ref{thm-sixterm} that for every $G$-algebra $B$ there exists 
a six-term exact sequence 
$$
\begin{CD} KK_0^G(C_0(Y), B)  @>\res_Y^*>> KK_0^G(C_0(X), B)  @>\iota^* >> KK_0^G(C_0(U),B)\\
@A\partial AA @. @VV\partial V\\
KK_1^G(C_0(U), B)  @<<\iota^* < KK_1^G(C_0(X), B)  @<<\res_K^*<  KK_0^G(C_0(Y),B).
\end{CD}
$$
We are now ready for 

\begin{proof}[Proof of Theorem \ref{thm-GD}  for $G$ discrete]
By Lemma \ref{lem-simplicial1} it suffices to show that for every (geometric realisation) of a
$G$-finite $G$-simplicial complex $X$, the map
$$\cdot\otimes_A x: KK_*^G(C_0(X), A)\to KK^G_*(C_0(X), B)$$
given by taking Kasparov product with the class $x\in KK^G(A,B)$ is an isomorphism.
We do the proof by induction on the dimension of $X$. Suppose first that $\dim(X)=0$.
In that case $X$ is discrete and therefore decomposes into a finite union of $G$-orbits 
$$X=G(x_1)\dot\cup G(x_2)\dot\cup \cdots \dot\cup G(x_l)$$
for suitable elements $x_1,\ldots, x_l$ in $X$. Then we have $C_0(X)\cong \bigoplus_{i=1}^l C_0(G(x_i))$ and 
$KK_*^G(C_0(X), A)=\prod_{i=1}^l KK_*^G(C_0(G(x_i)), A)$ (and similarly for $KK_*^G(C_0(X),B)$). 
Thus it suffices to show that $\cdot\otimes_A x: KK_*^G(C_0(G(x_i)), A)\stackrel{\cong}{\to} KK_G^*(C_0(G(x_i)), B)$
for all $1\leq i\leq l$. But $G(x_i)\cong G/G_{x_i}$ as a $G$-space, where $G_{x_i}=\{g\in G: gx_i=x_i\}$ denotes the
stabiliser of $x_i$. By properness, we have $G_{x_i}$  finite  for all $i$. We then get a commutative diagram
$$
\begin{CD}
KK_*^G(C_0(G/G_{x_i}), A)  @> \cdot\otimes_A x >> KK_*^G(C_0(G/G_{x_i}), B)\\
@V \comp_{G_{x_i}} VV   @VV \comp_{G_{x_i}} V\\
KK_*^{G_{x_i}}(\CC, A)  @> \cdot\otimes_A \res_{G_{x_i}}^G(x) >> KK_*^{G_{x_i}}(\CC, B)
\end{CD}
$$
in which all vertical arrows are isomorphisms by Lemma \ref{lem-compression} and the lower horizontal arrow 
is an isomorphism by the assumption of the theorem. Hence the upper horzontal arrow is an isomorphism as well.

Suppose now that $\dim(X)=n$. After performing a baricentric subdivision of $X$, if necessary, we may assume 
that the action of $G$ on $X$ satisfies the following condition: If $\Delta$ is a simplex in $X$, then an element $g\in G$
either fixes all of $\Delta$ or $g\cdot \inte(\Delta)\cap \inte(\Delta)=\emptyset$, where $\inte(\Delta)$ denotes the interior of $\Delta$.
Now let $\widetilde{X}$ denote the union of the interiors of all $n$-dimensional simplices in $X$. Then 
$X_{n-1}:=X\smallsetminus \widetilde{X}$ is an $n-1$-dimensional $G$-simplicial complex and by the induction  assumption we 
have $KK_*^G(C_0(X_{n-1}), A)\cong KK^G_*(C_0(X_{n-1}), B)$ via taking  Kasparov product with $x$. 
We now show that $KK_*^G(C_0(\widetilde{X}), A)\cong KK^G_*(C_0(\widetilde{X}), B)$ as well. If this is done, then the five-Lemma 
applied to the  diagram 
{\small
\begin{align*}
\begin{CD}
KK_{*-1}^G(C_0(\widetilde{X}), A) @>\partial >> KK_*^G(C_0(X_{n-1}), A) @>\res^*>> KK_*^G(C_0(X), A) \\
 @V\cdot\otimes_Ax VV  @V\cdot\otimes_Ax VV   @V\cdot\otimes_Ax VV   \\
 KK_{*-1}^G(C_0(\widetilde{X}), B) @>\partial >> KK_*^G(C_0(X_{n-1}), B) @>\res^*>> KK_*^G(C_0(X), B)  
 \end{CD}\quad\quad\quad\quad\quad\quad\\
\quad\quad\quad\quad\quad\quad \begin{CD}
 @>\iota^*>>  KK_{*}^G(C_0(\widetilde{X}), A) @>\partial >> KK_{*+1}^G(C_0(X_{n-1}), A)\\
 @. @VV\cdot\otimes_Ax V  @VV\cdot\otimes_Ax V\\
 @>\iota^*>> 
 KK_{*}^G(C_0(\widetilde{X}), B) @>\partial >> KK_{*+1}^G(C_0(X_{n-1}), B)\\
 \end{CD}
\end{align*}}
shows that $KK_*^G(C_0(X), A)\cong KK_G^*(C_0(X), B)$. 
 
 To see that $KK_*^G(C_0(\widetilde{X}), A)\cong KK^G_*(C_0(\widetilde{X}), B)$ we first observe that $\widetilde{X}$ 
 is a finite union of orbits of open simplices $\inte(\Delta_1), \ldots, \inte(\Delta_k)$ for some $k\in \NN$. 
 Via the corresponding product decomposition of the $KK$-groups, we may then assume that $\widetilde{X}=G\cdot \inte(\Delta)$ 
 for a single open $n$-simplex $\Delta$. By our assumption on the action of $G$ on $X$, we have 
 $$G\cdot\inte(\Delta)\cong G/G_\Delta\times \inte(\Delta)$$
 where $G_{\Delta}=\{g\in G: g\cdot \Delta=\Delta\}$ denotes the (finite!) stabiliser of $\Delta$ and where the $G$-action 
 on $G/G_\Delta\times \inte(\Delta)$ is given by left  translation on the first factor. 
 We then get a diagram
{\small $$
 \begin{CD} 
 KK_*^G(C_0(G/G_\Delta\times \inte(\Delta)), A)
@>\operatorname{Bott} >\cong>  KK_{*+n}^G(C_0(G/G_\Delta), A) @>\comp_{G_{\Delta}} >\cong> KK_{*+n}^{G_\Delta}(\CC,A)\\
  @V\cdot\otimes_Ax VV   @V\cdot\otimes_Ax VV    @V\cdot\otimes_A\res_{G_\Delta}^G(x) VV\\
 KK_*^G(C_0(G/G_\Delta\times \inte(\Delta)), B)
@>\operatorname{Bott} >\cong>  KK_{*+n}^G(C_0(G/G_\Delta), B) @>\comp_{G_{\Delta}} >\cong> KK_{*+n}^{G_\Delta}(\CC,B).
\end{CD}
$$}
Since, by assumption,  the last vertical arrow is an isomorphism, the result follows.
\end{proof}

We should note that for groups which satisfy the strong Baum-Connes conjecture in the sense 
of Definition \ref{def-strongBC} a stronger version of 
 Theorem \ref{thm-GD} has been shown by Meyer and Nest in \cite{MN}*{Theorem 9.3}:
 
 \begin{theorem}[Meyer-Nest]\label{thm-MN}
 Suppose that the second countable group $G$ satisfies the  strong Baum-Connes conjecture (e.g., this is satisfied if $G$ 
 is a-$T$-menable or, in particular,  if $G$ is amenable)
and assume that $x\in KK^G(A,B)$ such that  for every compact subgroup $K$  of $G$ the class 
$J_K(\res_K^G(x))\in KK(A\rtimes K,B\rtimes K)$ is a $KK$-equivalence. 
Then $J_G(x)\in KK(A\rtimes_rG, B\rtimes_rG)$ is $KK$-equivalence as well.
\end{theorem}

\subsection{Applications of the Going-Down principle}\label{subsec-appGD}
We are now going to give a number of applications. 
%
%
The first one is the proof that every exact locally compact group satisfies the 
strong Novikov conjecture. Recall that a locally compact group is called exact (in the sense 
of Kirchberg and Wassermann) if for every short exact sequence of $G$-$C^*$-algebras 
$$0\to I\stackrel{\iota}{\to} A\stackrel{q}{\to} A/I\to 0$$
the corresponding sequence of reduced crossed products 
$$0\to I\rtimes_rG\stackrel{\iota\rtimes_rG}{\to} A\rtimes_rG\stackrel{q\rtimes_rG}{\to} A/I\rtimes_rG\to 0$$
is also exact. It has been known for a long time by work of Ozawa and others (see \cite{Ozawa}) that 
a discrete group is exact if and only if it admits an amenable action on a compact space $X$.
This means that the transformation groupoid $X\rtimes G$ is topologically amenable in the sense
of \cite{ADR}. Very recently the result of Ozawa has been generalised by Brodzki, Cave and Li  (\cite{BCL}) 
to second countable locally compact groups. The following result has been shown first for discrete 
$G$ by Higson in \cite{H}. The result has been extended in \cite{CEO}  to the case of
 second countable locally compact 
groups acting amenably on a compact space. 

\begin{theorem}\label{thm-exact}
Let $G$ be an exact second countable locally compact group. Then 
$G$ satisfies the strong Novikov conjecture with coefficients, i.e., 
for each 
$G$-$C^*$-algebra $B$ the assembly map 
$$\mu_{(G,B)}:K_*^G(\EG; B)\to K_*(B\rtimes_rG)$$
is split injective.
A similar statement holds true for the full assembly map $\mu_{(G,B)}^{full}$.
\end{theorem}
\begin{proof} 
By \cite{BCL}, being exact is equivalent to the condition that there exists a compact amenable $G$-space
$X$.
Following the arguments given by Higson in \cite{H}*{Lemma 3.5 and Lemma 3.6} 
we may as well assume that $X$ is a metrisable convex space and $G$ acts by 
affine transformations. In particular, $X$ is $K$-equivariantly contractible for every 
compact subgroup $K$ of $X$. 
It then follows that the inclusion map
$\iota: \CC\to C(X)$ is a $KK^K$-equivalence for every compact subgroup $K$ of $G$ -- it's 
inverse is given by the map $C(X)\to\CC; f\mapsto f(x_0)$ for any $K$-fixed point $x_0\in X$. 
It then follows, that for every $G$-$C^*$-algebra $B$, the $*$-homomorphism $B\to B\otimes C(X); b\mapsto b\otimes 1_X$
is a $KK^K$-equivalence as well. Thus it follows from Theorem \ref{thm-GD} that 
$$\Phi_*: K_*^G(\EG; B)\to K_*^G(\EG; B\otimes C(X))$$
is an isomorphism. Moreover, by Tu's extension of the Higson-Kasparov theorem to groupoids (see \cite{Tu2}), 
the assembly map 
$$\mu_{(X\rtimes G, B\otimes C(X))}:K_*^{X\rtimes G}(\underline{E(X\rtimes G)}, B\otimes C(X))\to K_*((B\otimes C(X))\rtimes_rG)$$
is an isomorphism, since $X\rtimes G$ is an amenable, and hence a-$T$-menable groupoid. 
Moreover, it has been shown in \cite{CEO1} that the {\em forgetful map}
$$F: K_*^{X\rtimes G}(\underline{E(X\rtimes G)}, B\otimes C(X))\to K_*^{G}(\EG, B\otimes C(X))$$ is an isomorphism 
and that the diagram
$$\begin{CD}
K_*^{X\rtimes G}(\underline{E(X\rtimes G)}, B\otimes C(X))  @>\mu_{(X\rtimes G, B\otimes C(X))} >> K_*((B\otimes C(X))\rtimes_rG)\\
@VF VV     @VV=V\\
K_*^{G}(\EG, B\otimes C(X))@>\mu_{(G, B\otimes C(X))} >> K_*((B\otimes C(X))\rtimes_rG)
\end{CD}
$$
commutes. The result then follows from the commutative diagram
$$\begin{CD}
K_*^{G}(\EG, B)@>\mu_{(G, B\otimes C(X))} >> K_*(B\rtimes_rG)\\
@V\Phi_* VV  @VV \Phi\rtimes_rG_* V\\
K_*^{G}(\EG, B\otimes C(X))@>\mu_{(G, B\otimes C(X))} >> K_*((B\otimes C(X))\rtimes_rG)
\end{CD}
$$
in which the left vertical arrow and the bottom horizontal arrow are isomorphisms.
\end{proof}

 The main application of the Going-Down principle
 in \cite{CEO} was the proof of a version of the 
K\"unneth formula for $K_*^G(\EG, B)$ with applications for the Baum-Connes conjecture with trivial coefficients.
We don't want to go into the details here. 
But we would like to mention some other useful applications. By a homotopy between two 
actions $\alpha^0,\alpha^1:G\to\Aut(A)$  we  understand a path of actions  $\alpha^t:G\to \Aut(A)$, $t\in [0,1]$,
such that 
$$\big(\alpha_g(f)\big)(t):=\alpha^t_g(f(t))\quad \forall f\in A[0,1], g\in G, t\in [0,1]$$
defines an action on  $A[0,1]=C([0,1],A)$. The following is of course a direct consequence of Theorem \ref{thm-GD}:

\begin{corollary}\label{cor-action}
Suppose that $\alpha:G\to \Aut(A[0,1])$ is a homotopy between the actions $\alpha^0,\alpha^1:G\to\Aut(A)$
and assume that $G$ satisfies BC for $(A[0,1],\alpha)$ and $(A, \alpha^t)$ for $t=0,1$. 
Suppose further that for  $t=0,1$ and for every compact subgroup $K$ of $G$
 the evaluation map $\eps_t:A[0,1]\to A; f\mapsto f(t)$ induces an isomorphism
 $\eps_t\rtimes K_*: K_*(A[0,1]\rtimes_\alpha K)\stackrel{\cong}{\to} K_*(A\rtimes_{\alpha^t}K)$. Then 
 $$\eps_t\rtimes_rG_*: K_*(C([0,1],A)\rtimes_{\alpha,r}G)\to K_*(A\rtimes_{\alpha^t,r}G)$$
 is an isomorphism as well. In particular, we have $K_*(A\rtimes_{\alpha^0,r}G)\cong K_*(A\rtimes_{\alpha^1,r}G)$.
  \end{corollary}

Of course, the condition on the compact subgroups in the above corollary is quite annoying. However, for those
groups which have no compact subgroups other than the trivial group, the corollary becomes very nice, since
the evaluation maps $\eps_t: A[0,1]\to A; f\mapsto f(t)$ are always $KK$-equivalences.

\begin{corollary}\label{cor-action-trivial}
Suppose that $\alpha:G\to \Aut(A[0,1])$ is a homotopy between the actions $\alpha^0,\alpha^1:G\to\Aut(A)$
and assume that $G$ satisfies BC for $(A[0,1],\alpha)$ and $(A, \alpha^0)$, $(A,\alpha^1)$. If 
 $\{e\}$ is the only compact subgroup of $G$, then 
$K_*(A\rtimes_{\alpha^0,r}G)\cong K_*(A\rtimes_{\alpha^1,r}G)$.
\end{corollary}

In \cite{ELPW}  Corollary \ref{cor-action} has been used to show that for groups $G$ which  satisfy BC for 
suitable coefficients, the $K$-theory of  reduced twisted group algebras $C_r^*(G,\omega)$, where 
$\om:G\times G\to \TT$ is a Borel $2$-cocycle on $G$, only depends 
on the homotopy class of the $2$-cocycle $\omega$ (with a suitable definition of homotopy). 
We don't want to go into the details here, but  we 
want to mention that if $(\om_t)_{t\in [0,1]}$ is such a homotopy of $2$-cocycles, it induces a 
homotopy $\alpha:G\to \Aut(\K[0,1])$ of actions of $G$ on the compact operators $\K=\K(\ell^2(\NN))$
such that $\K\rtimes_{\alpha^t,r}G\cong \K\otimes C_r^*(G,\om_t)$ for all $t\in [0,1]$.
We refer to  \cite{CroPro}*{Section 2.8.6} for a discussion of twisted group algebras.
It follows from the results in \cite{EW1} that a homotopy of actions of a compact group $K$ on $\K$
must be exterior equivalent to a constant path of actions, hence 
$$\K[0,1]\rtimes K\cong (\K\rtimes K)[0,1]$$
for all compact subgroups of $G$, from which it follows that the evaluation maps
$$\eps_t\rtimes K: \K[0,1]\rtimes K\to \K\rtimes K$$
are $KK$-equivalences for all $K$. Thus, if $G$ satisfies BC for $\K$ and $\K[0,1]$ (for the relevant actions),
it follows from Corollary \ref{cor-action} that 
$$K_*(C_r^*(G,\om_0))\cong K_*(\K\rtimes_{\alpha^0,r}G)\cong K_*(\K\rtimes_{\alpha^1,r}G)\cong K_*(C_r^*(G, \om_1)).$$
Note that this result extends earlier results of Elliott (\cite{Ell}) for the case of finitely generated abelian groups $G$
and of Packer and Raeburn \cite{PR92} for a class of solvable groups $G$. 
The main application in \cite{ELPW} was given for the computation of the $K$-theory of the 
crossed products $A_\theta\rtimes F$ of the non-commutative $2$-tori $A_{\theta}$, $\theta\in [0,1]$ 
with finite subgroups $F\subseteq \SL(2,\ZZ)$ acting canonically on $A_\theta$.
It turned out that $A_\theta\rtimes F\cong C_r^*(\ZZ^2\rtimes F, \omega_\theta)$ for some cocycles 
$\omega_\theta$ which depend continuously on the parameter $\theta$. Since $\ZZ^2\rtimes F$ is amenable 
it satisfies strong BC, and then it follows from the above results that 
$$
K_*(A_\theta\rtimes F)\cong K_*(C_r^*(\ZZ^2\rtimes F, \omega_\theta))\cong K_*(C_r^*(\ZZ^2\rtimes F, \omega_0))
=K_*(C(\TT^2)\rtimes F).
$$
The last group can be computed by methods from classical topology. We refer to \cite{ELPW} for further details 
on this. Note that in this situation we can also use Theorem \ref{thm-MN} to deduce that all algebras 
$A_\theta\rtimes F$, $\theta\in [0,1]$,  are pairwise $KK$-equivalent. 

\subsection{Crossed products by actions on totally disconnected spaces}\label{subsec-disconnected}

We now want to apply our techniques to certain crossed products of groups $G$ acting ``nicely'' on totally disconnected 
spaces $\Om$. The main application of this will be given for reduced semigroup algebras and crossed products by 
certain semigroups, which is described in detail in  \cite{AlgAct} and  \cite{SgpC}.

If $\Om$ is a totally disconnected locally compact space we denote by 
$\mathcal U_c(\Omega)$ the collection of all compact open subsets of $\Omega$.
This set is countable if and only if $\Omega$ has a countable basis of its topology, i.e., 
$\Om$ is second countable. 
For any set $\mathcal  V\subseteq \mathcal U_c(\Om)$
we say that $\mathcal V$ {\em generates $\mathcal U_c(\Omega)$}, if every 
subset $\mathcal U\subseteq \mathcal U_c(\Om)$ which contains $\V$  
and  is closed under finite intersections, finite unions, and taking differences $U\smallsetminus W$ with $U,W\in \mathcal U$,
must coincide with $\U_c(\Om)$.

Let $C_c^\infty(\Om)$ denotes the dense subalgebra of $C_0(\Om)$ consisting 
of locally constant functions with compact supports on $\Omega$.  Then 
$$C_c^{\infty}(\Om)=\spn\{1_U: U\in \mathcal U_c(\Om)\},$$
where $1_U$ denotes the indicator function of $U\subseteq \Om$.
The straight-forward proof of the following lemma is given in \cite{CEL2}*{Lemma 2.2}:

\begin{lemma}
Suppose that $\mathcal V$ is a family of compact open subsets of the totally disconnected locally compact space $\Omega$. 
Then the following are equivalent:
\begin{enumerate}
\item The set $\{1_V: V\in \mathcal V\}$ of characteristic functions of the elements in $\mathcal V$ generates $C_0(\Om)$ as a $C^*$-algebra.
\item The set $\mathcal V$ generates $\mathcal U_c(\Om)$ in the sense explained above.
\end{enumerate}
If, in addition, $\mathcal V$ is closed under finite intersections, then (i) and (ii) are equivalent to 
\begin{enumerate}
\item[(iii)] $\spn\{1_V: V\in \mathcal V\}=C_c^{\infty}(\Om)$.
\end{enumerate}
\end{lemma}

We see in particular, that the commutative $C^*$-algebra $C_0(\Om)$ is generated as a $C^*$-algebra by a (countable) set of projections.
The converse is also true: If $D$ is any commutative $C^*$-algebra which is generated by a set of projections $\{e_i: i\in I\}\subseteq D$
and if $\Om=\Spec(D)$ is the Gelfand spectrum of $D$, then $\Omega$ is totally disconnected and the sets 
$$\mathcal V=\{\supp(\hat{e}_i): i\in I\},$$
where, for any $d\in D$, $\hat{d}\in C_0(\Om)$ denotes the Gelfand transform of $d$, is a family of compact open subsets of $\Om$
which generates $\mathcal U_c(\Om)$.  For a proof see \cite{CEL2}*{Lemma 2.3}. Thus there is an equivalence between studying sets of 
projections which generate $D$ or sets of compact open subsets of $\Om$ which generate $\mathcal U_c(\Om)$.

\begin{lemma}\label{lem-supp}
Suppose that $\{e_i:i\in I\}$ is a set of projections in the commutative $C^*$-algebra $D$. Then 
for each finite subset $F\subseteq I$ there exists a smallest projection $e\in D$ such that 
$e_i\leq e$ for every $i\in F$. We then write $e:=\vee_{i\in F}e_i$.
\end{lemma}
\begin{proof} By the above discussion we may assume that $D=C_0(\Om)$ for some totally disconnected space $\Om$.
For each $i\in F$ let $V_i:=\supp(e_i)$. 
Then $e=1_V$ with $V=\cup_{i=1}^l V_i$.
\end{proof}

The independence condition given in the following definition is central for the results of this section:

\begin{definition}\label{def-independence}
Suppose that $\{X_i: i\in I\}$ is a family of subsets of a set $X$. We then say that $\{X_i:i\in I\}$ is {\em independent}
if for any finite subset $F\subseteq I$ and for any index $i_0\in I$ we have
$$X_{i_0}=\cup_{i\in F} X_i\Rightarrow i_0\in F.$$
Similarly,  a family $\{e_i:i\in I\}$ of projections in the  commutative $C^*$-algebra $D$ is called {\em independent}
if for any finite subset $F\subseteq I$ and every $i_0\in I$ we have
$$e_{i_0}=\vee_{i\in F} e_i\Rightarrow i_0\in F.$$
\end{definition}

Of course, if $D=C_0(\Om)$ then  $\{e_i: i\in I\}$ is an independent family of projections in $D$ if and only 
if $\{\supp(e_i): i\in I\}$ is an independent family of compact open subsets of $\Om$. The following lemma is 
\cite{CEL2}*{Lemma 2.8}. The proof follows from \cite{LiNuc}*{Proposition 2.4}:

\begin{lemma}\label{lem-independence}
Suppose that $\{e_i:i\in I\}$ is an independent set of projections in the commutative $C^*$-algebra $D$ which is closed under 
finite multiplication up to $0$. Then $\{e_i:i\in I\}$ is independent if and only if it is linearly independent.
\end{lemma}

\begin{definition}\label{def-regular-basis}
Suppose that $\Om$ is a totally disconnected locally compact Hausdorff space. 
An independent family $\mathcal V$ of non-empty compact open subsets of $\Om$ is called a {\em regular basis} (for the 
compact open subsets of $\Om$) if it generates $\mathcal U_c(\Om)$ and if $\mathcal V\cup\{\emptyset\}$ is closed under finite intersections.

A family of projections  $\{e_i:i\in I\}$  in the commutative $C^*$-algebra $D$ is called a {\em regular basis} for $D$, if it 
is (linearly) independent, closed under finite multiplication up to $0$ and generates $D$ as a $C^*$-algebra.
\end{definition}

The  following lemma is  a consequence of the above discussions. We leave the details to the reader.

\begin{lemma}\label{lem-basis} 
A family of projections $\{e_i:i\in I\}$ in the commutative $C^*$-algebra $D$ is a regular basis for $D$ if and only if 
the set $\V=\{\supp\hat{e}_i: i\in I\}$ is a regular basis for the compact open subsets of $\Om=\Spec(D)$.
Conversely, $\V$ is a regular basis for the compact open subsets of the locally compact space $\Om$
if and only if $\{1_V: V\in \V\}$ is a regular basis for $C_0(\Om)$.
\end{lemma}

It is not difficult to see that every totally disconnected locally compact space $\Om$ has a regular basis 
for its compact open sets. A formal proof is given in \cite{CEL2}*{Proposition 2.12}. The following example
shows the existence for the Cantor set:

\begin{example}\label{ex-base}
Let $\Om= \{0,1\}^{\ZZ}$ denote the direct product of copies of $\{0,1\}$
over $\ZZ$ equipped with the product topology. For each finite subset $F\subseteq \ZZ$  
let 
$$V_{F}=\{(\eps_n)_n\in \Om: \eps_n=0 \, \forall n\in F\}$$
be the corresponding {\em cylinder set} in $\Om$.
It is then an easy exercise to check that the collection 
$\V=\{V_{F}: F\subseteq \ZZ\;\text{finite}\}$
is a regular basis for the compact open subsets of $\Om$. 
\end{example}

Assume now that $G$ is a second countable locally compact group and 
$\Om$ is a second countable totally disconnected  $G$-space such that there exists 
a {\bf $G$-invariant} regular basis   $\V=\{V_i:i\in I\}$ for the compact open subsets 
of $\Om$. For all $i\in I$ let $e_i=\one_{V_i}$ be the characteristic function of $V_i$.
Then, since $\V$  is $G$-invariant, the action of $G$ on $\Om$ induces an action of $G$
on $I$. Consider the unitary representation $U:G\to \U(\ell^2(I)); (U_g\xi)(i)=\xi(g^{-1} i)$
and let $\Ad U: G\to \Aut(\K(\ell^2(I)))$ denote the corresponding adjoint action.
For each $i\in I$ let  $\delta_i$ denote the Dirac function at $i$ and let $d_i:\ell^2(I)\to\CC\delta_i$ denote the 
orthogonal projection. 
Then there exists a unique $G$-equivariant $*$-homomorphism
$$\Phi:C_0(I)\to C_0(\Om)\otimes \K(\ell^2(I))\quad\text{such that}\quad \Phi(\delta_i)=e_i\otimes d_i,$$
for all $i\in I$,  where the action of $G$ on $C_0(I)$ is induced by the action on $I$ and the action on 
$C_0(\Om)\otimes \K(\ell^2(I))$ is given by the diagonal action $\tau\otimes \Ad U$, where 
$\tau:G\to \Aut(C_0(\Om))$ denotes the given action of $G$ on $C_0(\Om)$.
More generally if $\alpha:G\to\Aut(A)$ is an action of $G$ on a $C^*$-algebra $A$, then
 there exists a $G$-equivariant $*$-homomorphism
$$\Phi_A: C_0(I)\otimes A\to C_0(\Om)\otimes A\otimes   \K(\ell^2(I))\quad\text{s.t.}\quad \;\Phi_A(\delta_i\otimes a)= e_i\otimes a\otimes d_i.$$
Note that the action $\tau\otimes\alpha\otimes\Ad U$ of $G$ on $C_0(\Om)\otimes A\otimes \K(\ell^2(I))$ is Morita equivalent to the 
action $\tau\otimes \alpha$ of $G$ on $C_0(\Om)\otimes A$ via the $G$-equivariant equivalence bimodule
$\E:=(C_0(\Om)\otimes A\otimes \ell^2(I), \tau\otimes\alpha\otimes U)$. Thus we obtain a $KK^G$-class 
$$x=[\Phi_A]\otimes_{C_0(\Om)\otimes A\otimes \K}\E\in KK^G(C_0(I,A), C_0(\Om,A)).$$

The following is the main result of this section:

\begin{theorem}[{cf. \cite{CEL2}}]\label{thm-Om}
Suppose that $\{e_i:i\in I\}$ is a $G$-equivariant regular basis for $C_0(\Om)$, $A$ is any $G$-$C^*$-algebra,  and  
$G$ satisfies the Baum-Connes conjecture for $C_0(I,A)$ and $C_0(\Om,A)$. Then the descent
$$J_G(x)\in \KK( C_0(I,A)\rtimes_rG, C_0(\Om,A)\rtimes_rG)$$
of the class $x\in KK^G(C_0(I,A), C_0(\Om,A))$ constructed above 
induces an isomorphism $K_*(C_0(I,A)\rtimes_rG)\cong K_*(C_0(\Om,A)\rtimes_rG)$. 

If, moreover, $G$ satsfies the  strong Baum-Connes conjecture and if  $A$ is type I, then 
$J_G(x)\in \KK( C_0(I,A)\rtimes_rG, C_0(\Om,A)\rtimes_rG)$ is a $KK$-equivalence.
\end{theorem}

The above theorem has originally been shown in \cite{CEL2}*{\S 3}, extending an earlier result given in 
\cite{CEL1}. Before we present some of the crucial ideas of the proof, we would like to discuss a bit  why
this result might be useful for explicit $K$-theory calculations. The main reason is due to the 
relatively easy structure of crossed products by groups acting on {\em discrete} spaces $I$.
If such action is given (as in the situation of our theorem) and if $A$ is any other $G$-$C^*$-algebra,
we obtain a $G$-equivariant direct sum decomposition 
$$C_0(I,A)\cong \oplus_{[i]\in G\backslash I} C_0(G\cdot i)\otimes A,$$
in which $G\cdot i=\{g\cdot i: g\in G\}$ denotes the $G$-orbit of the representative $i$ of the class $[i]\in G\backslash I$.
Let $G_i:=\{g\in G: g\cdot i=i\}$ denote the stabiliser of $i$ in $G$. Then $G_i$ is open in $G$ 
and we have a $G$-equivariant bijection
$$G/G_i\stackrel{\cong}{\to} G\cdot i; g G_i\mapsto g\cdot i.$$
Moreover, by Green's imprimitivity theorem (\cite{CroPro}*{Theorem 2.6.4}; see also  \cite{CroPro}*{Remark 2.6.9}),
there are natural Morita equivalences
$$C_0(G/G_i, A)\rtimes_rG\cong A\rtimes_r G_i.$$
Putting things together, we therefore get
$$C_0(I,A)\rtimes_rG\cong \oplus_{[i]\in G\backslash I} C_0(G/G_i,A)\rtimes_rG\sim_M \oplus_{[i]\in G\backslash I} A\rtimes_r G_i.$$
Since Morita equivalences are $KK$-equivalences, we get
$$K_*(C_0(I,A)\rtimes_rG)\cong \oplus_{[i]\in G\backslash I}K_*(A\rtimes_rG_i).$$
Thus

\begin{corollary}\label{cor-Om}
Suppose that $G,\Om, A$ and $\{e_i:i\in I\}$ are as in Theorem \ref{thm-Om}. Then there is an isomorphism
$$K_*( C_0(\Om,A)\rtimes_rG)\cong \oplus_{[i]\in G\backslash I}K_*(A\rtimes_rG_i).$$
In particular, if $A=\CC$, there is an isomorphism 
$$K_*( C_0(\Om)\rtimes_rG)\cong \oplus_{[i]\in G\backslash I}K_*(C_r^*(G_i)).$$
If $G$ satisfies  strong BC and $A$ is type I, the isomorphism is induced by a $KK$-equivalence
between $C_0(\Om,A)\rtimes_rG$ and $\oplus_{[i]\in G\backslash I} A\rtimes_rG_i$.
\end{corollary}

In many interesting examples coming from the theory of $C^*$-semigroup algebras and crossed products 
of semigroups by automorphic actions 
of  semigroups $P\subseteq G$, the stabilisers for the action of $G$ on $I$ have very easy structure, so that
the $K$-theory groups of the crossed products $A\rtimes_rG$ are computable. 
This is in particular true in the case  $A=\CC$.  The applications of Theorem \ref{thm-Om} to $C^*$-semigroup algebras 
is discussed in more detail in  \cite{SgpC}*{Section 5.10}) and  \cite{AlgAct}*{Section 6.5}.

\begin{example}\label{ex-lamp}
To illustrate the usefulness of our approach we consider the group algebra of the  lamplighter group $\ZZ/2 \wr \ZZ$
which is the semi-direct product $\big(\oplus_{\ZZ} \ZZ/2\big)\rtimes \ZZ$, where the action is given 
 via translation of the summation index. Since the dual group of $\oplus_\ZZ \ZZ/2$ is equal to the 
direct product $\Om:=\prod_{\ZZ} \{1,-1\}=\{1,-1\}^{\ZZ}$ the 
group algebra $C^*( \ZZ/2\wr\ZZ)$ is isomorphic 
to $C(\Om)\rtimes \ZZ$. Moreover, by Example \ref{ex-base} there exists a regular basis $\V=\{V_F: F\subseteq \ZZ\;\text{finite}\}$ 
for the compact open subsets of $\Om$
consisting of the cylinder sets $V_F=\{(\eps_n)_n\in \Om: \eps_n=1  \;\forall n\in F\}$ attached to the finite subsets $F\subseteq \ZZ$. 
This basis is clearly $\ZZ$-invariant, hence our theorem applies to the corresponding regular  basis $\{e_F=\one_{V_F}: F\subseteq \ZZ\;\text{finite}\}$
of $C(\Om)$. 
Let $\mathcal F=\{F\subseteq \ZZ: F\;\text{finite}\}$ denote the index set of this basis and let $\mathcal F^*=\mathcal F\smallsetminus \{\emptyset\}$. 
The action of $\ZZ$ on $\mathcal F$ fixes $\emptyset$ and acts freely on $\mathcal F^*$. Hence, our theorem
gives 
\begin{align*}
K_*(C^*( \ZZ/2\wr\ZZ))=K_*(C(\Om)\rtimes \ZZ)&\cong K_*(C_0(\mathcal F)\rtimes \ZZ)\\
&\cong   
K_*(C^*(\ZZ))\oplus\left( \oplus_{[F]\in \ZZ\backslash \mathcal F^*} K_*(\CC)\right).
\end{align*}
Since $C^*(\ZZ)\cong C(\TT)$ and $K_0(C(\TT))=\ZZ=K_1(C(\TT))$ we get
$$K_0(C^*( \ZZ/2\wr\ZZ)\cong \oplus_{[F]\in G\setminus \mathcal F}\ZZ\quad\text{and}\quad K_1(C^*( \ZZ/2\wr\ZZ))=\ZZ.$$
Of course, the result can easily be extended to more general wreath products $\ZZ/2\wr G=\oplus_{G} \ZZ/2\rtimes G$, 
where $G$ is a countable discrete group which satisfies appropriate versions of the Baum-Connes conjecture.
\end{example}

\begin{proof}[Proof of Theorem \ref{thm-Om}]
For the sake of presentation, let us assume that 
$G$ is countable discrete (which is the case in most applications). 
Since the (strong) Baum-Connes conjecture is invariant with respect to  $KK^G$-equivalent actions, and 
since $G$-equivariant Morita equivalences are $KK^G$-equivalences, it suffices to proof that 
the descent
$$\Phi_A\rtimes_rG: C_0(I,A)\rtimes_r G\to \big(C_0(\Om,A)\otimes \K(\ell^2(I))\big)\rtimes_rG$$
of the homomorphism $\Phi_A$ induces an isomorphism in $K$-theory (resp.~a $KK$-equivalence in 
case where $G$ satisfies  strong BC and $A$ is type I). To see that this is the case we want to exploit the Going-Down principle
of the previous section, i.e., we need to show that  for every finite subgroup $F\subseteq G$, the 
map 
\begin{equation}\label{eq-finite}
\Phi_A\rtimes F: C_0(I,A)\rtimes F\to \big(C_0(\Om,A)\otimes \K(\ell^2(I))\big)\rtimes F
\end{equation}
induces an isomorphism in $K$-theory. Note that if $A$ is type I, the same is true 
for $\big(C_0(\Om,A)\otimes \K(\ell^2(I))\big)\rtimes F$ by  \cite{CroPro}*{Corollary 2.8.21} and therefore 
it follows from the Universal-Coefficient-Theorem of $KK$-Theory (e.g., see \cite{Bla86}*{Chapter 23}) that 
$\Phi_A\rtimes F$ being an isomorphism already implies that it is a $KK$-equivalence.
Thus, in this situation, the stronger result that $\Phi_A\rtimes_rG$ is a $KK$-equivalence will 
then follow from the Meyer-Nest Theorem \ref{thm-MN}.

Using the Green-Julg theorem, the map $\Phi_A\rtimes F$ of (\ref{eq-finite}) being an isomorphism 
 is equivalent to 
$$(\Phi_A)_*: K_*^F( C_0(I,A))\to K_*^F\big(C_0(\Om,A)\otimes \K(\ell^2(I))\big)$$
being an isomorphism.

So in what follows let us fix 
a finite subgroup $F$ of $G$. Let $J\subseteq I$ be any finite $F$-invariant subset such that $\{e_i: i\in J\}$ 
is closed under multiplication (up to $0$). Then  $D_J:=\spn\{e_i:i\in J\}$ is a finite dimensional 
commutative $C^*$-subalgeba of $C_0(\Om)$ of dimension $\dim(D_J)=|J|$. Consider the map
\begin{equation}\label{eq-J}\Phi_J: C_0(J)\to D_J\otimes \K(\ell^2(J)); \; \Phi_J(\delta_i)=e_i\otimes d_i.
\end{equation}
We want to show that $\Phi_J$ is invertible in $KK^F\big(C_0(J), D_J\otimes \K(\ell^2(J))\big)\cong KK^F(C_0(J), D_J)$.
If this happens to be true, then $\Phi_{A,J}:=\Phi_J\otimes\id_A: C_0(J,A)\to D_J\otimes A\otimes \K(\ell^2(J))$
will be $KK^F$-invertible as well, and the desired
result then follows from the following commutative diagram
$$
\begin{CD}
K_*^F(C_0(J,A)  ) @> (\Phi_{A,J})_*>\cong > K_*^F\big( D_J\otimes A \otimes \K(\ell^2(J))\big)\\
@V\iota_* VV   @VV\iota_*V\\
\lim_{J} K_*^F(C_0(J,A)  ) @>\lim_J (\Phi_{A,J})_*>\cong > \lim_JK_*^F\big( D_J\otimes A \otimes \K(\ell^2(J))\big)\\
@V\cong VV      @VV\cong V\\
K_*^F(C_0(I,A)  ) @> (\Phi_{A})_*>\cong > K_*^F\big( C_0(\Om)\otimes A \otimes \K(\ell^2(I))\big)
\end{CD}
$$

The $KK^F$-invertibility of $\Phi_J$ in (\ref{eq-J}) will be a consequence of a
UCT-type result for finite dimensional $F$-algebras, which we are now going to explain:
Suppose that $C$ and $D$ are commutative finite dimensional $F$-algebras
with $\dim(C)=n, \dim(D)=m$ (in our application, $C$ will be $C_0(J)$, $D=D_J$, and $n=m=|J|$). 
Let $\{c_1,\ldots, c_n\}$ and $\{d_1,\ldots, d_m\}$ be choices of pairwise orthogonal 
projections, which then form a basis of $C$ and $D$, respectively.
Then we have isomorphisms $\ZZ^n\cong K_0(C)$ and $\ZZ^m\cong K_0(D)$
sending the $j$th unit vector $e_j$ to $[c_j]$ (resp.~$[d_j]$). 
If we ignore the $F$-action, the UCT-theorem for $KK$ implies that
\begin{equation}\label{eq-finiteKK}
KK_0(C,D)\cong \Hom(K_0(C), K_0(D))\cong M(m\times n,\ZZ),
\end{equation}
where the first isomorphism is given by sending $x\in KK(C,D)$ to the map $\cdot\otimes_Cx:K_0(C)\to K_0(D)$
and the second map is given via the above identifications of $K_0(C)\cong \ZZ^n$ and  $K_0(D)\cong \ZZ^m$.
Suppose now that $C$ and $D$ are $F$-algebras such that $F$ acts via permutations of the basis elements 
in  $\{c_1,\ldots, c_n\}$ and $\{d_1,\ldots, d_m\}$, respectively. Note that the actions of $F$ on $C$ and $D$
are determined by two homomorphisms $\tau: F\to S_n$ and $\sigma:F\to S_m$ such that $g\cdot c_i=c_{\tau_g(i)}$ and 
$g\cdot d_j=d_{\sigma_g(j)}$ for all $i,j$. 
Then the equivariant version of (\ref{eq-finiteKK}) does not give an isomorphism in general, but we get a homomorphism
\begin{equation}\label{eq-finiteKK1}
\Psi_{C,D}: KK_0^F(C,D)\to  \Hom_F(K_0(C), K_0(D))\cong M_F(m\times n,\ZZ):x\mapsto \Gamma_x
\end{equation}
where $\Hom_F(K_0(C), K_0(D))$ denotes the $F$-equivariant homomorphisms (with $F$ acting on 
the basis elements $[c_i]$ and $[d_i]$ of $K_0(C)$ and $K_0(D)$, repectively) and $M_F(m\times n, \ZZ)$ 
denotes the set of all $m\times n$-matrices  $\Gamma=(\gamma_{ij})$ over $\ZZ$ which satisfy 
\begin{equation}\label{eq-gamma}
\gamma_{ij}=\gamma_{\tau_g(i),\sigma_g(j)}\quad\forall g\in F.\
\end{equation}
We need  to construct a section  $M_F(m\times n,\ZZ)\to KK^F(C,D); \Gamma\mapsto x_{\Gamma}$ for $\Psi_{C,D}$ 
which is compatible 
with taking Kasparov-products. For this let $\Gamma=(\gamma_{ij})\in M_F(m\times n,\ZZ)$ be given. 
Let $ \E_{ij}=\CC^{|\gamma_{ij}|}\otimes \CC d_i$ viewed as a Hilbert $D$-module 
in the canonical way. 
Let 
$\varphi_{ij}: C\to \K(\E_{ij})$ be the $*$-homomorphism such that $\varphi_{ij}(c_j)=\one_{\E_{ij}}$ and $\varphi_{ij}(c_k)=0$
for all $k\neq i$.  Let 
$$\E^+_{\Gamma}=\oplus_{\gamma_{ij}>0}\E_{ij}\quad \text{and}\quad \varphi^+=\oplus_{\gamma_{ij}>0}\varphi_{ij}:C\to \K(\E^+),$$
and, similarly,
$$ \E^-_{\Gamma}=\oplus_{\gamma_{ij}<0} \E_{ij}\quad \text{and}\quad \varphi^-=\oplus_{\gamma_{ij}<0}\varphi_{ij}:C\to \K(\E^-).$$
Because of (\ref{eq-gamma}) there are canonical actions of $F$ on $\E^+,\E^-$ such that $g\cdot \E_{ij}=\E_{\tau_g(i)\sigma_g(j)}$ 
for all $i,j$ and such that $(\E^+_{\Gamma},\varphi^+)$ and $(\E^-_{\Gamma},\varphi^-)$ become $F$-equivariant $C-D$ correspondences.
Finally, let $\E_{\Gamma}=\E^+_{\Gamma}\oplus \E^-_{\Gamma}$ with $\ZZ/2$-grading given by the matrix $\bmtr 1&0\\0&-1\emtr$ 
and let $\varphi=\bmtr \varphi^+&0\\ 0&\varphi^-\emtr$. Since $\varphi$ takes value in $\K(\E)$, we get a class 
$$x_\Gamma:=[\E_{\Gamma},\varphi, 0]\in KK^F(C,D).$$
The proof of the following lemma is left as an exercise for the reader (or see \cite{CEL2}*{Lemma A.2}):

\begin{lemma}\label{lem-finiteUCT}
Suppose that $B,C,D$ are finite dimensional commutative $F$-algebras such that $\{b_1,\ldots, b_k\}$, $\{c_1,\ldots, c_n\}$, and 
$\{d_1,\ldots, d_m\}$ are $F$-invariant bases consisting of orthogonal projections in $B,C,D$, respectively.
Then, for all matrices $\Lambda\in M_F(n\times k,\ZZ)$ and $\Gamma\in M_F(m\times n,\ZZ)$ we get
$$x_\Lambda\otimes_C x_{\Gamma}=x_{\Gamma\cdot\Lambda}\in KK^F(B,D).$$
In particular, if $n=\dim(C)=\dim(D)$ and $\Gamma\in M_F(n\times n,\ZZ)$ is invertible over $\ZZ$, then 
$x_\Gamma\in KK^F(C,D)$ is invertible as well.
\end{lemma}

Let us come back to the class $\Phi_J: C_0(J)\to D_J\otimes \K(\ell^2(J))$ which sends $\delta_i$ to $e_i\otimes d_i$, where 
$d_i$ denotes the orthogonal projection onto $\CC\delta_i$. The following lemma gives the crucial point of how independence
of the family $\{e_i:i\in I\}$ of the basis elements of $C_0(\Om)$ enters the picture: 

\begin{lemma}[{cf \cite{CEL2}*{Lemma 3.8}}]\label{lem-finite-independent}
Let $D$ be a commutative $C^*$-algebra generated by a multplicatively closed (up to $0$) and independent finite set
of projections $\{e_i:i\in J\}$. For each  $i\in J$ let $e'_i:=e_i-\vee_{e_j<e_i}e_j$. Then 
$\{e'_i:i\in J\}$ is a family of non-zero pairwise orthogonal projections spanning $D$. 
Moreover, the transition matrix $\Gamma=(\gamma_{ij})$ determined by the equation 
$$e_j=\sum_{i\in J} \gamma_{ij} e'_i$$ is unipotent and therefore invertible over $\ZZ$. Its entries are either $0$ or $1$.
\end{lemma}
\begin{proof} Independence implies that $e'_i\neq 0$ for all $i\in J$. If $i\neq j$, then we have either $e_i<e_j$, in which case 
$e'_je_i=e_je_i-\vee_{e_k<e_j}e_ke_i=0$ or $e_ie_j<e_i$, in which case $e'_ie_j=e_ie_j-\vee_{e_k<e_i}e_ke_j=0$.
Either case implies $e'_ie'_j=0$. Since $\dim(D)=|J|$, it follows that $D=\spn\{e'_i:i\in I\}$.

If $e'_i\leq e_j$, then $e'_i\leq e_ie_j\leq e_i$ by definition of $e'_i$. This shows that $\gamma_{ij}=1$ if 
$e_i\leq e_j$ and $\gamma_{ij}=0$ otherwise. Thus, if we choose an ordering $\{i_1,\ldots, i_n\}$ of $J$ such that $e_i\leq e_j\Rightarrow i\leq j$,
the matrix $\Gamma$ is upper triangular with $1$s on the diagonal, hence unipotent. Thus, $\one-\Gamma$ is nilpotent 
of order $n=|J|$ and $\Gamma$ is invertible with inverse $\Gamma^{-1}=\sum_{k=0}^n(1-\Gamma)^k$.
\end{proof}

To finish the proof of Theorem \ref{thm-Om} we observe that the class $[\Phi_J]\in KK^F(C_0(J), D_J\otimes \K(\ell^2(J)))\cong KK^F(C_0(J), D_J)$
coincides with the class $x_\Gamma=[\E_\Gamma, \varphi, 0]$ as constructed in the above lemma 
with respect to the basis $\{\delta_i: i\in J\}$ of $C_0(J)$ and the basis $\{e'_i:i\in J\}$ of $D_J$. 
A combination of Lemma \ref{lem-finite-independent} with Lemma \ref{lem-finiteUCT} then 
implies that $[\Phi_J]$ is invertible. 
Indeed, since $\gamma_{ij}=0$ or $1$, it follows that  $\E_\Gamma=\E_{\Gamma}^+=\oplus_{i\in J} \big(\oplus_{j\in J, \gamma_{ij}=1} \CC e'_i\big)$ 
embeds as a direct summand into $D_J\otimes \ell^2(J)$ such that $\Phi_J: C_0(J)\to D_J\otimes\K(\ell^2(J))\cong \K(D_J\otimes \ell^2(J))$
decomposes as $\varphi\oplus 0$.
\end{proof}

As remarked before,  the main applications for Theorem \ref{thm-Om} are given in case of computing the $K$-theory 
of  reduced semigroup algebras $C_{\lambda}^*(P)$, where $e\in P\subseteq G$ is a sub-semigroup of the countable group $G$.
In case where $P\subseteq G$ satisfies a certain Toeplitz condition (which is discussed in detail in \cite{SgpC}), there 
exists a totally disconnected $G$-space $\Om_{P\subseteq G}$ such that $C_{\lambda}^*(P)$ can be realised as a full corner in 
the crossed product $C_0(\Om_{P\subseteq G})\rtimes_rG$, hence $K_*(C_{\lambda}^*(P))\cong K_*(C_0(\Om_{P\subseteq G})\rtimes_rG)$.
Now, the existence of a $G$-invariant regular basis for $C_0(\Om_{P\subseteq G})$ will follow from a certain 
independence condition for the inclusion $P\subseteq G$, which, somewhat surprisingly, is satisfied in a large number of 
interesting cases. Again, we refer to the  \cite{SgpC} and \cite{AlgAct} for more details on this.

Unfortunately,  a $G$-invariant regular basis 
$\{e_i:i\in I\}$ for $C_0(\Om)$, as required for the proof of Theorem \ref{thm-Om}, does not exist
in general. In fact, we have the following result, which excludes a large number of interesting cases
from our theory:

\begin{proposition}[{\cite{CEL2}*{Proposition 3.18}}]\label{prop-nobasis}
Let $G$ be a countable discrete group which acts minimally on the totally disconnected locally compact space
$\Om$, i.e. for every non-empty open subset $U$ of $\Om$, we have  
$$\Om=\cup_{g\in G} gU.$$
Suppose further that there exists a non-zero $G$-invariant Borel measure $\mu$ on  $\Om$ (which holds if $G$ is amenable and  $\Om$ is compact). 
Then $\Om$  has a $G$-invariant regular basis for the compact open sets if and only if  $\Om$ is discrete.
\end{proposition}

\bibliography{references}

\end{document}